\documentclass[11pt]{article}
\usepackage[T1]{fontenc}
\usepackage[utf8]{inputenc}
\usepackage{amsfonts}
\usepackage{bm}
\usepackage{amssymb}
\usepackage{latexsym}
\usepackage{amsmath}
\usepackage{amscd}
\usepackage{amsthm}
\usepackage{pifont}
\usepackage{fancyhdr}
\usepackage{epsfig}
\usepackage{alltt}
\usepackage{multimedia}
\usepackage[T1]{fontenc}
\usepackage{psfrag}
\usepackage{url}
\usepackage{graphicx}
\usepackage{ulem}
\usepackage{xcolor}
\usepackage{pstricks}

\newtheorem{theorem}{Theorem}[section]
\newtheorem{definition}{Definition}

\newtheorem{proposition}[theorem]{Proposition}

\newtheorem{lemma}[theorem]{Lemma}
\newtheorem{remark}[theorem]{Remark}
\newtheorem{corollary}[theorem]{Corollary}

\newtheorem{prop}[theorem]{Proposition}
\newtheorem{corol}[theorem]{Corollary}

\newcommand{\NN}{{\mathcal{N}}}
\newcommand{\FF}{{\mathcal{F}}}
\newcommand{\WW}{{\mathcal{W}}}

\newcommand{\eps}{\varepsilon}

\newcommand{\E}{{\mathbb E}}
\newcommand{\N}{{\mathbb N}}
\newcommand{\Z}{{\mathbb Z}}
\renewcommand{\P}{{\mathbb P}}
\newcommand{\Q}{{\mathbb Q}}
\newcommand{\R}{{\mathbb R}}
\newcommand{\ind}{{\bf 1}}
\newcommand{\Card}{{\rm Card}}

\newcommand{\0}{{\mathbf 0}}
\newcommand{\x}{{\mathbf x}}
\newcommand{\y}{{\mathbf y}}
\renewcommand{\u}{{\mathbf u}}
\renewcommand{\v}{{\mathbf v}}
\newcommand{\w}{{\mathbf w}}
\newcommand{\z}{{\mathbf z}}
\def\i{\mathbf{i}}

\newcommand{\eqdist}{\stackrel{d}{=}}

\title{The 2d-directed spanning forest converges to the Brownian web}

\author{David Coupier\thanks{UPHF, CNRS, EA 4015 - LAMAV, F-59313 Valenciennes Cedex 9, France; david.coupier@uphf.fr}, Kumarjit Saha\thanks{Ashoka University, Sonipat, India; kumarjit.saha@ashoka.edu.in}, Anish Sarkar\thanks{Indian Statistical Institute, Theoretical Statistics and Mathematics Unit, New Delhi 110016, India; anish@isid.ac.in}, Viet Chi Tran\thanks{LAMA, Univ Gustave Eiffel, UPEM, Univ Paris Est Creteil, CNRS, F-77447, Marne-la-Vall\'ee, France; chi.tran@u-pem.fr}}

\begin{document}
\maketitle
\begin{abstract}
The two-dimensional \textit{directed spanning forest} (DSF) introduced by Baccelli and Bordenave is a planar directed forest whose vertex set is given by a homogeneous Poisson point process ${\cal N}$ on $\R^2$. If the DSF has direction $-e_y$, the ancestor $h(\u)$ of a vertex $\u \in {\cal N}$ is the nearest Poisson point (in the $L_2$ distance) having strictly larger $y$-coordinate. {This construction induces complex geometrical dependencies.} In this paper we show that the collection of DSF paths, properly scaled, converges in distribution to the Brownian web (BW). See Theorem \ref{th:FF_BW}. This verifies a conjecture made by Baccelli and Bordenave in 2007 \cite{BB07}.

\end{abstract}



\noindent \textbf{Acknowledgements:} We thank the anonymous Reviewer for very interesting comments that contribute to the improvement of the paper. This work was supported by the GdR GeoSto 3477 and Labex CEMPI (ANR-11-LABX-0007-01). D.C. is funded by ANR PPPP (ANR-16-CE40-0016), and K.S. is funded by the AIRBUS Group Corporate Foundation Chair in Mathematics of Complex Systems established in TIFR CAM. This work was  partially done while K.S., A.S. and V.C.T. were visiting the Institute for Mathematical Sciences, National University of Singapore in 2017. The visit was supported by the Institute.

\section{Introduction and results}
\subsection{The DSF and its conjectured scaling limit}

Let us consider a homogeneous Poisson point process (PPP) $\NN$ with intensity $\lambda>0$ {on the plane $\R^2$, equipped with the Euclidean distance $\|\x\|^2_2=\x^2(1)+\x^2(2)$, where we denote by} $\x(i)$, for $i=1,2$, the $i$-th coordinate of $\x\in\R^2$. In this work, horizontal and vertical axes will be respectively interpreted as space and time axes. Let us also denote by $\mathbb{H}^+(l):=\{\x\in\R^2 : \x(2)\geq l\}$ the half plane of points with ordinates greater than $l\in\R$. The \textit{ancestor} of $\x\in\NN$ is defined as the closest Poisson point to $\x$ {in the open half plane $\{\y \in \R^2 : \y(2) > \x(2)\}$}:
\begin{equation}
\label{def:h-stepL2}
h(\x,\NN) := \text{argmin}\{\|\y-\x\|_2 : \, \y \in \NN , \, \y(2) > \x(2)\} ~.
\end{equation}
In most occasions, we drop the second argument for $h(\x,\NN)$ and merely denote it by $h(\x)$. It is useful to observe that for all $\x\in\R^2$, the point $h(\x)$ is well defined. The \textit{Directed Spanning Forest} (DSF) with direction $-e_y$ on $\R^2$ is the random geometric graph $\mathfrak{F}$ with vertex set $\NN$ and edge set $E:=\{(\x, h(\x)) : \x\in\NN\}$. Since for any $\x\in\NN$, the point $h(\x)$ a.s. denotes a unique Poisson point, the DSF is a directed outdegree-one graph without cycle. This justifies it is called a forest.

The DSF was introduced in 2007 by Baccelli and Bordenave \cite{BB07} as a tool to study the asymptotic properties of the \textit{Radial Spanning Tree} (RST) which actually was the main subject of study in \cite{BB07}. The RST is a tree rooted at the origin $O$ of $\R^{2}$, with vertex set $\NN\cup\{O\}$, in which each $\x\in\NN$ is connected to the closest Poisson point inside the open ball $\{\y\in\R^2 : \, \|\y\|_2<\|\x\|_2\}$. The authors showed that the DSF is an approximation of the RST, in distribution, locally and far from the origin.

However, the DSF appears as truly interesting in itself since it admits beautiful conjectures, already mentioned in \cite{BB07}. A trajectory of the DSF is a sequence $(\x,h(\x),h(h(\x)),\ldots)$ of successive ancestors. First, is it true that any two given trajectories of the DSF  eventually coalesce with probability $1$? In other words, is the DSF a tree? This question was solved in \cite{CT13} by Coupier and Tran using an efficient percolation technique, namely the Burton and Keane argument \cite{burtonkeane}. Besides,  Baccelli and Bordenave showed that under diffusive scaling, any trajectory of the DSF converges in distribution to a Brownian motion. Then they conjectured a stronger result {\cite[Section 7.3]{BB07}}: the convergence under this diffusive scaling, of the whole forest $\mathfrak{F}$ to the so-called Brownian web (BW).

In this paper we prove this second and stronger conjecture. In fact, we prove a slightly stronger result in the sense that we construct a dual forest and show that under diffusive scaling, the DSF and its dual jointly converge in distribution to the BW and its dual.\\

A natural strategy to answer these questions would be to exhibit some independence (or Markov) properties in time (i.e. w.r.t. the vertical axis) for any couple of trajectories of the DSF. But this strategy runs up against strong dependencies, due to the construction rule of the DSF $\mathfrak{F}$, which are of two types: between different trajectories on the one hand and within a single trajectory on the other hand. See Figure \ref{fig:notmarkov} for an illustration of these two dependence phenomena. Let us denote by $B(\x,r)$ the closed Euclidean ball with radius $r$. The construction of the ancestor $h(\x)$ of $\x$ implies that the interior of the semi-ball $B^+(\x,\|\x-h(\x)\|_2) := B(\x,\|\x-h(\x)\|_2)\cap\mathbb{H}^+(\x(2))$ is empty of Poisson points. Since this semi-ball overlaps the half-plane $\mathbb{H}^+(h(\x)(2))$, we have information coming from the past steps: the ancestor of $h(\x)$ cannot belong to the resulting intersection. Roughly speaking, the past of a DSF trajectory may influence its future. Furthermore, when the successive ancestors of $\x$ are constructed, the resulting empty region, called the \textit{history set}, may have a complicated shape: it is a union of semi-balls centered at already visited vertices intersected with a proper half plane (we shall be more precise in the sequel). This random region is not necessarily connected and cannot be \textit{a priori} bounded.

\begin{figure}[!ht]
\begin{center}
\begin{tabular}{cc}
\psfrag{z}{\small{$\z$}}
\psfrag{y}{\small{$\y$}}
\psfrag{x}{\small{$\x$}}
\includegraphics[width=4.8cm,height=2.5cm]{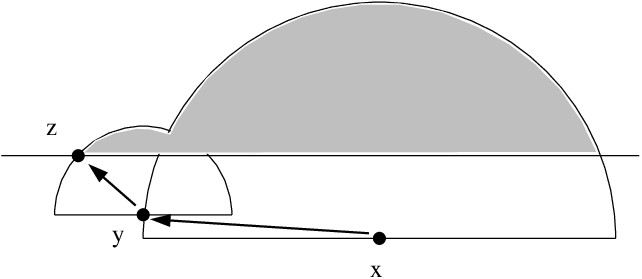}
\hspace{1cm}
&
\psfrag{a}{\small{$\x$}}
\psfrag{b}{\small{$\y$}}
\includegraphics[width=6.3cm,height=4.5cm]{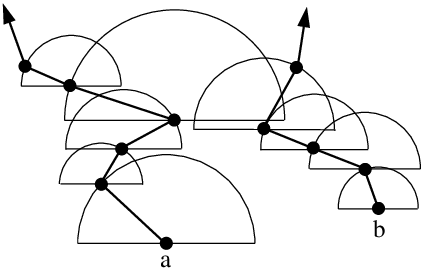}
\\
(a) & (b)
\end{tabular}
\caption{{\small \textit{(a) This picture illustrates the dependence phenomenon within a single trajectory and how the past trajectory may influence its next steps. It represents a Poisson point $\x$ and its first two ancestors, i.e. $\y=h(\x)$ and $\z=h(h(\x))$, and the two resulting semi-balls. The grey area corresponds to the history set of this trajectory which can not have Poisson points in its interior. It is worth pointing out here that the (large) empty semi-ball $B^+(\x,\|\x-h(\x)\|_2)$ may influence the construction of many ancestors of the initial vertex $\x$. (b) This second picture illustrates the dependence phenomenon between two DSF trajectories when the resulting semi-balls corresponding to their constructions overlap. This overlapping locally acts as a repulsive effect between trajectories starting at $\x$ and $\y$.}}}\label{fig:notmarkov}
\end{center}
\end{figure}

In \cite{FINR04}, Fontes et al introduced a suitable Polish space to study the BW, characterized its distribution (in Theorem \ref{theorem:Bwebcharacterisation} below) and provided criteria ensuring weak convergence to the BW (see Theorem \ref{thm:BwebConvergenceNoncrossing1} in Section \ref{sect:NewCVCriteria}). Since then, convergence to the BW for various directed forests or navigation schemes have been extensively studied and thence, the BW appeared as the universal scaling limit for a large number of seemingly unrelated models. Let us cite: \cite{berestyckigarbansen,NRS05} in the context of coalescing system of independent nonsimple random walks; \cite{CFD09,FFW05} in the context of drainage networks; \cite{SS13} for an oriented percolation model; \cite{NT12} in connection with Hastings-Levitov planar aggregation models; and \cite{berestyckigarbansen,colettidiasfontes,colettivalle,coupiermarckerttran} in the context of radial systems of coalescing trajectories. In many of these papers, the choice of the ancestor of any vertex $\x$ does not depend on the past, i.e., on what happens below ordinate $\x(2)$, allowing to easily introduce Markov processes and use martingale convergence theorems or Lyapunov functions. As explained above, this is no longer true for the DSF  because of complex geometrical dependencies. Recently, several papers \cite{RSS15,vallezuaznabar}-- Saha and Sarkar are involved in the first one --have considered modifications of the DSF in order to make the problem more tractable but until this paper, the conjecture of Baccelli and Bordenave remained open.

\subsection{The Brownian web}

The BW  appeared for the first time in the literature in the seminal papers of Arratia \cite{A79,arratia-unpub}. In \cite{A79}, the author studied the diffusive scaling limit of coalescing simple symmetric random walks starting from every point of $2\Z$ at time $0$ and showed that this collection converges to a collection of coalescing Brownian motions starting from every point on $\R$ at time $0$. {In \cite{arratia-unpub}, Arratia generalizes this by proposing a construction with paths starting from space-time points instead of just starting at time $0$.} For a general review on the BW see \cite{SSS17} and references therein. Later T\'{o}th and Werner \cite{TW98} gave a construction of a system of coalescing Brownian motions starting from every point in space-time plane $\R^2$ and used it to construct the true self-repelling motion.

The framework (topologies, spaces, characterization and convergence criteria) that we will use in this paper have been provided by Fontes et al. in \cite{FINR04}. Let us recall some relevant details. Let $\R^{2}_c$ be the completion of the space time plane $\R^2$ with respect to the metric
\begin{equation*}
\rho((x_1,t_1),(x_2,t_2)) := |\tanh(t_1)-\tanh(t_2)|\vee \Bigl| \frac{\tanh(x_1)}{1+|t_1|} - \frac{\tanh(x_2)}{1+|t_2|} \Bigr| ~.
\end{equation*}
As a topological space, $\R^{2}_c$ can be identified with the continuous image of $[-\infty,\infty]^2$ under a map that identifies the line $[-\infty,\infty]\times\{\infty\}$ with the point $(\ast,\infty)$, and the line $[-\infty,\infty]\times\{-\infty\}$ with the point $(\ast,-\infty)$. We define a path $\pi$ with starting time $\sigma_{\pi}\in [-\infty,\infty]$ as a continuous mapping $\pi :[\sigma_{\pi},\infty]\rightarrow [-\infty,\infty] \cup \{ \ast \}$ such that $\pi(\infty)= \ast$ and, when $\sigma_\pi = -\infty$, $\pi(-\infty)= \ast$. Notice that the mapping $t \mapsto (\pi(t),t) \in (\R^2_c,\rho)$ is continuous on $[\sigma_{\pi},\infty]$. We then define $\Pi$ to be the space of all paths in $\R^{2}_c$ with all possible starting times in $[-\infty,\infty]$.
The following metric, for $\pi_1,\pi_2\in \Pi$
\begin{align*}
d_{\Pi} (\pi_1,\pi_2) := & |\tanh(\sigma_{\pi_1})-\tanh(\sigma_{\pi_2})|\vee \\
& \sup_{t\geq \sigma_{\pi_1}\wedge \sigma_{\pi_2}} \Bigl|\frac{\tanh(\pi_1(t\vee\sigma_{\pi_1}))}{1+|t|}-\frac{\tanh(\pi_2(t\vee\sigma_{\pi_2}))}{1+|t|}\Bigr|
\end{align*}
makes $\Pi$ a complete, separable metric space. The metric $d_{\Pi}$ is slightly different from the original choice in \cite{FINR04} which is somewhat less natural as explained in \cite{SS08}. Convergence according to this metric can be described as locally uniform convergence of paths as well as convergence of starting times. Let ${\mathcal H}$ be the space of compact subsets of $(\Pi,d_{\Pi})$ equipped with the Hausdorff metric $d_{{\mathcal H}}$ given by,
\begin{equation*}
d_{{\mathcal H}}(K_1,K_2) := \sup_{\pi_1 \in K_1} \inf_{\pi_2 \in K_2} d_{ \Pi} (\pi_1,\pi_2) \vee \sup_{\pi_2 \in K_2} \inf_{\pi_1 \in K_1} d_{\Pi} (\pi_1,\pi_2) ~.
\end{equation*}
The couple $({\mathcal H},d_{{\mathcal H}})$ is a complete separable metric space. Let also $B_{{\mathcal H}}$ be the Borel $\sigma$-algebra on the metric space $({\mathcal H},d_{{\mathcal H}})$. The Brownian web ${\mathcal W}$ is then defined and characterized by the following result:

\begin{theorem}[Theorem 2.1 of \cite{FINR04}]
\label{theorem:Bwebcharacterisation}
There exists an $({\mathcal H}, {\mathcal B}_{{\mathcal H}})$-valued random variable
${\mathcal W}$ whose distribution is uniquely determined by
the following properties:
\begin{itemize}
\item[$(a)$] from any deterministic point $\x\in\R^2$, there is  almost surely a unique path $\pi^{\x}\in {\mathcal W}$  starting from $\x$;
\item[$(b)$] for a finite set of deterministic points $\x^1,\dotsc, \x^k \in \R^2$, the collection $(\pi^{\x^1},\dotsc,\pi^{\x^k})$ is distributed as coalescing Brownian motions starting from $\x^1,\dotsc,\x^k$;
\item[$(c)$] for any countable deterministic dense set ${\mathcal D}$ of $\R^2$, ${\mathcal W}$ is the closure of $\{\pi^{\x}: \x\in {\mathcal D} \}$ in $(\Pi, d_{\Pi})$  almost surely.
\end{itemize}
\end{theorem}

The above theorem shows that the collection is almost surely determined by countably many coalescing Brownian motions.

\subsection{Our convergence theorem and the key ideas of the proof}

Let us return to the DSF. To state our result formally we need to introduce some more notation. From a vertex $\u \in \NN$, define $h^0(\u) := \u$ and $h^{k}(\u) := h(h^{k-1}(\u))$, for $k\geq 1$. Taking the edges $\{ (h^{k-1}(\u),h^{k}(\u)) : k \geq 1\}$ to be straight line segments, we parameterize the path started from $\u$ and formed by these edges as the piecewise linear function $\pi^{\u} : [\u (2), \infty) \to \mathbb R$ such that $\pi^{\u}(h^{k}(\u)(2)) := h^{k}(\u)(1)$ for every $k \geq 0$ and
$\pi^{\u}(t)$ is linear in the interval $[h^{k}(\u)(2),h^{k+1}(\u)(2)]$. The collection of all DSF paths is denoted by ${\cal X} := \{\pi^{\u}:\u \in \NN\}$.

For given real numbers $\gamma,\sigma>0$, integer $n\geq 1$ and for a path $\pi$ with starting time $\sigma_{\pi}$, the {diffusively} scaled path $\pi_n(\gamma,\sigma) : [\sigma_{\pi}/n^2\gamma, \infty] \to [-\infty, \infty]$ is given by
\begin{equation}
\label{defi:ScaledPath}
\pi_n(\gamma,\sigma)(t) := \frac{\pi(n^2\gamma t)}{n \sigma} ~.
\end{equation}
Hence, the scaled path $\pi_n(\gamma,\sigma)$ has the starting time $\sigma_{\pi_n(\gamma,\sigma)} = \sigma_{\pi}/n^2\gamma$. For each $n\geq 1$, let ${\cal X}_n(\gamma,\sigma) := \{\pi_n^{\u}(\gamma,\sigma):\u \in \NN\}$ be the collection of all the scaled paths. The closure $\overline{\cal X}_n(\gamma,\sigma)$ of ${\cal X}_n(\gamma,\sigma)$ in $(\Pi,d_{\Pi})$ is a $({\cal H},{\cal B}_{{\cal H}})$-valued random variable which a.s. consists of non-crossing paths only. This property will be used in the sequel frequently.

Recall that $\lambda > 0$ is the intensity of the homogeneous PPP $\mathcal{N}$. Our main result, illustrated by Figure \ref{fig1}, solves the conjecture of Baccelli and Bordenave {\cite[Section 7.3]{BB07}}:

\begin{theorem}
\label{th:FF_BW}
{There exist $\sigma=\sigma(\lambda)>0$ and $\gamma=\gamma(\lambda)>0$ such that the sequence $
\big\{  \overline{{\cal X}}_n(\gamma,\sigma) : \, n \geq 1 \big\}$
converges in distribution to ${\cal W}$ as $({\cal H},{\cal B}_{{\cal H}})$-valued random variables as $n\rightarrow\infty$.}
\end{theorem}

\begin{figure}[!ht]
\begin{center}
\begin{tabular}{cc}
\includegraphics[width=4.5cm,height=4.5cm]{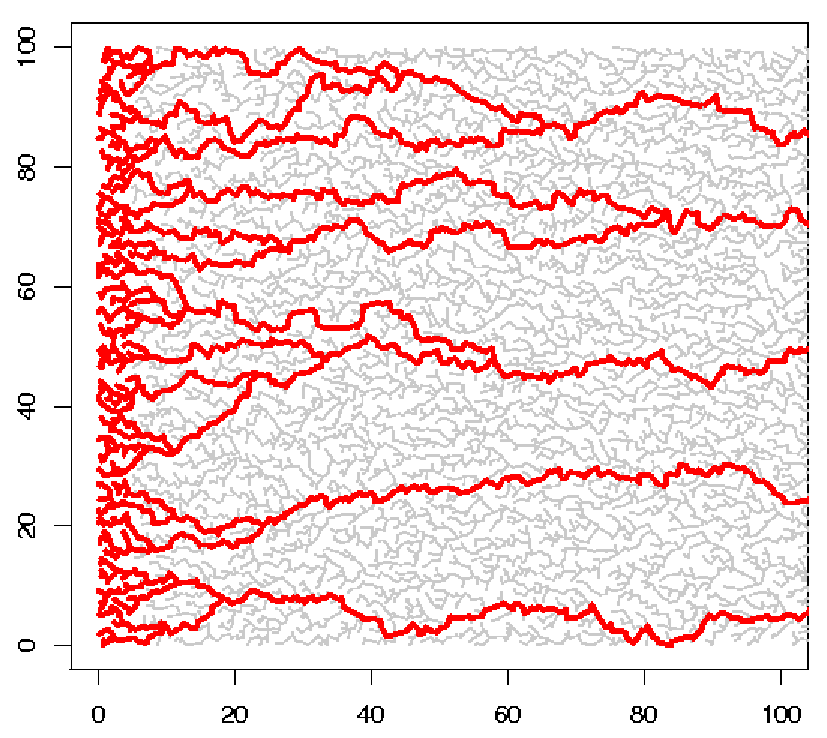} & \includegraphics[width=7.5cm,height=4.5cm]{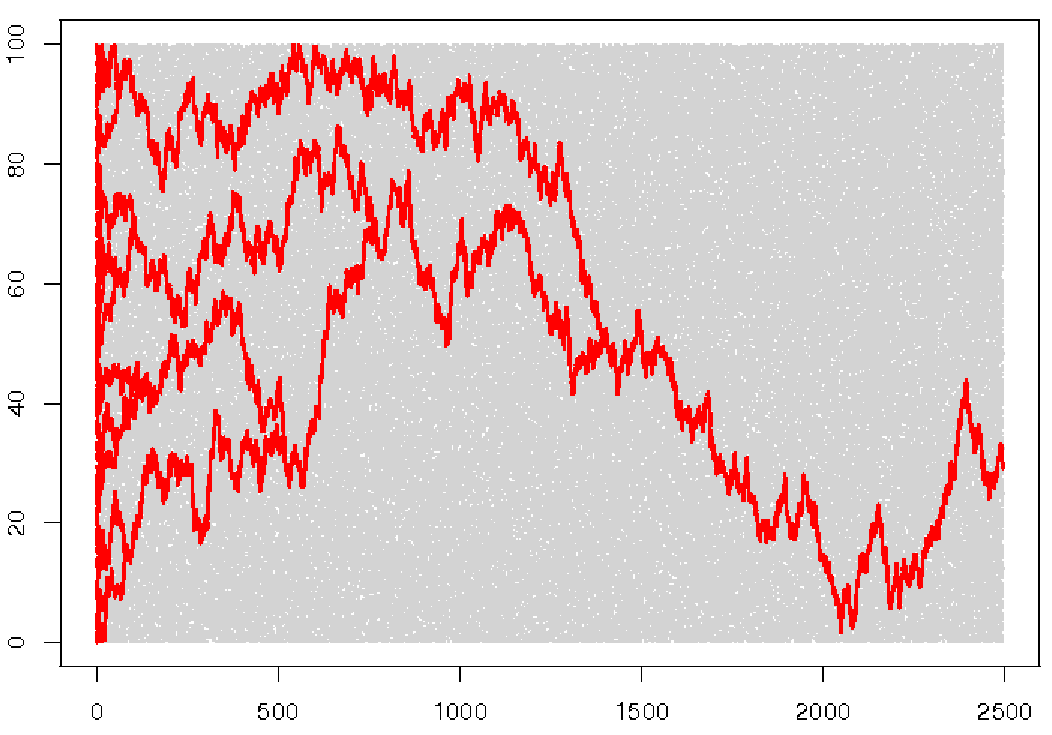} \\
(a) & (b)\\
\end{tabular}
\caption{{\small \textit{Simulations of the Directed Spanning Forest with direction $-e_{x}$ (this direction is chosen for the convenience of the graphical representations). The trajectories coming from vertices with abscissa $0\leq x\leq 5$ and ordinates $0\leq y\leq 100$ are represented in bold red lines. These simulations are taken from \cite{CT13}. On (b), the red paths clearly look like coalescing Brownian motions and they all coalesce before time $1500$.}}}\label{fig1}
\end{center}
\end{figure}

{In Section \ref{sec:cvBW} (Theorem \ref{th:FF_BW_complete}) we prove a stronger version of Theorem \ref{th:FF_BW} by showing that the DSF and its dual forest (which is defined later) jointly converge to the Brownian web and its dual process.}

Our proof actually appears as the combination of three main arguments or ideas described below. First, the criteria ensuring (weak) convergence to the BW have been meaningfully relaxed since the original convergence result in \cite{FINR04}, recalled here in Theorem \ref{thm:BwebConvergenceNoncrossing1} (Section \ref{sect:NewCVCriteria}). Indeed, in the literature \cite{colettidiasfontes,FFW05}, the proofs of criterion $(B_2)$ systematically require that the considered forest satisfies some FKG inequality (on its trajectories). But, this strong property becomes difficult to check, or even false, when dependence phenomena arise as it is the case for the DSF. Recently, in the context of non-crossing path models, Schertzer et al. \cite[Theorem 6.6]{SSS17} have replaced criterion $(B_2)$ with a \textit{wedge condition} involving a suitable dual of the considered forest. In this paper, we provide a new criteria (Theorem \ref{thm:JtConvBwebGenPath}), similar in the spirit to \cite[Theorem 6.6]{SSS17}, in which criterion $(B_2)$ is replaced with the fact that ``no limiting primal and dual paths can spend positive Lebesgue time together''. This is condition $(iv)$ of Theorem \ref{thm:JtConvBwebGenPath}.

The second key tool is a new and general Laplace type argument, stated in Theorem \ref{thm:LaplaceCoalescingTimeTailNew}, allowing to establish a coalescence time estimate for any couple of trajectories of the DSF (Theorem \ref{thm:CoalescingTimetail}). Obtaining such coalescence time estimate is always a crucial step in the literature to prove convergence to the BW. We also think that Theorem \ref{thm:LaplaceCoalescingTimeTailNew} is interesting in itself and very robust. In particular, it should provide the required coalescence time estimates for all the drainage network models in the basin of attraction of the BW \cite{colettidiasfontes,colettivalle,FFW05,RSS15,vallezuaznabar}. See Remark \ref{rmk:GeneralLaplace} for further details. The coalescence time estimate for the DSF (Theorem \ref{thm:CoalescingTimetail}) plays a central role in the proof of condition $(iv)$ previously cited.

The third main ingredient is a very accurate study, conducted in Section \ref{sect:Renewal}, of the joint evolution of DSF trajectories. Exploiting the geometric properties of the DSF, we are able to exhibit some \textit{renewal events} (at some random times) for the joint evolution of multiple trajectories. In case of evolution of a single trajectory, these renewal events give some suitable configurations allowing us to recover Markovian structure (see Proposition \ref{prop:SinglePtRWalk}).
For joint evolution of two paths, we show that the distance between the paths observed at these random times behaves like a random walk when the paths are sufficiently far apart.
Moreover, we show that both time and width of the explored region (by the trajectories) between two consecutive renewal events admit tail distributions with sub-exponential decays. All these properties allow us to show that the distance process satisfies the  conditions of the Laplace argument, more precisely conditions of Theorem \ref{thm:LaplaceCoalescingTimeTailNew2}.



\subsection{Application to the RST: the highways and byways problem}

In \cite[Theorem 2.1]{BB07}, Baccelli and Bordenave also described the semi-infinite paths of the Radial Spanning Tree (RST). In particular, they showed that the (random) number $\chi_r$ of semi-infinite paths of the RST crossing the circle $\mathcal{C}_r$-- centered at the origin $O$ and with radius $r$ --tends to infinity with probability $1$ as $r\to\infty$. A natural question is then to specify the growth rate of $\chi_r$ w.r.t. the radius $r$. Since the article of Hammersley and Welsh \cite{HaWe}, this question is known as the \textit{highways and byways problem}.

A general method, recently proposed by Coupier \cite{C2017} and applied to various geometrical random trees, asserts that $\chi_r$ is negligible w.r.t. $r$. Such result for the RST was already known since \cite{baccellicoupiertran}. Furthermore, this method can be performed whenever the considered tree satisfies the two following conditions (see Section 6 of \cite{C2017}). First, it can be approximated, locally and far from the origin, by a directed forest-- as the DSF approximates the RST. Secondly, the approximating directed forest has to satisfy a suitable coalescence time estimate. Theorem \ref{thm:CoalescingTimetail} fulfills this last condition for the DSF. Hence, the method developed in Section 6 of \cite{C2017} applies without major modifications to the RST and leads to the following result:

\begin{theorem}
\label{theo:sublinRST}
{For any $\epsilon>0$, $r^{-(3/4+\epsilon)} \chi_{r}$ tends to $0$, almost surely and in expectation, as $r$ tends to infinity.}
\end{theorem}

\subsection{Organization of the paper}

In Section \ref{sect:JointProcess}, a discrete process called the \textit{joint exploration process} is introduced to describe the joint evolution of DSF paths. The dependence structure of this process is encoded with the notion of \textit{history set}.
In Section \ref{sect:GoodStep} we able to obtain good control over the evolution of history sets.
Some particular random times, called \textit{renewal steps} and corresponding to the renewal
events mentioned above, are put forward in Section \ref{sect:Renewal}.
In Section \ref{sec:HittingTimeEst}, we present a general technique to study the coalescence time tail decay based on a Laplace criterion.  The coalescence time estimate (Theorem \ref{thm:CoalescingTimetail}) is stated and proved by applying this criterion in Section \ref{sect:CoalTimeMain}. In Section \ref{sect:NewCVCriteria}, we describe new criteria (Theorem \ref{thm:JtConvBwebGenPath}) ensuring the weak convergence of a forest and a suitable dual to the BW and its dual.

Several qualitative results of this paper involve constants. For the sake of clarity, we will use $C_0$ and $C_1$ to denote two positive constants, whose exact values may change from one line to the other. The important thing is that both $C_0$ and $C_1$ are universal constants whose values will depend only on the intensity of the PPP, the number $k$ of considered trajectories and a constant $\kappa$ that we will introduce to describe the renewal steps (see \eqref{def:tau_j} and \eqref{def:EventAj}).



\section{The joint exploration process}
\label{sect:JointProcess}

\subsection{Construction}

Let $k\in\N$ be a positive integer. Let us consider $k$ starting points $\u_1,\ldots,\u_k\in\R^2$. In this section, following \cite{RSS15}, we define a discrete time process $\{(g_n(\u_1),\ldots, g_n(\u_k), H_n) : n\geq 0\}$ in an inductive way for the joint exploration of the $k$ paths $\pi^{\u_1},\ldots,\pi^{\u_k}$ so that they move together. This discrete time process is the joint exploration process which makes the subject of this section. The sequence $\{(g_n(\u_1),\ldots, g_n(\u_k)) : n\geq 0\}$ is a representation of the trajectories while $\{H_n : n\geq 0\}$ will be the associated dependence set.

Before defining precisely the joint exploration process, let us first discuss the typical \textit{initial configuration} $(\u_1,\ldots,\u_k,H_0)$ from which the joint exploration process starts. The starting points $\u_1,\ldots,\u_k$ can be deterministic, and possibly with the same ordinate (as in Section \ref{sec:Tail}), or merely points of the PPP $\NN$. In order to cover the case of configurations obtained at good step, we have to take into account some initial extra information encoded with a random compact set $H_0$. Sometimes (as in Section \ref{sec:Tail}), $H_0$ will be empty. For the moment, we only demand that $H_0$ a.s. satisfies
\begin{equation}
\label{AvoidingH0}
\left( \NN \cup \{\u_1,\ldots,\u_k\} \right) \cap \mbox{int}(H_0) = \emptyset ~,
\end{equation}
where $\mbox{int}(H_0)$ denotes the interior of $H_0$. Notice that the points $\u_1,\cdots, \u_k$ can be on the boundary of $H_0$. Extra conditions will be added in Section \ref{sect:GoodStep} but we can omit them for the moment.

Set $g_0(\u_i)= \u_i$ for $i=1,\ldots,k$. In the joint exploration process, only the lowest vertex moves, denoted by $W_{n}^{\text{move}}$, while the $k-1$ other ones remain unchanged. In case several vertices have the same lower ordinate, we move them one by one starting from the leftmost one:
\begin{itemize}
\item[(i)] $W_{0}^{\text{move}}:= \text{argmin}\{\w(1): \w \in \{\u_1, \ldots, \u_k\} \mbox{ and }\w(2)=r_0\}$ where $r_{0}:= \min\{\u_1(2),\ldots , \u_k(2)\}$, and $W_{0}^{\text{stay}}:= \{\u_1, \dots , \u_k\} \setminus W_{0}^{\text{move}}$;
\item[(ii)] For $1\leq i \leq k$,
\begin{align*}
g_{1}(\u_i) :=
\begin{cases}
h \bigl (g_{0}(\u_i), {\NN}\cup W_0^{\text{stay}}\bigr) &\text{ if } g_{0}(\u_i) = W_0^{\text{move}}\\
g_0(\u_i)&\text{ otherwise}.
\end{cases}
\end{align*}
\end{itemize}
After the first step, the history set $H_0$ is updated into $H_{1} = H_{1}(\u_1,\ldots, \u_k)$:
$$
H_{1} := \Big( H_0 \cup {B^+(W_{0}^{\text{move}},\|h(W_{0}^{\text{move}})-W_{0}^{\text{move}}\|_2)} \Big) \cap \mathbb{H}^+(W_{1}^{\text{move}}(2)) ~,
$$
where $W_{1}^{\text{move}}(2):=\min\{g_{1}(\u_i)(2): 1\leq i \leq k\}$ is the next moving vertex.\\

By induction, given $(g_{n}(\u_1),\ldots,g_{n}(\u_k), H_n(\u_1,\ldots, \u_k))$, for any $n\geq 1$, let us set
\begin{itemize}
\item[(i)] $W_{n}^{\text{move}}:= \text{argmin}\{\w(1): \w \in \{g_n(\u_1), \ldots, g_n(\u_k)\} \mbox{ and } \w(2)=r_{n}\}$ where $r_{n} := \min\{g_n(\u_1)(2),\ldots, g_n(\u_k)(2)\}$, and $W_{n}^{\text{stay}}:= \{g_n(\u_i): 1 \leq i \leq k\} \setminus W_{n}^{\text{move}}$;
\item[(ii)] For $1\leq i \leq k$,
\begin{align*}
g_{n+1}(\u_i) :=
\begin{cases}
h \bigl(g_{n}(\u_i),{\NN} \cup W_n^{\text{stay}}\bigr) &\text{ if } g_{n}(\u_i) = W_n^{\text{move}}\\
g_n(\u_i)&\text{ otherwise}.
\end{cases}
\end{align*}
\end{itemize}
When $g_n(\u_1),\dots, g_n(\u_k)$ all have different ordinates-- and this is a.s. the case whenener they are points of $\mathcal{N}$ --, $W_n^{\text{move}}$ is given by the $g_{n}(\u_i)$ having the smallest ordinate. When this smallest ordinate is realized by at least two vertices, then $W_n^{\text{move}}$ corresponds to the one having the smallest abscissa.

After the $(n+1)$-th move, the new level $W_{n+1}^{\text{move}}(2):=\min\{g_{n+1}(\u_i)(2), 1\leq i\leq k\}$ allows to define the next history set $H_{n+1}=H_{n+1}(\u_1,\ldots, \u_k)$:
$$
H_{n+1} := \Big( H_{n} \cup {B^+(W_{n}^{\text{move}},\|h(W_{n}^{\text{move}})-W_{n}^{\text{move}}\|_2)} \Big) \cap \mathbb{H}^+(W_{n+1}^{\text{move}}(2)) ~.
$$

\begin{figure}[!ht]
\begin{center}
\psfrag{x}{$\small{\u_1}$}
\psfrag{y}{$\small{\u_2}$}
\psfrag{z}{$\small{\u_3}$}
\psfrag{a}{${\tiny W_{5}^{\text{move}}(2)}$}
\psfrag{b}{${\tiny W_{2}^{\text{move}}(2)}$}
\psfrag{d}{${\tiny W_{6}^{\text{move}}(2)}$}
\psfrag{e}{${\tiny W_{0}^{\text{move}}(2)}$}
\includegraphics[width=11cm,height=4.2cm]{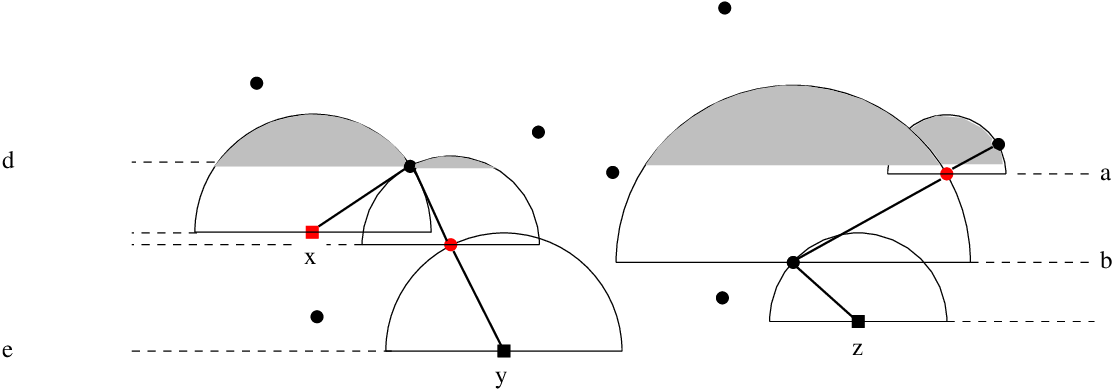}
\caption{\label{Fig:gn} {\small \textit{First $6$ steps of the joint process $(g_{n}(\u_1), g_{n}(\u_2), g_{n}(\u_3))_{n\geq 0}$ starting from $\u_1, \u_2, \u_3$ (given by the squares). To simplify the picture, we take $H_0
= \emptyset$. The first move concerns $\u_2$, i.e. $W_{0}^{\text{move}}=\u_2$, while the second and third ones concern the trajectory starting at $\u_3$. The triplet $(g_{3}(\u_1), g_{3}(\u_2), g_{3}(\u_3))$ is represented by red vertices. Until the fourth step, $\u_1$ has not moved yet: $g_{4}(\u_1)=\u_1$. And at the fifth step, the ancestor of $W_4^{\text{move}}=g_{4}(\u_1)$ is an element of $W_4^{\text{stay}}$, namely $g_{4}(\u_2)$, which means that $\pi^{\u_1}$ and $\pi^{\u_2}$ coalesce. The grey area corresponds to $H_{6}(\u_1,\u_2,\u_3)$. On both sides of the picture, the levels $W_{n}^{\text{move}}(2)$, $0\leq n\leq 6$, are indicated.}}}
\end{center}
\end{figure}

In the sequel, we need to work with a filtration encoding all the information until the current step including the initial configuration $(\u_1,\ldots,\u_k,H_0)$. For any integer $n$, let us set
$$
\mathcal{F}_n := \sigma \Big \{g_l(\u_i) ,\ i\in \{1,\dots k\},\ 0\leq l\leq n ,\ H_0 \Big\} ~.
$$

The next result summarizes some elementary properties of the joint exploration process which can be proved by induction.

{\begin{lemma}
\label{lem:ElemProperties}
Under Assumption (\ref{AvoidingH0}), the following properties hold.
\begin{itemize}
\item[$(i)$] The joint exploration process $\{(g_n(\u_1),\ldots, g_n(\u_k),H_n) : n \geq 0\}$ is an $(\mathcal{F}_n)$-Markov chain with state space $(\R^2)^k\times \{A \subseteq \R^2 : A \text{ is compact}\}$.
\item[$(ii)$] A.s. and for any $n$, $(\NN \cup \{\u_1,\ldots,\u_k\}) \cap \mbox{int}(H_n)=\emptyset$.
\item[$(iii)$] A.s. and for any $n$, any $1\leq i\leq k$, the vertex $g_n(\u_i)$ necessarily lies on the boundary of $\mathbb{H}^+(W_{n}^{\text{move}}(2)) \setminus H_n$ whenever it is different from the $\u_i$'s (see Figure \ref{Fig:gn}).
\item[$(iv)$] A.s. the sequence $(W_{n}^{\text{move}}(2))_{n\geq 0}$ is non-decreasing.
\end{itemize}
\end{lemma}}

\subsection{An auxiliary process}

The fact that the interior part of $H_n$ avoids the PPP $\mathcal{N}$ provides information {(coming from past steps)} on which the next steps of the joint exploration process $\{(g_n(\u_1),\ldots, g_n(\u_k), H_n): n\geq 0\}$ depend. {This dependence phenomenon is the main obstacle to the study of the joint exploration process since} it kills all direct Markov properties.

A first tool to deal with this difficulty consists in the use of an auxiliary discrete time process $\{\widetilde{g}_n(\u_1),\ldots, \widetilde{g}_n(\u_k), \widetilde{H}_n: n\geq 0\}$ starting from the same initial configuration $(\u_1,\ldots,\u_k,H_0)$. This new exploration process obeys the same evolution rule as the original one{-- this is the claim of Proposition \ref{prop:HistMarkov} below --but} each move uses a new PPP on $\R^2$, independent of those previously used. The use of independent PPP's at each move will be very useful to exhibit independent r.v.'s in the sequel. This amounts to throwing at each step of the construction a new PPP outside the region already explored, namely the dependence set. This technique was already used in \cite{BB07} without being clearly stated.

Let us explain this more precisely. Consider a collection $\{\NN_n : n \in \N\}$ of i.i.d. Poisson point processes on $\R^2$, independent of the original process $\NN$ from which $\{g_{n}(\u_1),\ldots,g_{n}(\u_k), H_n: n\geq 0\}$ is defined. Set $\widetilde{g}_0(\u_i)= \u_i$ for $1\leq i \leq k$, $\widetilde{H}_0 = H_0$.
\begin{itemize}
\item[(i)] $\widetilde{W}_{0}^{\text{move}}:= \text{argmin}\{\w(1): \w \in \{\u_1, \ldots, \u_k\} \mbox{ and } \w(2)=\tilde{r}_{0}\}$ where $\tilde{r}_{0}:= \min\{\u_1(2), \ldots, \u_k(2)\}$, and $\widetilde{W}_{0}^{\text{stay}}:= \{\u_1, \dots , \u_k\} \setminus \widetilde{W}_{0}^{\text{move}}$;
\item[(ii)] For $1\leq i \leq k$,
\begin{align*}
\widetilde{g}_{1}(\u_i) :=
\begin{cases}
h \bigl (\widetilde{g}_{0}(\u_i), (\NN_1 \setminus \widetilde{H}_0)\cup \widetilde{W}_0^{\text{stay}}\bigr) &\text{ if } \widetilde{g}_{0}(\u_i) = \widetilde{W}_0^{\text{move}}\\
\widetilde{g}_0(\u_i)&\text{ otherwise}.
\end{cases}
\end{align*}
\end{itemize}

We use the PPP $\mathcal{N}_1\setminus \widetilde{H}_0$ to construct $\widetilde{g}_1(\u_1),\cdots \widetilde{g}_1(\u_k)$.
The history set $\widetilde{H}_{1} = \widetilde{H}_{1}(\u_1,\ldots, \u_k)$ after the first move is defined as:
$$
\widetilde{H}_{1} := \Big( \widetilde{H}_0 \cup {B^+(\widetilde{W}_{0}^{\text{move}},\|h(\widetilde{W}_{0}^{\text{move}})-\widetilde{W}_{0}^{\text{move}}\|_2)} \cap \mathbb{H}^+(\widetilde{W}_{1}^{\text{move}}(2)) ~,
$$
where $\widetilde{W}_{1}^{\text{move}}(2):=\min\{\widetilde{g}_{1}(\u_i)(2): 1\leq i \leq k\}$.

Conditional on $(\widetilde{g}_n(\u_1),\ldots,\widetilde{g}_n(\u_k), \widetilde{H}_n)$ let
\begin{itemize}
\item[(i)] $\widetilde{W}_{n}^{\text{move}}:= \text{argmin}\{\w(1): \w \in \{\widetilde{g}_n(\u_1), \ldots, \widetilde{g}_n(\u_k)\} \mbox{ and } \w(2)=\tilde{r}_{n}\}$ where $\tilde{r}_{n} := \min\{\widetilde{g}_n(\u_1)(2), \ldots, \widetilde{g}_n(\u_k)(2)\}$, and $\widetilde{W}_{n}^{\text{stay}}:= \{\widetilde{g}_n(\u_i): 1 \leq i \leq k\} \setminus \widetilde{W}_{n}^{\text{move}}$;
\item[(ii)] For $1\leq i \leq k$,
\begin{align*}
\widetilde{g}_{n+1}(\u_i) :=
\begin{cases}
h \bigl(\widetilde{g}_{n}(\u_i),(\NN_{n+1} \setminus \widetilde{H}_n)\cup \widetilde{W}_n^{\text{stay}} \bigr) &\text{ if } \widetilde{g}_{n}(\u_i) = \widetilde{W}_n^{\text{move}}\\
\widetilde{g}_n(\u_i) & \text{ otherwise}.
\end{cases}
\end{align*}
\end{itemize}
{Note that, to get $\widetilde{g}_{n+1}(\u_i)$ in the above definition, we re-sample the PPP only outside the explored region, i.e. with $\NN_{n+1} \setminus \widetilde{H}_n$, since the PPP $\NN_{n+1}$ may have points in $\widetilde{H}_n$. This precaution was not required for the original exploration process since it uses at each step the same PPP $\NN$ which avoids the current history set $H_n$.}

The joint history set $\widetilde{H}_{n+1}=\widetilde{H}_{n+1}(\u_1,\ldots,\u_k)$ at the $(n+1)$-th move is given by:
$$
\widetilde{H}_{n+1} := \Big( \widetilde{H}_{n} \cup {B^+(\widetilde{W}_{n}^{\text{move}},\|h(\widetilde{W}_{n}^{\text{move}})-\widetilde{W}_{n}^{\text{move}}\|_2)} \Big) \cap \mathbb{H}^+(\widetilde{W}_{n+1}^{\text{move}}(2)) ~,
$$
where $\widetilde{W}_{n+1}^{\text{move}}(2)=\min\{\widetilde{g}_{n+1}(\u_i)(2): 1\leq i \leq k\}$.

\begin{prop}
\label{prop:HistMarkov}
{Under Assumption (\ref{AvoidingH0}),} the joint exploration process $\{(g_n(\u_1),\ldots,g_n(\u_k),H_n) : n \geq 0\}$ and the auxiliary exploration process $\{(\widetilde{g}_n(\u_1),\ldots,\widetilde{g}_n(\u_k),\widetilde{H}_n) : n \geq 0\}$ are identically distributed.
\end{prop}

\begin{proof}
Let $\{\NN_n : n \in \N\}$ be a collection of i.i.d. Poisson processes on $\R^2$. We work conditional on $(g_n(\u_1),\ldots, g_n(\u_k),H_n) = (\x_1,\ldots,\x_k,\Lambda_n)$, for some $\x_1,\ldots,\x_k \in \R^{2}$ and $\Lambda_n \subset \R^2$, on $\{(g_j(\u_1),\ldots, g_j(\u_k),H_j) : j<n\} $ and on $\mathcal{N}\cap H_0=\emptyset$. The region $\mathbb{H}^+(\min\{\x_i(2): 1 \leq i \leq k\}) \setminus \Lambda_n$ has not been explored yet and the Poisson point process $\NN$ on this region can be replaced with any independent Poisson point process $\mathcal{N}_{n+1}$. Thus, we have in distribution that:
$$
 g_{n+1}(\u_i)
\stackrel{d}{=} \left\{
\begin{array}{ll}
h \bigl(\x_i, (\NN_{n+1}\setminus \Lambda_n)\cup \{\x_1,\ldots,\x_{k}\} \bigr) & \mbox{ if }\, \x_i = W_n^{\text{move}}\\
\x_i & \mbox{ otherwise.}
\end{array}
\right.
$$
For $\x_i=W_n^{\text{move}}$, setting $\x^\prime_i:= h \bigl( \x_i, (\NN_{n+1}\setminus \Lambda_n)\cup \{\x_1,\ldots,\x_{k}\}\bigr) $, we have
$$
H_{n+1} \stackrel{d}{=} \left( B^{+}(\x_i,\|\x_i - \x^\prime_i\|_2) \cup \Lambda_n \right) \cap \mathbb{H}^{+} \big( \x'_{i}(2) \wedge \min\{ \x_j(2) : j\not= i \} \big) ~.
$$
Hence, the original joint exploration process and the auxiliary one have the same transition probabilities. They are identically distributed.

Moreover, conditional on $\mathcal{F}_n$, the process  $(g_{n+1}(\u_1),\ldots, g_{n+1}(\u_k),H_{n+1})$ admits a random mapping representation of the form
$$
(g_{n+1}(\u_1),\ldots, g_{n+1}(\u_k),H_{n+1}) \stackrel{d}{=} f((g_{n}(\u_1),\ldots, g_{n}(\u_k),H_n), \NN_{n+1})
$$
for some measurable mapping $f$. This gives its Markovian character (see \cite{LPW08}).
\end{proof}

\section{Good steps}
\label{sect:GoodStep}

Let us define the height of any non empty bounded subset $\Delta$ of $\R^2$, as
$$
L(\Delta) := \sup \{\y(2)-\x(2) : \x,\y\in\Delta\}
$$
and $L(\emptyset)=0$. The goal of this section, i.e. Proposition \ref{propo:Tau_tail}, consists in stating that the height of the history set $L(H_{n})$ returns regularly under a given positive {integer} $\kappa$ which will be specified later.

Precisely, let us set $\tau_0=\tau_0(\u_1,\ldots,\u_k)=0$ and for $j\geq 1$,
\begin{equation}
\label{def:tau_j}
\tau_j = \tau_j(\u_{1},\ldots, \u_k) := \inf \left\lbrace kn > \tau_{j-1} : \begin{array}{c}
n\geq 1, \; L(H_{kn}) \leq \kappa \; \mbox{ and } \\
W^{\text{move}}_{kn}(2) \geq W^{\text{move}}_{\tau_{j-1}}(2)+\kappa+1
\end{array} \right\rbrace ~.
\end{equation}
Such a step is called a \textit{good step} of the joint process $\{(g_n(\u_1),\ldots,g_n(\u_k),H_n) : n \geq 0\}$. At a good step, the height of the history set is at most $\kappa$. The condition that $W^{\text{move}}_{\tau_{j}}(2)-W^{\text{move}}_{\tau_{j-1}}(2)$ should be more than $\kappa+1$ is to ensure that the history regions involved at different good steps are disjoint. As additional and technical requirements, $\tau_{j}$ has to be a multiple of the number $k$ of trajectories. This condition portends that in the sequel we will consider blocks of $k$ consecutive steps. Let us also remark that the $\tau_j$'s are stopping times w.r.t. the filtration $(\mathcal{F}_n)_{n\geq 0}$.

In Section \ref{sect:Renewal}, we will select some suitable (in some sense) good steps and will call them \textit{renewal steps}. \\

Only for this section, we will work with the auxiliary exploration process $\{(\widetilde{g}_n(\u_1),\ldots,\widetilde{g}_n(\u_k),\widetilde{H}_n): n \geq 0\}$ instead of the (original) joint exploration process, and for ease of notation, we denote this process itself by $\{(g_n(\u_1),\ldots,g_n(\u_k),H_n) : n \geq 0\}$.\\

{Proposition \ref{propo:Tau_tail} holds whenever the following conditions on the initial configuration $(\u_1,\ldots,\u_k,H_0)$ are satisfied.
\begin{itemize}
\item[(\textbf{H1})] \textbf{Shape of $H_0$.} The initial history set $H_0$ is a compact set defined as the intersection with $\mathbb{H}^+(W_0^{\text{move}}(2))$ of a finite number of closed balls whose centers are in $\mathbb{H}^{-}(W_0^{\text{move}}(2))$. Moreover, the height of $H_0$ is such that $L(H_0) \leq \kappa$.
\item[(\textbf{H2})] \textbf{Locations of $\u_1,\ldots,\u_k$.} The starting points $\u_1,\ldots,\u_k$ are deterministic, they belong to the closure of $\mathbb{H}^+(W_0^{\text{move}}(2)) \setminus H_0$ and satisfy
$$
\max\{\u_1(2),\ldots,\u_k(2)\} \leq W_0^{\text{move}}(2) + \kappa ~.
$$
\end{itemize}}

Roughly speaking, Assumption (\textbf{H1}) says that the initial history set $H_0$ is a finite union of balls intersected with the half-space $\mathbb{H}^+(W_0^{\text{move}}(2))$ and whose height is bounded by $\kappa$. Assumption (\textbf{H2}) requires that all the information associated to the initial configuration is contained in a strip of height $\kappa$, which is a little bit more than avoiding the interior part of $H_0$: the reason for this will appear clearly in the proof of Lemma \ref{lem:UnexploredCone}.

{Mainly two types of initial configurations will be considered. Either $H_0=\emptyset$-- this case is covered by (\textbf{H1}) --and the starting points are deterministic with possibly the same ordinate (as in Section \ref{sec:Tail}). Or $H_0\not=\emptyset$ and the $\u_i$'s are located on the boundary of $\mathbb{H}^+(W_0^{\text{move}}(2)) \setminus H_0$. This second type exactly corresponds to configurations obtained at a good step. An example of this second type is given by $(g_6(\u_1),g_6(\u_2),g_6(\u_3),H_6)$ in Figure \ref{Fig:gn}.}\\

{From now on, we assume that (\textbf{H1}) and (\textbf{H2}) hold.} Proposition \ref{propo:Tau_tail} states that the number of steps between two consecutive good steps can be stochastically dominated by a r.v. having exponential decay.

\begin{proposition}
\label{propo:Tau_tail}
Let $j\geq 0$. There exists a r.v. $T$ whose distribution does not depend on $\mathcal{F}_{\tau_j}$ such that, for all $n$,
\begin{equation}
\label{Tau_tail}
\P\big(\tau_{j+1}-\tau_j\geq n \ |\ \mathcal{F}_{\tau_j}\big)\leq \P \big( T \geq n \big) \leq C_0 e^{-C_1 n} ~.
\end{equation}
\end{proposition}

We will prove Proposition \ref{propo:Tau_tail} through a sequence of lemmas. To understand how our proof is organized, we start with describing the evolution of the height of the history set during a single step. Two situations may actually occur. If the semi-ball $B^{+}(W^{\text{move}}_{n},\|h(W^{\text{move}}_{n}) - W^{\text{move}}_{n}\|_2)$ created during the $(n+1)$-th move, does not exceed the horizontal line $\{\x : \x(2)=W^{\text{move}}_{n}(2)+L(H_{n})\}$ then
$$
L(H_{n+1}) = L(H_{n}) - \big( W^{\text{move}}_{n+1}(2) - W^{\text{move}}_{n}(2) \big) <L(H_n) ~.
$$
In this case, the height of the history set is decreasing and, on some suitable events (occurring with positive probability), we will be able to quantify its decrease. See Lemmas \ref{lem:ProbIn}, \ref{lem:E_nW_n} and \ref{lem:EnLn}.

Otherwise, the new height $L(H_{n+1})$ is realized by the last created semi-ball $B^{+}(W^{\text{move}}_{n},\|h(W^{\text{move}}_{n}) - W^{\text{move}}_{n}\|_2)$ and
$$
L(H_{n+1}) = \|h(W^{\text{move}}_{n}) - W^{\text{move}}_{n}\|_2 - \big( W^{\text{move}}_{n+1}(2) - W^{\text{move}}_{n}(2) \big) ~.
$$
In this second case, the height of the history set may increase or not. A priori, a large distance $\|h(W^{\text{move}}_{n}) - W^{\text{move}}_{n}\|_2$ should occur with small probability since this would force the PPP to avoid the (large) semi-ball $B^{+}(W^{\text{move}}_{n},\|h(W^{\text{move}}_{n}) - W^{\text{move}}_{n}\|_2)$. However, a large part of that semi-ball can be already covered by the history set $H_{n}$, which by definition avoids the PPP. In this case, having a large distance $\|h(W^{\text{move}}_{n}) - W^{\text{move}}_{n}\|_2$ becomes quite possible. Lemmas \ref{lem:UnexploredCone} and \ref{lem:IncreaseLn} allow us to overcome this obstacle and to control the growth of $L(H_n)$.

In both situations, the sequence $\{L(H_n) : n\geq 0\}$ satisfies the following fundamental and useful induction relation: a.s. and for any $n$,
\begin{equation}
\label{InductionLHn}
L(H_{n+1}) \leq \max\{ L(H_{n}), \|h(W^{\text{move}}_{n}) - W^{\text{move}}_{n}\|_2 \} ~.
\end{equation}

At the end of this section, we will combine these results in Lemmas \ref{lem:MngeqLn} and \ref{lem:M_tail} to get Proposition \ref{propo:Tau_tail}.

\subsection{How much is $L(H_n)$ increasing?}

Let us introduce some notations. For a real number $l>0$ and an integer $n\geq 0$, let us set
$$
g^{\uparrow,l}_n := W_n^{\text{move}} + (0,l)
$$
(recall that $W_n^{\text{move}}$ a.s. denotes a single point). Let
$$
C_{\pi/2}(\mathbf{0}) := \{ re^{\i\theta} :\, r > 0 , \, \theta \in [\pi/4,3\pi/4] \}
$$
be the cone with apex $\bf{0}$ and making an angle $\pi/4$ with the vertical axis. We also define, for $\x\in\R^2$, $C_{\pi/2}(\x):=\x + C_{\pi/2}(\mathbf{0})$.

Conditional on the current configuration $(g_{n}(\u_1),\ldots,g_{n}(\u_k), H_{n})$, the next lemma exhibits deterministic regions avoiding the history set $H_{n}$. Such regions are unexplored yet and will allow us to control how the history set grows (see Lemma \ref{lem:IncreaseLn}). Notice that Baccelli and Bordenave used in \cite[Lemma 4.2]{BB07} a similar geometric argument, which is false. Actually, it is impossible to exhibit a cone, with a positive and deterministic angle and with apex at the moving vertex $W_n^{\text{move}}$, which almost surely avoids the history set $H_n$. To get such a property, the cone has to be pushed upward and this is what we do with $g^{\uparrow,l}_n$.

\begin{lemma}
\label{lem:UnexploredCone}
For all $n\geq 0$ and for any $l\geq L(H_n)/2$, the cone $C_{\pi/2}(g^{\uparrow,l}_n)$ a.s. avoids the history set $H_n$, i.e. $C_{\pi/2}(g^{\uparrow,l}_n)\cap H_n = \emptyset$.
\end{lemma}

Remark that although the unexplored cone $C_{\pi/2}(g^{\uparrow,L(H_n)/2}_n)$ avoids the history set $H_n$, it could contain a starting point $g_n(\u_i)=\u_i$ which has not moved yet {(until step $n$)} and could still be outside $H_n$.

\begin{proof}[Proof of Lemma \ref{lem:UnexploredCone}]
{Let $l\geq L(H_n)/2$. By definition of the history set $H_n$, we have to check that the cone $C_{\pi/2}(g^{\uparrow,l}_n)$ avoids each semi-ball $B^{+}(W_m^{\text{move}},\|W_m^{\text{move}}-h(W_m^{\text{move}})\|_{2})$ created at a previous step $m<n$ and each semi-ball contributing to $H_0$ (recall Assumption (\textbf{H1})). Let us denote by $B^{+}(A,R)$ such generic semi-ball.}

By translation and symmetry, we can assume without loss of generality that $g^{\uparrow,l}_n=(0,0)$. So $W_n^{\text{move}}=(0,-l)$. {Here, we use in a crucial way that $W_n^{\text{move}}$ belongs to the boundary of $\mathbb{H}^+(W_n^{\text{move}}(2)) \setminus H_n$, i.e. Assumption (\textbf{H2}). Also, by Assumption (\textbf{H1}), $A(2)\leq W_n^{\text{move}}(2)=-l$ and $B^{+}(A,R)$ is below the horizontal line with ordinate $l$ as $l\geq L(H_n)/2$.} So, the worst case is obtained when the semi-ball $B^{+}(A,R)$ realizes the height $L(H_n)$ and is tangent to $W_n^{\text{move}}$ with a maximal ordinate $A(2)$, i.e. $A=(W_n^{\text{move}}(1)+2l,W_n^{\text{move}}(2))=(2l,-l)$ and $R=2l$. See Figure \ref{fig:unexplored} for an illustration of this worst situation.

Finally, an elementary geometric computation allows to conclude. If the cone $C_{\pi/2}(g^{\uparrow,l}_n)$ overlaps $B^{+}(A,2l)$ then the point $M=(l/2,l/2)$ has to belong to $B^{+}(A,2l)$ since it is the closest point to $A$ in the cone. But $\|A-M\|_{2}^{2}=18l^{2}/4>(2l)^{2}$. So this concludes the proof.
\end{proof}

For $n\geq 0$, we denote by $\zeta_{n+1}$ the distance between $g^{\uparrow,L(H_n)/2}_n$ and its nearest Poisson point inside the unexplored cone $C_{\pi/2}(g^{\uparrow,L(H_n)/2}_n)$:
\begin{equation}
\label{zeta(1)_n}
\zeta_{n+1} := \inf \left\{ \|g^{\uparrow,L(H_n)/2}_n - \x\|_2 : \, \x \in {\cal N}_{n+1} \cap C_{\pi/2}(g^{\uparrow,L(H_n)/2}_n) \right\} ~.
\end{equation}
As we will consider blocks of $k$ consecutive steps in the sequel, let us introduce for $n\geq 0$,
\begin{equation}
\label{def:X1}
X_{n+1} := \sum_{j = 1}^{k} \left( \lfloor 2 \zeta_{kn+j} \rfloor + {1} \right) ~.
\end{equation}
{The random variable $X_{n+1}$ is an integer-valued random variable and the reason for choosing $\lfloor 2 \zeta_{kn+j} \rfloor + 1$ will appear in the proof of Lemma \ref{lem:IncreaseLn}. Later, this integer valued random variable $X_{n+1}$ will be used to construct a discrete state space Markov chain to dominate the `height' process $L(H_n)$.}

\begin{figure}[!ht]
\begin{center}
\psfrag{x}{$\small{W^{\text{move}}_n}$}
\psfrag{y}{$\small{\x}$}
\psfrag{z}{$\small{L(H_{n})}$}
\psfrag{m}{$\small{\zeta_{n+1}}$}
\psfrag{n}{$\small{C_{\pi/2}(\x)}$}
\includegraphics[width=13.2cm,height=4.4cm]{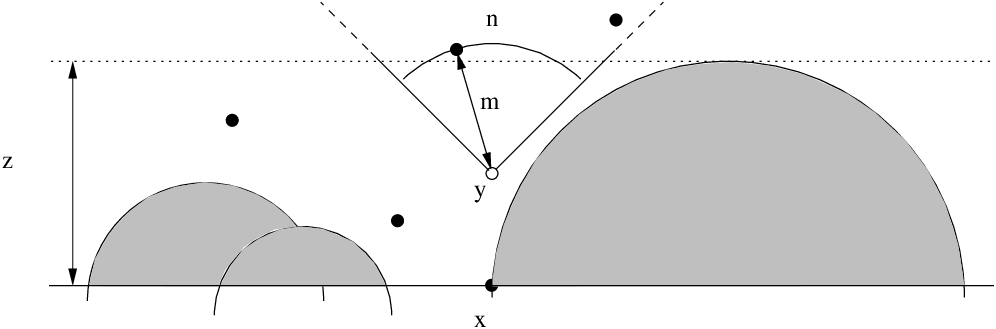}
\caption{\label{fig:unexplored} {\small \textit{Black vertices are Poisson points. The gray area corresponds to the history set $H_{n}$. The white point-- denoted by $\x$ --is $g^{\uparrow,L(H_n)/2}_n$. The cone with apex $\x$ and bisector the ordinate axis is the unexplored cone and avoids the history set $H_{n}$.}}}
\end{center}
\end{figure}

The next result says that when the height of the history set increases between steps $kn$ and $k(n+1)$ then the new height $L(H_{k(n+1)})$ is bounded from above by the r.v. $X_{n+1}$ which admits an exponential tail.

\begin{lemma}
\label{lem:IncreaseLn}
Using the previous notations:
\begin{itemize}
\item[(i)] For all $n \geq 0$, the following inequality holds with probability $1$:
$$
L(H_{k(n+1)}) \mathbf{1}_{\{L(H_{k(n+1)}) > L(H_{kn})\}} \leq X_{n+1} ~.
$$
\item[(ii)] The r.v.'s $\{X_{n+1} : n \geq 0\}$ are i.i.d. and satisfy $\forall n,m \geq 0$,
\begin{equation}
\label{X^(1)ExpoTail}
\P( X_{n+1} > m ) \leq C_0 e^{-C_1 m} ~.
\end{equation}
\end{itemize}
\end{lemma}

\begin{proof}
{Let us first show $(i)$. For a single trajectory (i.e. $k=1$), we have a.s.
\begin{equation}
\label{Item(i)k=1}
L(H_{n+1}) \mathbf{1}_{\{L(H_{n+1}) > L(H_{n})\}} \leq \lfloor 2\zeta_{n+1}\rfloor + 1 ~.
\end{equation}
This implies that $L(H_{n+1}) \leq \max\{ L(H_{n}),\lfloor 2\zeta_{n+1}\rfloor +1\}$. Applying this inequality $k$ times leads to Item $(i)$ for any $k\geq 2$.}

Denoting by $\y$ the element of the unexplored cone realizing the r.v. $\zeta_{n+1}$, it follows:
{\begin{eqnarray}
\label{LHnZeta}
\| h(W_n^{\text{move}}) - W_n^{\text{move}} \|_2 & \leq & \| \y - W_n^{\text{move}}\|_2 \nonumber\\
& \leq & \zeta_{n+1} + \| g^{\uparrow,L(H_n)/2}_n - W_n^{\text{move}}\|_2 \nonumber\\
& = & \zeta_{n+1} + L(H_n)/2 ~.
\end{eqnarray}
If $L(H_{n+1}) > L(H_{n})$ then the last created semi-ball increases the history set. So,
$$
\| h(W_n^{\text{move}}) - W_n^{\text{move}}\|_{2} \geq L(H_{n+1}) \geq L(H_n) ~.
$$
With (\ref{LHnZeta}), we get $\zeta_{n+1}\geq L(H_n)/2$ and $L(H_{n+1})\leq 2\zeta_{n+1}$. And (\ref{Item(i)k=1}) follows.}

Item $(ii)$ is mainly based on the independence between the random variables $\zeta_{n+1}$, $n\geq 0$, which is due to the fact that independent PPP's are used for each step of the joint process $\{(g_n(\u_1),\ldots,g_n(\u_k),H_n) : n \geq 0\}$. Moreover, by Lemma \ref{lem:UnexploredCone}, the r.v. $\zeta_{n+1}$'s are i.i.d. with an exponential tail distribution since
$\P(\zeta_{n+1}>r)$ is the probability that there is no Poisson point in $C_{\pi/2}(\mathbf{0})\cap B(\mathbf{0},r)$. The same holds for the $X_{n+1}$'s.
\end{proof}

\subsection{How much is $L(H_n)$ decreasing?}

Now let us show that $(L(H_n))_{n\geq 0}$ is submitted to a `negative drift' so that the sequence regularly returns to small values. {We introduce an event of positive probability on which the ordinate of the moving vertex indeed increases of at least 1 between the $kn$-th and $k(n+1)$-th steps. Working a bit more, we will obtain as a consequence  that the height of the history set decreases by at least 1 on this event if it is greater than $\kappa$.} Notice that such event also allows to control the number of steps needed for the ordinate of the moving vertex to reach a distance at least $\kappa+1$ from the last good step.\\

For $\x\in\R^2$ and for $w,l>0$, the rectangle of width $2w$ and of height $l$, whose base is centered at $\x$, is denoted by
$$
\text{Rec}(\x;w,l) := \x + [-w,w]\times [0,l] ~.
$$
Thus we set
\begin{equation}
\label{def:ln}
l_n := \inf\{l \geq 0: \text{Area}(\text{Rec}(g^{\uparrow,1}_n;1,l) \setminus H_n) \geq 1/2 \} ~.
\end{equation}
In other words, $l_n$ is the random height of the rectangle centered at $W_n^{\text{move}}+(0,1)$ with width $2$ so that the area of its unexplored part becomes at least $1/2$. The justification of the constant $1/2$ in the definition of $l_n$ will appear in the proof of Lemma \ref{lem:E_nW_n}. Besides, the overlap of $\text{Rec}(g^{\uparrow,1}_n;1,L(H_n)/2)$ with the unexplored cone $C_{\pi/2}(g^{\uparrow,L(H_n)/2}_n)$ has area $1$. Thanks to Lemma \ref{lem:UnexploredCone}, this means that a.s.
\begin{equation}
\label{lnBound}
l_n \leq \frac{L(H_n)}{2} ~.
\end{equation}

For any integer $n\geq 0$, $I_{n+1}$ is the indicator random variable defined as
$$
I_{n+1} := \mathbf{1}_{\{(\text{Rec}(g^{\uparrow,1}_n;1,l_n) \setminus H_n) \cap {\cal N}_{n+1}\neq\emptyset \,\mbox{ and }\, \text{Rec}(W^{\text{move}}_n;5,1) \cap {\cal N}_{n+1}=\emptyset \}} ~.
$$
Let us now explain the ideas behind Lemmas \ref{lem:ProbIn}, \ref{lem:E_nW_n} and \ref{lem:EnLn}. First notice that $\text{Rec}(g^{\uparrow,1}_n;1,l_n) \setminus H_n$ and $\text{Rec}(W^{\text{move}}_n;5,1)$ are two disjoint regions with area $1/2$ and $10$ respectively. So the events indicated by the $I_{n+1}$'s all occur with the same fixed positive probability, denoted by $p_0$ in Lemma \ref{lem:ProbIn}. Such an event will be pleasant in the sense that, provided there is no point of $W_n^{\text{stay}}$ in the horizontal rectangle $\text{Rec}(W^{\text{move}}_n ; 5,1)$, the ancestor $h(W_n^{\text{move}})$ advances by at least 1 in ordinate w.r.t. $W_n^{\text{move}}$. Combining this with (\ref{lnBound}) should force the height of the history set to decrease by at least $1$ during the $(n+1)$-th move. However, it can happen that some points of $W_n^{\text{stay}}$ are in $\text{Rec}(W^{\text{move}}_n;5,1)$ (or $\text{Rec}(g^{\uparrow,1}_n;1,l_n)$) as illustrated in Figure \ref{fig:tricky}. In this case, $h(W_n^{\text{move}})\in W_n^{\text{stay}}$ and the increment $h(W_n^{\text{move}})(2)-W_n^{\text{move}}(2)$ cannot be bounded from below. But, this situation corresponds to the coalescence of two paths among the $\pi^{\u_1},\ldots,\pi^{\u_k}$. Here is the reason why we consider blocks of $k$ consecutive steps: on the event $\{\prod_{j=1}^k I_{kn+j}=1\}$ where such pleasant events occur between the $kn$-th and the $k(n+1)$-th steps, the ordinate of the current moving vertex is forced to progress by at least $1$ (Lemma \ref{lem:E_nW_n}) and the history set to decrease by at least $1$ (Lemma \ref{lem:EnLn}).

\begin{figure}[!ht]
\begin{center}
\psfrag{x}{$\small{W^{\text{move}}_n}$}
\psfrag{y}{$\small{Y}$}
\psfrag{z}{$\small{Z}$}
\psfrag{a}{$\small{l_n}$}
\psfrag{b}{$\small{1}$}
\includegraphics[width=5cm,height=6cm]{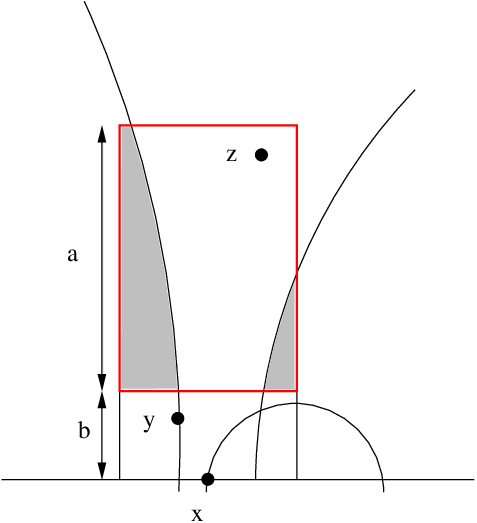}
\caption{\label{fig:tricky} {\small \textit{The red rectangle represents $\text{Rec}(g^{\uparrow,1}_n;1,l_n)$, where $g^{\uparrow,1}_n=W^{\text{move}}_{n}+(0,1)$, which is partially covered by the history set $H_n$ (the gray regions). By definition of $l_n$ the area of $\text{Rec}(g^{\uparrow,1}_n;1,l_n)\setminus H_n$ is equal to $1/2$. The vertex $Y$ is an element of $W^{\text{stay}}_{n}$ while $Z$ is a Poisson point of $\NN_{n+1}$. This picture illustrates the tricky situation occurring in the proof of Lemma \ref{lem:E_nW_n}: although $I_{n+1}=1$, $W^{\text{move}}_{n+1}(2)$ is not larger than $W^{\text{move}}_{n}(2)+1$ since $h(W^{\text{move}}_{n})=Y$.}}}
\end{center}
\end{figure}
\begin{lemma}
\label{lem:ProbIn}
Let $p_0:=(1-e^{-\lambda/2})e^{-10\lambda}>0$ where $\lambda$ denotes the (common) intensity of the Poisson point processes. Then, for any $n\geq 0$,
$$
\P(I_{n} = 1) = p_0 \,\text{ and }\, \P \Big( \prod_{j = 1}^{k} I_{kn + j} = 1 \Big) = p_0^k ~.
$$
\end{lemma}

\begin{proof}
Recall that the process $\{(g_n(\u_1),\ldots,g_n(\u_k),H_n) : n \geq 0\}$ is an ${\cal F}_n$-Markov chain. {Since $\text{Rec}(g^{\uparrow,1}_n;1,l_n)\setminus H_n$ and $\text{Rec}(W^{\text{move}}_n;5,1)$ are disjoint sets with constant areas,} we get a.s.
\begin{eqnarray*}
\E \big( I_{n+1} | {\cal F}_{n} \big) & = & \P \big( (\text{Rec}(g^{\uparrow,1}_n;1,l_n) \setminus H_n) \cap {\cal N}_{n+1} \neq \emptyset \,|\, {\cal F}_{n} \big) \\
 & & \hspace*{1.5cm}\times \P \big( \text{Rec}(W^{\text{move}}_n;5,1) \cap {\cal N}_{n+1} = \emptyset \,|\, {\cal F}_{n} \big) \\
& = & (1-e^{-\lambda/2}) e^{-10\lambda} =: p_0~.
\end{eqnarray*}
Taking expectation, $\P(I_{n+1}=1)$ also equals $p_0$.

Thus, the conditional expectation $\E(\prod_{j = 1}^{k} I_{kn + j} | {\cal F}_{kn})$ can be written as
$$
\E \Big( \prod_{j = 1}^{k-1} I_{kn + j} \E \Big( I_{k(n+1)} \,|\, {\cal F}_{kn} , \NN_{kn+1} , \ldots , \NN_{k(n+1)-1} \Big) \,|\, {\cal F}_{kn} \Big)
$$
where
$$
\E \Big( I_{k(n+1)} \,|\, {\cal F}_{kn} , \NN_{kn+1} , \ldots , \NN_{k(n+1)-1} \Big) = \E \Big( I_{k(n+1)} \,|\, {\cal F}_{k(n+1)-1} \Big) = p_0
$$
a.s. thanks to the previous computation. Taking expectation, we get
$$
\P \Big( \prod_{j = 1}^{k} I_{kn + j} = 1 \Big) = p_0 \, \P \Big( \prod_{j = 1}^{k-1} I_{kn + j} = 1 \Big)=p_0^k\quad \mbox{ a.s.}
$$
by an immediate induction.
\end{proof}

\begin{lemma}
\label{lem:E_nW_n}
On the event $\{\prod_{j=1}^{k} I_{kn+j}=1\}$, the ordinate of the moving vertex increases by at least $1$ between the $kn$-th and the $k(n+1)$-th steps:
\begin{equation}
\label{E_nW_n}
W^{\text{move}}_{k(n+1)}(2) \geq W^{\text{move}}_{kn}(2) + 1 \quad \mbox{ a.s.}
\end{equation}
\end{lemma}

\begin{proof}
Let us first prove it for only one path, i.e. $k=1$. On the event $\{I_{n+1}=1\}$, the rectangle $\text{Rec}(g^{\uparrow,1}_n;1,l_n)$ contains at least one Poisson point. So, $X:=h(W^{\text{move}}_n)=W^{\text{move}}_{n+1}$ belongs to $B^{+}(W^{\text{move}}_n,l_{n}+2)$. Let us prove that $X(2)\geq W^{\text{move}}_{n}(2)+1$. This is clear by definition of $\{I_{n+1}=1\}$ whenever $l_n \leq 3$. We can now focus on the case $l_n>3$. Without loss of generality we assume that $W^{\text{move}}_n = (0,0)$. Since $k=1$ (and so $W^{\text{stay}}_{n}=\emptyset$), it suffices to prove that the set
$$
U := \{\x\in\R^{2} : \, 1 \leq |\x(1)| \leq l_n + 2 \,\mbox{ and }\, 0 \leq \x(2) \leq 1 \}
$$
is included in $H_n$ {and, consequently, contains no Poisson point. This means that the moving vertex will make a vertical progress of at least $1$.} To do so, let us remark that both points $A:=(1,l_n+1/2)$ and $B:=(-1,l_n+1/2)$ belong to the history set $H_n$. Otherwise, the region $\text{Rec}(g^{\uparrow,1}_n ; 1, l_n)\setminus H_n$ would contain at least one of the two rectangles $[B(1),0]\times [l_n + 1/2,l_n+1]$ or $[0, A(1)]\times [l_n + 1/2,l_n+1]$, each of area $1/2$, which is impossible by definition of $l_n$ (recall (\ref{def:ln})). Now, it is not difficult to check that any semi-ball $B^{+}(g_{m}(\u_{1}),\cdot)$, for $0\leq m\leq n-1$, which contains $A$ but not $W^{\text{move}}_n=(0,0)$ in its interior, also contains the strip $[1, l_n+2]\times [0,1]$ when $l_n>3$. By symmetry, the same holds for the left part of $U$.

It remains to prove (\ref{E_nW_n}) for any $k\geq 2$. If $W^{\text{move}}_{kn + 1}(2)$ is already larger than $W^{\text{move}}_{kn}(2)+1$ then this is also the case for $W^{\text{move}}_{k(n+1)}(2)$. Otherwise, the ancestor of $W^{\text{move}}_{kn}$ coincides with an element of $W^{\text{stay}}_{kn}$: this is the tricky situation described in Figure \ref{fig:tricky}. Actually the worst case is the following: $W^{\text{move}}_{kn},\ldots,W^{\text{move}}_{k(n+1)-2}$ are $k-1$ different vertices which have all merged with $W^{\text{move}}_{k(n+1)-1}$ during the $k-1$ last steps. In other words, the $k$ paths starting from $\u_{1},\ldots,\u_{k}$ were still disjoint at the $kn$-th step but have all coalesced $k-1$ steps after. Then, it remains to apply the argument for $k=1$ to the only remaining path, i.e. to $W^{\text{move}}_{k(n+1)-1}$:
$$
W^{\text{move}}_{k(n+1)}(2) = h(W^{\text{move}}_{k(n+1)-1})(2) \geq W^{\text{move}}_{k(n+1)-1}(2)+1 \geq W^{\text{move}}_{kn}(2) + 1~.
$$

It is the above tricky situation, described in Figure \ref{fig:tricky}, which justifies that we consider blocks of $k$ steps when defining the $\tau_j$'s.
\end{proof}

{Lemma \ref{lem:E_nW_n} leads to the next result which provides a `drift condition': on the event $\{\prod_{j=1}^{k} I_{kn+j}=1\}$, the height of the history set has to decrease by at least $1$ between the steps $kn$ and $k(n+1)$, if it is larger than $\kappa$.}

\begin{lemma}
\label{lem:EnLn}
Without loss of generality we will assume that the constant $\kappa$ appearing in the definition of the $\tau_j$'s \eqref{def:tau_j} is an integer larger than $6$. For any $n\geq 0$, on the event $\{\prod_{j=1}^{k} I_{kn+j}=1\}$, we have a.s. that
$$
L(H_{k(n+1)}) \leq \max \{L(H_{kn}) - 1, \kappa \} ~.
$$
\end{lemma}

\begin{proof}
Let $n\geq 0$. Let us first assume that $L(H_{kn})\leq \kappa$. If $I_{kn+1}=1$ there is a Poisson point in $\text{Rec}(g^{\uparrow,1}_{kn};1,l_{kn})$. So, by (\ref{lnBound}),
$$
\| h(W^{\text{move}}_{kn}) - W^{\text{move}}_{kn} \|_2 \leq l_{kn} + 2 \leq L(H_{kn})/2 + 2 \leq \kappa
$$
since $\kappa\geq 6$. By (\ref{InductionLHn}), we deduce that $L(H_{kn+1})$ is also smaller than $\kappa$. By induction, the same holds for $L(H_{k(n+1)})$.\\

From now on, let us assume that $L(H_{kn})\geq \kappa$. Two cases must be distinguished.\\
If none of the semi-balls
$$
B^{+}(W^{\text{move}}_{kn+j} , \| h(W^{\text{move}}_{kn+j}) - W^{\text{move}}_{kn+j} \|_{2}) , \, j=0,\ldots,k-1 ,
$$
generated between the $(kn+1)$-th and the $k(n+1)$-th steps exceed the horizontal line $\{\x : \x(2)=W^{\text{move}}_{kn}(2)+L(H_{kn})\}$ then
$$
L(H_{k(n+1)}) \leq L(H_{kn}) - \big( W^{\text{move}}_{k(n+1)}(2) - W^{\text{move}}_{kn}(2) \big)\leq L(H_{kn})-1,
$$
by Lemma \ref{lem:E_nW_n}.\\
Otherwise, we necessarily have
\begin{equation}
\label{etape1}
L(H_{k(n+1)}) \leq \max_{0\leq j\leq k-1} \| h(W^{\text{move}}_{kn+j}) - W^{\text{move}}_{kn+j} \|_{2} ~.
\end{equation}
Combining $I_{kn+1}=1$ and (\ref{lnBound}), we get
$$
\| h(W^{\text{move}}_{kn}) - W^{\text{move}}_{kn} \|_2 \leq l_{kn} + 2 \leq L(H_{kn})/2 + 2 \leq L(H_{kn}) - 1
$$
whenever $L(H_{kn})\geq \kappa \geq 6$. Here is the justification for the choice of $\kappa\geq 6$. This and \eqref{InductionLHn} imply that $L(H_{kn+1})\leq L(H_{kn})$. Then, $I_{kn+2}=1$ and 
$$
\| h(W^{\text{move}}_{kn+1}) - W^{\text{move}}_{kn+1} \|_2 \leq L(H_{kn+1})/2 + 2 \leq L(H_{kn})/2 + 2 \leq L(H_{kn}) - 1
$$
since $L(H_{kn})\geq \kappa$. By induction, we have on $\{\prod_{j=1}^{k} I_{kn+j}=1\}$ that, for all $j\in \{0,\dots k-1\}$, $L(H_{kn+j})\leq L(H_{kn})$ and $\| h(W^{\text{move}}_{kn+j}) - W^{\text{move}}_{kn+j} \|_2 \leq L(H_{kn})-1$. Hence:
$$
\max_{0\leq j\leq k-1} \| h(W^{\text{move}}_{kn+j}) - W^{\text{move}}_{kn+j} \|_{2} \leq L(H_{kn}) - 1 ~,
$$which by \eqref{etape1} concludes the proof.
\end{proof}

\subsection{Synthesis}

Now we are going to define inductively a discrete time integer-valued process $\{M_n = M_n(\u_1, \ldots, \u_k) : n \geq 1\}$ whose role is to dominate the height of the history set. Set $M_0:=\kappa{+1}$, where $\kappa$ is the integer introduced in \eqref{def:tau_j}. For $n \geq 0$ given $M_n$, we define $M_{n+1}$ as follows:
\begin{equation}
\label{defM_n}
M_{n+1} := \begin{cases}
\max \{M_n - 1, \kappa\} & \text{ if } \prod_{j = 1}^{(\kappa+1)k} I_{(\kappa+1)kn + j} = 1,\\
\max \{M_n ,  \sum_{j=1}^{(\kappa+1)} X_{(\kappa+1)n+j} , (\kappa+1)\} & \text{ otherwise,}\\
\end{cases}
\end{equation}
where the r.v.'s $X_{n}$ are defined in (\ref{def:X1}). The sequence $\{M_n : n \geq 1\}$ is a Markov chain with state space $\{\kappa,\kappa + 1, \kappa + 2, \cdots \}$. Let $\tau^{M}$ be the hitting time of $\kappa$:
\begin{equation}
\label{def:sigmaM}
\tau^{M} := \inf \{n \geq 1 : M_n = \kappa \} ~.
\end{equation}

{In order to hit the level $\kappa$ starting from $M_0=\kappa+1$, the Markov chain $(M_n)$ has to decrease by $1$ at least once, i.e. the event
$$
\prod_{j = 1}^{(\kappa+1)k} I_{(\kappa+1)kn + j} = 1
$$
has to occur at least once before $\tau^{M}$. This happens with probability $p_0^{(\kappa+1)k}$. By Lemma \ref{lem:E_nW_n}, considering blocks of $(\kappa+1)k$ steps guarantees a progress of at least $\kappa+1$ for the ordinate of the moving vertex when this event happens, which is the last condition required for good steps. See \eqref{def:tau_j}.}

The r.v. $M_n$ is built in order to dominate the height $L(H_{(\kappa+1)kn})$.

\begin{lemma}
\label{lem:MngeqLn}
The inequality $L(H_{(\kappa+1)kn}) \leq M_{n}$ holds a.s. for all $n \geq 0$. As a consequence, the random time $\tau_1$ defined in (\ref{def:tau_j}) satisfies a.s. $\tau_1 \leq (\kappa+1)k \tau^{M}$.
\end{lemma}

\begin{proof}
Let us first prove $L(H_{(\kappa+1)kn}) \leq M_{n}$ by induction. Assumption (\textbf{H1}) says $L(H_{0}) \leq \kappa \leq M_0$. Assume that $L(H_{(\kappa+1)kn})\leq M_{n}$ for some $n\geq 0$. Either $\prod_{j=1}^{(\kappa+1)k} I_{(\kappa+1)kn+j}=1$, and then by Lemma \ref{lem:EnLn},
$$
L(H_{(\kappa+1)kn+k}) \leq \max \{L(H_{(\kappa+1)kn}) - 1, \kappa \} \leq \max \{M_{n} - 1, \kappa \} = M_{n+1} ~.
$$
We can easily iterate the argument, still applying Lemma \ref{lem:EnLn}:
$$
L(H_{(\kappa+1)kn+2k}) \leq \max \{L(H_{(\kappa+1)kn+k}) - 1, \kappa\} \leq M_{n+1}
$$
to finally get $L(H_{(\kappa+1)k(n+1)})\leq M_{n+1}$. \\
{Or, $\prod_{j=1}^{(\kappa+1)k} I_{(\kappa+1)kn+j}=0$ and then, }by Lemma \ref{lem:IncreaseLn} (i),
\begin{eqnarray*}
L(H_{(\kappa+1)k(n+1)}) & \leq & \max \{ L(H_{(\kappa+1)kn+\kappa k}) , X_{(\kappa+1)n+(\kappa+1)}\} \\
 & \leq & \max \{ L(H_{(\kappa+1)kn+(\kappa-1)k}) , X_{(\kappa+1)n+\kappa}+X_{(\kappa+1)n+(\kappa+1)}\} \\
 & \leq & \max \{ L(H_{(\kappa+1)kn}) , X_{(\kappa+1)n+1}+\ldots+X_{(\kappa+1)n+(\kappa+1)}\} \\
 & \leq & M_{n+1} ~.
\end{eqnarray*}
This completes the proof by induction. As a consequence, at the $(\kappa+1)k\tau^{M}$-th step the height $L(H_{(\kappa+1)k\tau^{M}})$ is smaller than $\kappa$. Moreover, by construction of the chain $(M_n)$, there exists an integer $m'<\tau^{M}$ such that $\prod_{j=1}^{(\kappa+1)k} I_{(\kappa+1)km'+j}=1$. Lemma \ref{lem:E_nW_n}, applied $(\kappa+1)$ times, implies
$$
W^{\text{move}}_{(\kappa+1)k\tau^{M}}(2) \geq W^{\text{move}}_{(\kappa+1)k(m'+1)}(2) \geq W^{\text{move}}_{(\kappa+1)km'}(2)+(\kappa+1) \geq W^{\text{move}}_{0}(2)+(\kappa+1) ~.
$$
We finally get $\tau_1 \leq (\kappa+1)k\tau^{M}$.
\end{proof}

A proof similar to the one of Lemma 2.6 in \cite{RSS15} leads to the next result.

\begin{lemma}
\label{lem:M_tail}
For any $n\in\N$ we have
$$
\P( \tau^{M} > n) \leq C_0 e^{-C_1 n} ~.
$$
\end{lemma}

\begin{proof}
Thanks to \cite[Proposition 5.5, Chapter 1]{A03} (see also \cite[Chap. 15]{meyntweediebook}) it is enough to show that there exists a function $f:\N \to \R_{+}$, an integer $n_0$ and real numbers $r>1$, $\delta > 0$ such that
\begin{itemize}
\item[$\bullet$] $f(l)>\delta$ for any $l\in\N$;
\item[$\bullet$] $\E[f(M_{1})| M_0 = l] < \infty$ for any $l\leq n_0$;
\item[$\bullet$] and $\E[f(M_{1})| M_0 = l] \leq f(l)/r$ for any $l>n_0$.
\end{itemize}
Indeed, this implies the existence of some $r>1$ such that $\E(r^{\tau^M(n_0)}|M_0 = n_0)<\infty$ where $\tau^M(n_0):=\inf\{n \geq 1: M_n \in [0,n_0]\}$. In other words, the hitting time $\tau^M(n_0)$ admits an exponential moment. Finally, Lemma \ref{lem:M_tail} follows from the fact that starting from any $l\leq n_0$, $p_0^{(\kappa+1)kn_0}>0$ gives a lower bound for the probability that the chain hits the state $\kappa$ within the next $n_0$ steps where $p_0$ is defined in Lemma \ref{lem:ProbIn}.

We take $f:\N\to\R_{+}$ to be $f(l):=e^{\alpha l}$ where $\alpha>0$ is small enough so that $\E[e^{\alpha Y}]<\infty$ with $Y:=\max\{X_{1}+\ldots+X_{(\kappa+1)},(\kappa+1)\}$. This is possible by Lemma \ref{lem:IncreaseLn}. So, for any $l\leq n_0$,
$$
\E[f(M_{1})| M_0 = l] \leq e^{\alpha n_0} \E[e^{\alpha(M_{1}-M_0)}| M_0 = l] \leq e^{\alpha n_0} \E[e^{\alpha Y}] < \infty ~.
$$
Then, pick $r>1$ such that $e^{-\alpha}p_0^{(\kappa+1)k} + (1-p_0^{(\kappa+1)k}) < 1/r$. Using (\ref{defM_n}), we can write for $l\geq n_0>\kappa$:
\begin{multline*}
\E[e^{\alpha(M_{1} - M_{0})}|M_0 = l] \\
\begin{aligned}
 = & \E[e^{\alpha(M_{0}-1 - M_{0})} \ind_{\prod_{j=1}^{(\kappa+1)k} I_{j}=1} |M_0 = l]+
\E[e^{\alpha(\max(M_0,Y) - M_{0})} \ind_{\prod_{j=1}^{(\kappa+1)k} I_{j}=0} |M_0 = l]\\
 \leq & e^{-\alpha} p_0^{(\kappa+1)k} + (1 - p_0^{(\kappa+1)k}) + e^{-\alpha l}\E[{\mathbf 1}_{\{Y>l\}} e^{\alpha Y}] \\
 < & 1/r' ,
\end{aligned}
\end{multline*}
for $n_0$ large enough and $r'\in (1,r)$. This completes the proof.
\end{proof}

We are now able to prove Proposition \ref{propo:Tau_tail}. As suggested by Lemma \ref{lem:MngeqLn}, the dominating r.v. $T$ occurring in Proposition \ref{propo:Tau_tail} is given by $(\kappa+1)k\tau^M$.

\begin{proof}[Proof of Prop. \ref{propo:Tau_tail}]
Let us first start with the case $j=0$. Lemmas \ref{lem:MngeqLn} and \ref{lem:M_tail} ensure that
{$$
\P \big(\tau_{1}-\tau_0 \geq n \ |\ \mathcal{F}_{\tau_0} \big) = \P \big( \tau_{1} \geq n \ |\ \mathcal{F}_{0} \big) \leq \P \big( (\kappa+1) k \tau^M \geq n \big) \leq C_0 e^{-C_1 n}
$$}
for suitable positive constants $C_0,C_1$. So, conditional to $\mathcal{F}_{0}$ (which contains the information given by the initial configuration $(\u_1,\ldots,\u_k,H_0)$), $\tau_{1}-\tau_0$ is stochastically dominated by {$T:=(\kappa+1) k \tau^M$}.

Now, let us prove the result for $j=1$; we will deduce the result for any $j$ similarly. The idea consists in working conditionally on the $\sigma$-algebra $\FF_{\tau_1}$ and applying the previous strategy (i.e. Lemmas \ref{lem:UnexploredCone} to \ref{lem:M_tail}) to the ``new starting configuration'' $(g_{\tau_{1}}(\u_1),\ldots, g_{\tau_{1}}(\u_k), H_{\tau_{1}})$. {First remark that the elements $g_{\tau_{1}}(\u_1),\ldots, g_{\tau_{1}}(\u_k)$ as well as the history set $H_{\tau_{1}}$ are measurable w.r.t. $\FF_{\tau_1}$: they are deterministic conditionally to $\FF_{\tau_1}$. Assumptions (\textbf{H1}) and (\textbf{H2}) are also clearly satisfied by the definition of the hitting time $\tau_1$ and by the construction of the joint exploration process (see Lemma \ref{lem:ElemProperties}).}


Then, from step $\tau_1$ onwards, we can apply the strategy developed throughout this section, and dominate the height of the history set by a new Markov chain, say $(M'_n)$, built as in (\ref{defM_n}) and distributed as $(M_n)$. Hence, conditionally to $\FF_{\tau_1}$, the increment $\tau_2-\tau_1$ is stochastically dominated by a r.v. {$\tau^{M'}$} where $\tau^{M'}$ is the hitting time of $\kappa$ for the chain $(M'_n)$. Of course, $\tau^{M'}$ and $\tau^{M}$ are identically distributed.
\end{proof}


\section{Renewal steps}
\label{sect:Renewal}

From now on, we come back to the original joint exploration process built with a single PPP $\NN$ and starting from an initial configuration $(\u_1,\ldots,\u_k,H_0)$ satisfying (\textbf{H1}) and (\textbf{H2}).

Before describing the mathematical details of renewal steps, we provide
the general idea. Let us first introduce some notations. Consider the $j$-th good step with $j\geq 1$.
Let us set for any $1\leq i\leq k$,
\begin{equation*}
g_{\tau_j}^{\downarrow}(\u_i) := (g_{\tau_j}(\u_i)(1) ,
W^{\text{move}}_{\tau_j}(2)) \; \mbox{ and } \; g_{\tau_j}^{\uparrow}(\u_i)
:= (g_{\tau_j}(\u_i)(1) , W^{\text{move}}_{\tau_j}(2) + \kappa)
\end{equation*}
respectively the projections of $g_{\tau_j}(\u_i)$ onto the horizontal
axes with ordinates $W^{\text{move}}_{\tau_j}(2)$ and
$W^{\text{move}}_{\tau_j}(2)+\kappa$. The vertices $g_{\tau_j}(\u_i)$'s lie in the horizontal strip delimited by these two axes.

Let us now define the following `renewal' events.

\begin{definition}
\label{def:renewalStep}
For $j \geq 1$, the $j$-th good step is called a renewal step, if the
following event, henceforth called the renewal event, $A_j=A_j(\u_1, \ldots, \u_k)$
occurs:
\begin{equation}
\label{def:EventAj}
\begin{split}
A_{j}  := \bigcap_{i=1}^k \Big\{ & \Card \bigl(
\bigl( B^+(g_{\tau_j}^{\downarrow}(\u_i), \kappa+1) \setminus\!
H_{\tau_j} \bigr) \cap \NN  \bigr) =  1\\
 & \mbox{and }\, \Card \bigl( B^+(g_{\tau_j}^{\uparrow}(\u_i), 1) \cap \NN \bigr) = 1 \Big\}.\end{split}
\end{equation}
\end{definition}

Note that the definition of $A_j $ is associated with the $j$-th good step and then, only good steps can be renewal steps.

The event $A_j$ asserts that for any index $i$, the semi-ball $B^+(g_{\tau_j}^{\downarrow}(\u_i), \kappa+1)$ contains only one Poisson point which is actually included in $B^+(g_{\tau_j}^{\uparrow}(\u_i), 1)$. Besides, the semi-balls $B^+(g_{\tau_j}^{\uparrow}(\u_i),1)$'s may have non-empty intersections and the corresponding Poisson points (involved by $A_j$) may not necessarily be distinct in this case.



Thus we set $\gamma_0=0$ and for $\ell \geq 1$, let $\gamma_\ell=\gamma_\ell(\u_1,\dots,\u_k)$ be the number of good steps required for the $\ell$-th renewal step:
$$
\gamma_\ell := \inf \big\{ j > \gamma_{\ell - 1} : \mbox{ the event } A_j \mbox{ occurs} \big\} ~.
$$
Moreover
\begin{equation}
\label{def:betaell}
\beta_\ell := \tau_{\gamma_\ell}
\end{equation}
denotes the total number of steps required for $\ell$-th renewal step. {In the sequel, we will break down the DSF paths starting from $\u_1,\ldots,\u_k$ according to these renewal steps $\beta_\ell$, $\ell \geq 1$, in order to obtain the searched decay of the tail distribution for the coalescence time of two paths (Theorem \ref{thm:CoalescingTimetail} of Section \ref{sec:Tail}) and the convergence of scaled DSF paths to coalescing Brownian motions (Section \ref{sect:(I1)}).}

In this section, we first show that the renewal times are a.s. finite and establish an exponential decay for the tail distribution of the number of steps between two consecutive renewal steps (Section \ref{sect:TailNbSteps}) as well as the size of the region explored by the DSF paths between two consecutive renewal steps (Section \ref{sect:Widths}). Though we define the renewal events for $k$ paths, with a general integer $k$, only two cases will be important for us: the single path case in Section \ref{sect:OnePath}, in which the role of the renewal event $A_j$ is explained, and the two paths case in Section \ref{sect:TwoPaths}.\\

Because of the event $A_j$ concerns the PPP $\NN$ inside $\mathbb{H}^{+}(W_{\tau_{j}}^{\text{move}}(2))\setminus H_{\tau_j}$, it does not belong to the $\sigma$-field $\mathcal{F}_{\tau_j}$. {Hence, for any $\ell\geq 1$, the r.v. $\gamma_\ell$ is \textit{not} a $(\mathcal{F}_{\tau_j})$-stopping time. This is the reason why we enrich the $\sigma$-field $\mathcal{F}_{\tau_j}$ by including the events $A_1, A_2,\dots, A_j$. Let ${\cal S}_0=\FF_0$ and, for $j\geq 1$, we consider the enhanced $\sigma$-field:
\begin{equation}
\label{def:FiltrationGl}
{\cal S}_{j} := \sigma \big( {\cal F}_{\tau_j}, A_1, A_2, \dots , A_j \big) ~.
\end{equation}}
For any $\ell \geq 0$, the r.v. $\gamma_\ell$ is then a stopping time w.r.t. the filtration $\{{\cal S}_j : j \geq 0\}$. Next we introduce the filtration $\{{\cal G}_\ell : \ell\geq 0\}$ where
\begin{equation}
\label{Filtration}
{\cal G}_\ell  := {\cal S}_{\gamma_\ell} ~.
\end{equation}

\begin{lemma}
The sequence of r.v.'s $\{\beta_\ell : \ell \geq 0\}$ denoting the total number of steps required for $\ell$-th renewal step is adapted to the filtration $\{{\cal G}_\ell : \ell \geq 0\}$.
\end{lemma}

\subsection{Tail distribution of the number of steps between consecutive renewal steps (for $k$ paths)}
\label{sect:TailNbSteps}

We first show that the probability of a renewal event occurring at a good step is bounded strictly away from $0$ as well as $1$, conditionally on the previous steps. This will be used to show that the renewal steps must occur and that a geometric number of good steps at most are required to reach a renewal step.
The main result of the current section is:

\begin{proposition}
\label{propo:GoodRegStepTail}
There exist positive constants $C_0, C_1$ such that for any $\ell\geq 0$, for any $n\geq 1$,
\begin{equation}
\label{Beta_tail}
\P \big( \beta_{\ell+1} - \beta_{\ell} \geq n \ |\ {\cal G}_{\ell} \big) \leq C_0 e^{-C_1 n} ~.
\end{equation}
\end{proposition}

In order to prove the above proposition, we need the following result.

\begin{lemma}
\label{lem:AjProbBound}
{There exist $0<p_1 \leq p_2<1$ depending only on $k, \kappa, \lambda$ such that, for any $j\geq 1$, the following holds:
\begin{equation}
\label{A_jEvent_Bound}
\P \big( A_j \ |\ {\cal F}_{\tau_j}, \mathbf{1}_{A_1},\dots, \mathbf{1}_{A_{j-1}} \big) = \P \big( A_j \ |\ {\cal F}_{\tau_j} \big) \in [ p_1 , p_2 ] ~.
\end{equation}}
\end{lemma}

\begin{proof}
We first show that (\ref{A_jEvent_Bound}) holds for $j=1$ (where we set $\mathbf{1}_{A_0}=1$). Observe that $ {\cal F}_{\tau_1}$ does not contain any information about the PPP in the half plane $\mathbb{H}^+(W^{\text{move}}_{\tau_1}(2) + \kappa + 1)$. We can choose
$$
p_2 := \P( \Card (B^+(g_{\tau_1}^{\uparrow}(\u_1), 1) \cap \NN) = 1 ) = \frac{\lambda \pi}{2} \exp (-\lambda \pi/2)
$$
as an upper bound strictly smaller than $1$.

For a single path, the lower bound is straight forward. In fact the
event $ \{  \NN \cap \bigl( B^+(g_{\tau_1}^{\downarrow}(\u_1), \kappa+1)
\setminus\! B^+(g_{\tau_1}^{\uparrow}(\u_1), 1) \bigr) = \emptyset,
\Card  (\NN \cap B^+(g_{\tau_1}^{\uparrow}(\u_1), 1)) = 1 \} $
implies that renewal occurs and hence provides a lower bound.
However, finding a strictly positive lower bound for more than one path requires more work
when the paths are close and when the regions overlap.

We assume here without loss of generality that all the $g_{\tau_1}(\u_i)$'s are different.
Then the same holds for $g_{\tau_1}^{\downarrow}(\u_i)$'s and
$g_{\tau_1}^{\uparrow}(\u_i)$'s a.s.  Let us set
$$
F := \bigcup_{1\leq i\leq k} \bigl( B^+(g_{\tau_1}^{\downarrow}(\u_i),
\kappa+1) \setminus\! B^+(g_{\tau_1}^{\uparrow}(\u_i), 1) \bigr) ~.
$$

The probability that $F \cap \NN$ is empty is easily bounded from below
by a positive constant depending on $k,\kappa$ and the intensity
$\lambda$ of the PPP. But it requires tedious geometric computations
to show that each $B^+(g_{\tau_1}^{\uparrow}(\u_i), 1)$ contains
exactly one Poisson point with positive probability, because
$g_{\tau_j}(u_i)$'s can be very close to each other. So, we
give the main arguments and skip the details.

Let $\epsilon>0$. If
\begin{equation}
\label{Condition:p1}
\min_{i\not= i'} | g_{\tau_1}^{\uparrow}(\u_i)(1) - g_{\tau_1}^{\uparrow}(\u_{i'})(1) | \geq \epsilon
\end{equation}
then it is possible to exhibit deterministic (conditional on
$\mathcal{F}_{\tau_1}$) regions $\Lambda_{1},\ldots,\Lambda_k$ such
that their areas are equal to some constant $c(\epsilon)>0$, they do not overlap with $F$ and each $\Lambda_{i}$ satisfies
$$
\Lambda_{i} \subset B^+(g_{\tau_1}^{\uparrow}(\u_i), 1) \setminus\! \Big( \bigcup_{i'\not= i} B^+(g_{\tau_1}^{\uparrow}(\u_{i'}, 1) \Big) ~.
$$
From there, putting exactly one Poisson point inside each $\Lambda_{i}$, it is not difficult to conclude. When (\ref{Condition:p1}) is no longer true, we split $\{1,\ldots,k\}$ into disjoint subsets in which consecutive elements $g_{\tau_1}^{\uparrow}(\u_i)$ and $g_{\tau_1}^{\uparrow}(\u_{i+1})$ are at a distance smaller than $\epsilon$. We treat these subsets separately. Let us consider the first of them, say $\{g_{\tau_1}^{\uparrow}(\u_1),\ldots,g_{\tau_j}^{\uparrow}(\u_{s})\}$ with $1\leq s\leq k$ and $|g_{\tau_1}^{\uparrow}(\u_i)(1)-g_{\tau_1}^{\uparrow}(\u_{i+1})(1)| < \epsilon$ for $1\leq i\leq s-1$. For $\epsilon>0$ small enough (depending on $\kappa$ and $k$), the semi-balls $B^+(g_{\tau_1}^{\uparrow}(\u_i), 1)$, $1\leq i\leq s (\leq k)$, have a non-empty intersection. Hence, we can exhibit a deterministic region $\Lambda^{(1)}$ with area $c'(\epsilon)>0$ such that
$$
\Lambda^{(1)} \subset \bigcap_{1\leq i\leq s} B^+(g_{\tau_1}^{\uparrow}(\u_i), 1) \;
\mbox{ and } \; \Lambda^{(1)} \cap \Big( F \cup
\Big( \bigcup_{i>s} B^+(g_{\tau_1}^{\uparrow}(\u_i), 1) \Big) \Big) = \emptyset ~.
$$
Now we repeat the argument with the remaining set $\{g_{\tau_1}^{\uparrow}(\u_{s+1}) ,
\dotsc,g_{\tau_1}^{\uparrow}(\u_{k})\}$ (possibly empty). We finally exhibit
disjoint subsets $\Lambda^{(1)},\ldots,\Lambda^{(n)}$ with good properties.
Putting exactly one Poisson point inside each of them and no point in $F$,
we may obtain the lower bound $p_1$.\\

Next consider Equation (\ref{A_jEvent_Bound}) for $j = 2$. Recall that the random
variable $\mathbf{1}_{A_1}$ depends only on the PPP in the half plane
$\mathbb{H}^-(W^{\text{move}}_{\tau_1}(2) + \kappa +1)$, where $\mathbb{H}^-(l):=\{\x\in\R^2 : \x(2)\leq l\}$, and where by
definition we have $W^{\text{move}}_{\tau_2}(2) - W^{\text{move}}_{\tau_1}(2)
> \kappa +1$.  Given ${\cal F}_{\tau_2}$, since the subsequent steps as well as the event $A_2$ depend
only on the PPP in the half plane $\mathbb{H}^+(W^{\text{move}}_{\tau_2}(2))$ we have
$$
\P \big( A_2 \ |\ {\cal F}_{\tau_2}, \mathbf{1}_{A_1} \big)
= \P \big( A_2\ |\ {\cal F}_{\tau_2} \big),
$$
and then the proof follows using the same argument as in the case of $j = 1$. Finally for general $j \geq 1$, the proof follows by method of induction.
\end{proof}

Now we are ready to prove Proposition \ref{propo:GoodRegStepTail}.

\begin{proof}
We work conditionally on ${\cal G}_{\ell}$, for $\ell\geq 0$. Let us first show that
$$
\P(\gamma_{\ell + 1} - \gamma_\ell > j \ |\ {\cal G}_\ell ) \leq \P(G > j)
$$
where $G$ is a geometric r.v. with success probability $p_1$. In other words, the r.v. $\gamma_{\ell + 1}-\gamma_\ell$ is stochastically dominated by $G$. First we prove it for $\ell=0$. The argument for general $\ell\geq 0$ is the same. In what follows, $A^c$ denotes the complement event of $A$.
\begin{eqnarray*}
\P(\gamma_1 - \gamma_0 > j  \mid  {\cal G}_0) & = & \P(A_1^c,\, A_2^c, \dotsc, A_j^c  \mid  {\cal G}_0) \\
& = & \E \bigl[ \prod_{i=1}^{j-1} \mathbf{1}_{A^c_i} \E [\mathbf{1}_{A^c_j} \mid {\cal S}_{j-1}, {\cal F}_{\tau_j} ] \mid  {\cal G}_0 \bigr] \\
& = & \E\bigl[ \prod_{i=1}^{j-1} \mathbf{1}_{A^c_i} \E [\mathbf{1}_{A^c_j} \mid {\cal F}_{\tau_j}] \mid {\cal G}_0\bigr] \\
& \leq & (1- p_1) \E\bigl[ \prod_{i=1}^{j-1} \mathbf{1}_{A^c_i} \mid {\cal G}_0\bigr] \\
& \leq & (1- p_1)^j ~,
\end{eqnarray*}
by Lemma \ref{lem:AjProbBound} and a direct recursion.
Next we show that given the $\sigma$-field ${\cal S}_j$, the difference $\tau_{j+1}-\tau_j$ still decays exponentially fast. Observe that the events $A_{1},\dots,A_{j-1}$ depend only on the Poisson points in $\mathbb{H}^{-}(W^{\text{move}}_{\tau_{j-1}} (2)+\kappa+1)$ while $\tau_{j+1}-\tau_{j}$ depends on the points in $\mathbb{H}^{+}(W^{\text{move}}_{\tau_{j}} (2))$ and hence are independent given ${\cal F}_{\tau_j}$.
{\begin{multline}
\P(\tau_{j+1} - \tau_j \geq n \mid {\cal S}_j) = \P(\tau_{j+1} - \tau_j \geq n \mid \sigma({\cal F}_{\tau_j}, A_{j-1}) ) \\
\begin{aligned}
& = \frac{\P(\tau_{j+1} - \tau_j \geq n, A_{j}^c \mid {\cal F}_{\tau_j})}{\P(A_{j}^c \mid {\cal F}_{\tau_j})} \mathbf{1}_{A_j^c} + \frac{\P(\tau_{j+1} - \tau_j \geq n, A_j \mid {\cal F}_{\tau_j})}{\P(A_{j} \mid {\cal F}_{\tau_j})} \mathbf{1}_{A_j} \\
& \leq \P(\tau_{j+1} - \tau_j \geq n \mid {\cal F}_{\tau_j}) \left( \frac{1}{\P(A_{j}^c \mid {\cal F}_{\tau_j})} + \frac{1}{\P(A_{j} \mid {\cal F}_{\tau_j})} \right) ~.
\end{aligned}
\label{etape2}
\end{multline}}
So, using Proposition \ref{propo:Tau_tail} and Lemma \ref{lem:AjProbBound} we obtain the expected decay:
$$
\P(\tau_{j+1} - \tau_j \geq n \mid {\cal S}_j) \leq C_0 \exp{(- C_1 n)} ~,
$$
where $C_0,C_1 > 0$ are constants only depending on $k,\kappa,\lambda$.

{Therefore, we can construct a random variable $T$ satisfying $\P(T\geq n)\leq C_0 \exp {(- C_1 n)}$ and stochastically dominating the difference $\tau_{j+1}-\tau_j \mid {\cal S}_j$. Further, we can stochastically bound, for any $\ell,m$, the sum
$$
\sum_{j=0}^{m-1} \tau_{\gamma_{\ell}+j+1} - \tau_{\gamma_{\ell}+j}
$$
conditionally to ${\cal G}_{\ell}$ by $T_1+\ldots+T_m$ where the $T_j$'s are i.i.d. copies of the r.v. $T$ defined above.} Now, let $\vartheta>0$ small enough so that $\E( e^{\vartheta T})<\infty$. Then, for any constant $c>0$, we write
\begin{multline*}
 \P ( \beta_{\ell+1} - \beta_{\ell} \geq n \mid {\cal G}_{\ell} ) \\
\begin{aligned}
&\leq
\P \Big( \sum_{j= 0}^{\lfloor cn \rfloor }
\tau_{ \gamma_{\ell} + j+1} - \tau_{\gamma_{\ell} +j} \geq n \mid {\cal G}_{\ell}  \Big)
+ \P \big( G > \lfloor cn \rfloor
\mid {\cal G}_{\ell}  \big)  \\
& \leq  \P \Big( \sum_{j=0}^{\lfloor cn \rfloor}  T_ j \geq n ) + (1-p_1)^{\lfloor cn \rfloor} \\
& \leq e^{-\vartheta n} \E( e^{\vartheta T})^{\lfloor cn \rfloor}+ (1-p_1)^{\lfloor cn \rfloor} ~.
\end{aligned}
\end{multline*}
This completes the proof of \eqref{Beta_tail} by choosing $c=c(\vartheta)$ sufficiently small.
\end{proof}

\subsection{Size of the renewal blocks (for $k$ paths)}
\label{sect:Widths}

{Let $\u_1,\ldots,\u_k$ be the starting points and fix $\ell\geq 0$. In this section, our goal is to exhibit random rectangles containing the regions explored by the $k$ trajectories of the joint exploration process $\{(g_n(\u_1),\ldots,g_n(\u_k)) : n\geq 0\}$  between the $\ell$-th and the $(\ell+1)$-th renewal steps. We define
\begin{equation}
\label{def:WidthJtRegenration}
W_{\ell+1} = W_{\ell+1}(\u_1,\cdots,\u_k) := \sum^{\beta_{\ell+1}-1}_{m=\beta_{\ell}}
\| W^{\text{move}}_{m} - h(W^{\text{move}}_{m}) \|_2 ~.
\end{equation}}
By construction the r.v. $W_{\ell+1}$ is such that, for any $1\leq i\leq k$, the random set
\begin{equation}
\label{EnsFonda}
\bigcup_{m=\beta_{\ell}}^{\beta_{\ell+1}-1} B^+ \left( g_m(\u_i) , \| g_m(\u_i) - g_{m+1}(\u_i) \|_2 \right) ~,
\end{equation}
where the union is made up with all the semi-balls created by the path starting at $\u_i$ between the $\ell$-th and the $(\ell+1)$-th renewal steps, is included in $g_{\beta_{\ell}}(\u_i)+[-W_{\ell+1},W_{\ell+1}] \times [0,W_{\ell+1}]$. This rectangle is called a \textit{renewal block}. This is the reason why $W_{\ell+1}$ is termed as the \textit{size} of these renewal blocks.

It is important to remark that the $k$ trajectories between the $\ell$-th and the $(\ell+1)$-th renewal steps depend only on the Poisson points inside the random set (\ref{EnsFonda}). Hence, these $k$ trajectories are not altered by any change of the
PPP $\NN$ outside the renewal blocks $g_{\beta_{\ell}}(\u_i)+[-W_{\ell+1},W_{\ell+1}]\times [0,W_{\ell+1}]$'s. This suggests that  two paths far from each other evolve almost independently; such argument will be used in the proof of Theorem \ref{thm:CoalescingTimetail}.

\begin{proposition}
\label{prop:IIDMaxSubExpTail}
There exist constants $C_0,C_1>0$ such that for any $\ell\geq 0$, for all $n \geq 1$,
\begin{equation}
\label{eq:JtRegWidthExpTail}
\P( W_{\ell+1} \geq n \ |\ {\cal G}_{\ell} ) \leq C_0 e^{-C_1 n^{1/2}} ~.
\end{equation}
\end{proposition}

This shows that the conditional distribution of $W_{\ell+1} \ |\ {\cal G}_{\ell}$ is dominated by a random variable $ {\cal W} $ with sub-exponential tail given by $C_0 e^{-C_1 n^{1/2}}$. With more works, it could be possible to show that the distribution of ${\cal W}$ admits an exponentially decaying tail, but (\ref{eq:JtRegWidthExpTail}) will be sufficient for our purpose.

\begin{proof}
Let $\ell\geq 0$. We will work conditionally on ${\cal G}_{\ell}$. Let us recall the definition of the random variables $\{\zeta_{m+1} : m\geq 0\}$ in (\ref{zeta(1)_n}) which are i.i.d. with exponentially decaying tails. Let us now show by recursion that, for any $m\geq 0$,
\begin{equation}
\label{inductionWl}
\| W^{\text{move}}_{\beta_{\ell}+m} - h(W^{\text{move}}_{\beta_{\ell}+m}) \|_2 \leq \max_{0\leq n\leq m}
(\lfloor 2\zeta_{n+1}\rfloor + 1) + \kappa + 1 ~.
\end{equation}
First, \eqref{inductionWl} holds for $m=0$ since, on the renewal event,
\begin{equation*}
\| g_{\beta_{\ell}} ( \u_i) - h (g_{\beta_{\ell}}(\u_i)) \|_2 \leq \kappa + 1 ~.
\end{equation*}
Thus, assume that \eqref{inductionWl} holds for a given integer $m$. If
$$
\|W^{\text{move}}_{\beta_{\ell}+m+1} - h(W^{\text{move}}_{\beta_{\ell}+m+1}) \|_2 \leq \max_{0\leq n\leq m} \|W^{\text{move}}_{\beta_{\ell}+n} - h(W^{\text{move}}_{\beta_{\ell}+n})\|_2 ~,
$$
then \eqref{inductionWl} is obviously satisfied for $m+1$. Otherwise, we have
$$
\|W^{\text{move}}_{\beta_{\ell}+m+1} - h(W^{\text{move}}_{\beta_{\ell}+m+1})\|_2 > \max_{0\leq n\leq m} \|W^{\text{move}}_{\beta_{\ell}+n} - h(W^{\text{move}}_{\beta_{\ell}+n})\|_2 \geq L(H_{\beta_{\ell}+m}) ~,
$$
which forces, via similar arguments to those developed in the proof of Lemma \ref{lem:IncreaseLn}, part $(i)$, that
$$
\|W^{\text{move}}_{\beta_{\ell}+m+1} - h(W^{\text{move}}_{\beta_{\ell}+m+1})\|_2 \leq \lfloor 2\zeta_{m+2} \rfloor + 1 ~.
$$
This concludes the proof of \eqref{inductionWl}.

The fact that the r.v.'s $\{\max_{n\leq m}\lfloor 2\zeta_{n+1}\rfloor + \kappa + 2 : m\in \Z_+\}$ are not identically distributed prevents us from immediately obtaining exponential decay for the r.v. $ W_{\ell+1} \mid {\cal G}_{\ell}$. So we content ourself with the following computation leading to sub-exponential decay. First,
\begin{equation}
\label{W-FinalComput}
\begin{split}
\P(W_{\ell+1} \geq n \mid {\cal G}_{\ell}) & \leq \P \left(\sum_{m=0}^{\lfloor n^{1/2}\rfloor}
\max_{l\leq m}(\lfloor 2\zeta_{l+1}\rfloor + \kappa + 2) \geq n \right) \\
& + \P( \beta_{\ell+1} - \beta_{\ell} \geq n^{1/2} \mid {\cal G}_{\ell}) ~.
\end{split}
\end{equation}
The second term of the l.h.s. of (\ref{W-FinalComput}) is bounded from above
by $C_0 e^{-C_1 n^{1/2}}$ thanks to Proposition \ref{propo:GoodRegStepTail}
while the first one is treated as follows:
\begin{align*}
& \P \left( \sum_{m=0}^{\lfloor n^{1/2}\rfloor} \max_{l\leq m}(\lfloor
2\zeta_{l+1}\rfloor + \kappa  + 2) \geq n \right)  \\
& \leq  \P \left( \bigcup_{m\leq \lfloor n^{1/2}\rfloor}
\big\{ \lfloor 2\zeta_{m+1}\rfloor  + \kappa  + 2 \geq n^{1/2}-1 \big\} \right) \\
& \leq  ( \lfloor n^{1/2}\rfloor + 1 ) \P \left( \lfloor 2\zeta_{1}\rfloor + \kappa
+ 2 \geq n^{1/2}-1 \right) \\
& \leq  ( \lfloor n^{1/2}\rfloor + 1 ) C_0 e^{-C_1 (n^{1/2} - \kappa -3)}
\end{align*}
by (\ref{X^(1)ExpoTail}). We conclude by adjusting the constants $C_0,C_1>0$.
\end{proof}

\subsection{Renewals for a single path ($k=1$)}
\label{sect:OnePath}

Consider the path started from the vertex $\u_1$. Suppose that we are at the $j$-th good step $\tau_j$ and let us set
$$
\y_1 := g_{\tau_j}(\u_1) = g_{\tau_j}^{\downarrow}(\u_1) = \y_{1}^{\downarrow} \, \mbox{ and} \, \y_{1}^{\uparrow} := g_{\tau_j}^{\uparrow}(\u_1) ~.
$$
The realization of the renewal event $A_j$ means that the semi-ball $B^{+}(\y_{1}^{\downarrow},\kappa+1)$ contains exactly one Poisson point, say $X$, which is actually included in $B^{+}(\y_{1}^{\uparrow},1)$. It is important to remark that conditionally to the occurrence of $A_j$, the location of the Poisson point $X$ is completely free inside $B^{+}(\y_{1}^{\uparrow},1)$ and it is uniformly distributed on $B^{+}(\y_{1}^{\uparrow},1)$. We call such a good step a \textit{renewal step}. Let $\tau_{j^{\prime}}$ with $j'>j$ be another good step which is the next renewal step after $\tau_j$. Let $\y_2:=g_{\tau_{j^{\prime}}}(\u_1)$. On the renewal event $A_j$, the vertex $\y_1$ is connected in one step to the (unique) Poisson point $X$ inside $B^{+}(\y_1^{\uparrow},1)$. Now let us consider a new DSF path, called a \textit{regenerated path}, starting from the projected point $\y_1^{\uparrow}$. Like $\y_1$, the projected point $\y_1^{\uparrow}$ is also connected in one step to $X$ so that the original path (from $\u_1$) and the regenerated path (from $\y_1$) coincide beyond $X$. See Figure \ref{fig:renewalEventAndz0}.

\begin{figure}[httb]
\begin{center}
\psset{unit=0.9}
\begin{pspicture}(-2,-0.5)(12.5,5)
\psarc[fillcolor=gray,fillstyle=solid](0.5,0.5){1.4}{0}{180}
\psarc[fillcolor=white,fillstyle=solid](0.5,1.5){0.4}{0}{180}
\psarc[fillcolor=gray,fillstyle=solid](2.5,3.3){1.4}{0}{180}
\psarc[fillcolor=white,fillstyle=solid](2.5,4.3){0.4}{0}{180}
\psdots[dotsize=0.15,fillstyle=solid](0,0)
\psdots[dotsize=0.15,fillstyle=solid](0.5,0.5)
\psdots[dotsize=0.15,fillstyle=solid](0.68,1.65)
\psdots[dotsize=0.15,fillstyle=solid](2.68,2)
\psdots[dotsize=0.15,fillstyle=solid](3,2.5)
\psdots[dotsize=0.15,fillstyle=solid](2.5,3.3)
\psdots[dotsize=0.15,fillstyle=solid](2.2,4.4)
\psdots[dotsize=0.15, linecolor = red](0.5,1.5)
\psdots[dotsize=0.15, linecolor=red](2.5,4.3)
\psline(0,0)(.5,.5)
\psline(0.5,.5)(0.68,1.65)
\psline(0.68,1.65)(2.68,2)
\psline(2.68,2)(3,2.5)
\psline(3,2.5)(2.5,3.3)
\psline[linestyle=dashed,linecolor=red](0.5,1.5)(0.68,1.65)
\psarc(2.68,2){.60}{0}{180}
\psarc(0,0){.71}{0}{180}
\psarc(0.68,1.65){2.03}{0}{180}
\psarc(3,2.5){0.97}{0}{180}
\rput[tr](0.5,1.45){$\y_1^{\uparrow}$} 
\rput(0.5,0.2){$\mathbf{y}_1$}
\rput[tl](2.7,3.2){$\mathbf{y}_2$}
\psline[linewidth=0.5pt](-1.7,1.5)(3.5,1.5)
\psline[linewidth=0.5pt](-1.7,0.5)(3.5,0.5)
\psline[linewidth=0.5pt]{<->}(3.4,0.5)(3.4,1.5)
\rput(3.6,1){$\kappa$}
\psline[linewidth=0.5pt](0,3.3)(5,3.3)
\psline[linewidth=0.5pt](0,4.3)(5,4.3)
\psline[linewidth=0.5pt]{<->}(4.9,3.3)(4.9,4.3)
\rput(5.1,3.8){$\kappa$}
\psarc[fillcolor=gray,fillstyle=solid](7.5,0.5){1.4}{0}{180}
\psarc[fillcolor=white,fillstyle=solid](7.5,1.5){0.4}{0}{180}
\psarc[fillcolor=gray,fillstyle=solid](9.5,3.3){1.4}{0}{180}
\psarc[fillcolor=white,fillstyle=solid](9.5,4.3){0.4}{0}{180}
\psline[linewidth=0.5pt](7,3.3)(11.5,3.3)
\psline[linewidth=0.5pt](7,4.3)(11.5,4.3)
\psline[linewidth=0.5pt]{<->}(11.4,3.3)(11.4,4.3)
\rput(11.6,3.8){$\kappa$}
\psdots[dotsize=0.15,fillstyle=solid](9.68,2)
\psline(7.68,1.65)(9.68,2)
\psarc(7.68,1.65){2.03}{0}{180}
\psframe*[linecolor=white](5,0)(9.5,1.5)
\psdots[dotsize=0.15,fillstyle=solid](7.68,1.65)
\psline(7.5,1.5)(7.68,1.65)
\psdots[dotsize=0.15,fillstyle=solid](10,2.5)
\psline(9.68,2)(10,2.5)
\psarc(9.68,2){.60}{0}{180}
\psdots[dotsize=0.15,fillstyle=solid](9.5,3.3)
\psline(10,2.5)(9.5,3.3)
\psarc(10,2.5){0.97}{0}{180}
\psdots[dotsize=0.15,fillstyle=solid](9.2,4.4)
\psdots[dotsize=0.15,linecolor=red](7.5,1.5)
\psdots[dotsize=0.15,linecolor=red](9.5,4.3)
\psline[linewidth=0.5pt](5.3,1.5)(10.5,1.5)
\rput(7.5,1.2){$(0,0)$}
\rput[tl](9.7,3.2){$\mathbf{z}_0$}
\end{pspicture}
\caption{{\small \textit{Left hand picture represents the two successive renewal
steps with positions of the path being $\y_1$ and $\y_2$ respectively at those
steps. The new path started from $ \y_1^{\uparrow} $ continues together with
the original path. The right hand picture shows how translating to the origin
provides the distribution of the increment $ \z_0 $.
}}}
\label{fig:renewalEventAndz0}
\end{center}
\end{figure}

{Let us remark that by construction the two semi-balls generated by the connections of $\y_1$ and $\y_1^{\uparrow}$ to the same ancestor $X\in B^{+}(\y_1^{\uparrow},1)$ are included in $B^{+}(\y_1,\kappa+1)$. So no information about the PPP
$$
\NN \cap \left( \mathbb{H}^+(\y_1^{\uparrow}) \setminus B^{+}(\y_1,\kappa+1) \right)
$$
is revealed.}

Thus, let us think of the projected point $\y_1^{\uparrow}$ as the origin. So, by translation invariance of the Poisson process, the evolution of the regenerated path from $\y_1^{\uparrow}$ till $\tau_{j^{\prime}}$ (the next renewal step) can be constructed as follows. Start a path from ${\bf 0}$ until the occurrence of the first renewal event with the following set of initial information:
\begin{itemize}
\item[a)] Distribute a single point uniformly in $B^{+}({\bf 0},1)$;
\item[b)] The set $ (B^{+} ( (0, -\kappa), \kappa +1) \setminus
B^{+} ( {\bf 0},  1)) \cap \mathbb{H}^+ (0) $ has no point;
\item [c)] Take an independent Poisson process on $ \mathbb{H}^+ (0) \setminus (B^{+} ( (0, -\kappa), \kappa +1) $.
\end{itemize}
Let $\z_{\bf 0}$ be the position of the above path be at the first renewal step (see Figure \ref{fig:renewalEventAndz0}). Now, translate this path to the projected point $\y_1^{\uparrow} $ to get the position of the next renewal. Therefore, we must have that
\begin{equation}
\label{eqn:IncrementDistrn}
\y_2 - \y_1 \eqdist \z_{\bf 0}  + (0, \kappa ) .
\end{equation}
Let us show the latter rigorously. Consider the process $\{g_{\beta_\ell}(\u_1) : \ell \geq 0\}$, where the renewal times $\beta_\ell$ have been introduced in \eqref{def:betaell}. Define
\begin{equation}
\label{eq:SinglePtRwalk}
Y_{\ell} = Y_{\ell}(\u_1) := g_{\beta_\ell}(\u_1) \text{ for }\ell \geq 0.
\end{equation}

\begin{prop}
\label{prop:SinglePtRWalk}
The process $\{ Y_{\ell} - Y_{\ell - 1} : \ell\geq 2\}$ is a sequence
of i.i.d. random vectors, whose distribution is given by
$ \z_0 + (0, \kappa) $ where $ \z_0$ is as defined
in equation (\ref{eqn:IncrementDistrn}).
\end{prop}

Because $\u_1$ does not benefit from a renewal environnement, the first increment $Y_1-Y_0=Y_1-\u_1$ is not distributed according to the other increments $Y_{\ell}-Y_{\ell - 1}$, for $\ell\geq 2$. Since the definition of $ \z_0 $ does not depend on the starting vertex $\u_1$, the same holds for the increment distribution, for $\ell\geq 2$.

\begin{proof}
Fix $ m \geq 3 $ and Borel subsets $ B_2, \dotsc, B_m $ of $ \R^2$. Let
$ I_{\ell} ( B_{\ell} ) $ be the indicator random variable of the event $ \{
Y_{\ell} - Y_{\ell - 1} \in B_{\ell } \} $. Then, we have
\begin{align*}
& \P (  Y_{\ell} - Y_{\ell - 1} \in B_{\ell } \text{ for } \ell =  2, \dotsc, m )
= \E (  \prod_{ \ell = 2}^{m } I_{\ell} ( B_{\ell} )) \\
&  = \E \Bigl( \E  \bigl( \prod_{ \ell = 2}^{m } I_{\ell} ( B_{\ell} )
\mid {\cal G}_{m-1} \bigl) \Bigr) = \E \Bigl( \prod_{ \ell = 2}^{m-1 } I_{\ell} ( B_{\ell} )
\E  \bigl(  I_{m} ( B_{m} ) \mid {\cal G}_{m-1} \bigl) \Bigr)
\end{align*}
as the random variables $ I_{\ell} ( B_{\ell} )  $ are measurable w.r.t. $ {\cal G}_{m-1} $
for $ \ell =  2, \dotsc, m-1 $. Note that the $\sigma$-algebra
${\cal G}_{m-1} = {\cal G}_{m-1}(\u_1)$
contains the information brought by the single path started at $\u_1$ until its
$(m-1)$-th renewal step.

Since any renewal step is first a good step, history regions are bounded
below the horizontal line  $ y =  g_{\beta_{m-1}}(\u_1)(2) + \kappa $. Therefore only
information of relevance carried at the $(m-1)$-th renewal step for constructing the path from
$ g^{\uparrow}_{\beta_{m-1}}(\u_1) $  is that
there is exactly one point in $ B^{+} ( g^{\uparrow}_{\beta_{m-1}}(\u_1), 1) $
(hence it must be uniformly distributed inside $ B^{+} ( g_{\beta_{m-1}}^{\uparrow}(\u_1), 1) $)
and no point in $  (B^{+} ( g_{\beta_{m-1}}^{\downarrow}(\u_1), \kappa+ 1)
\setminus B^{+} ( g^{\uparrow}_{\beta_{m-1}}(\u_1), 1) )
\cap {\mathbb H}^{+} ( g^{\uparrow}_{\beta_{m-1}}(\u_1)(2) ) $.
Therefore, all information that is used
to construct the path till the good step resulting in the $(m-1)$-th renewal
is no more required (see  Figure \ref{fig:renewalEventAndz0}).

Together with the properties of the PPP, these observations allow us
to use the discussion before equation (\ref{eqn:IncrementDistrn}) to say that
the conditional distribution of $ g_{\beta_m}(\u_1) - g^{\uparrow}_{\beta_{m-1}}(\u_1) $
given $ {\cal G}_{m-1}$ is given by $ \z_0 $. Therefore, we have
$$  g_{\beta_m}(\u_1) - g_{\beta_{m-1}}(\u_1) \mid {\cal G}_{m-1}(\u_1)  \eqdist
\z_0 + (0, \kappa) $$
so that we obtain
\begin{align*}
& \P (  Y_{\ell} - Y_{\ell - 1} \in B_{\ell } \text{ for } \ell =  2, \dotsc, m )
=  \E \Bigl( \prod_{ \ell = 2}^{m-1 } I_{\ell} ( B_{\ell} )
\E  \bigl(  I_{m} ( B_{m} ) \mid {\cal G}_{m-1} \bigl) \Bigr) \\
& =
\P ( (\z_0 + (0, \kappa) \in B_m ) \E \Bigl( \prod_{ \ell = 2}^{m - 1 } I_{\ell} ( B_{\ell} ) \Bigr).
\end{align*}
Now, induction on $m$ completes the proof.
\end{proof}

The distribution of $ \z_{\bf 0} $ depends only on the  uniformly
distributed point in $  B^{+} ( {\bf 0},  1) $ and an independent Poisson process
on $ \mathbb{H}^+ (0) \setminus B^{+} ( (0, -\kappa), \kappa +1) $.
Therefore, from the symmetry property of Poisson process and symmetry of the regions
considered above (by reflecting the configuration with respect to the $ y$-axis), we have
that the first coordinate of $ \z_{\bf 0} $ has a \textit{symmetric} distribution, which we sum up in the:

\begin{corollary}
\label{cor:FirstCoOrdRW}
The process $\{ Y_{\ell+1}(1) - Y_{\ell }(1) : \ell\geq 1\}$ is a sequence
of i.i.d. random variables, whose distribution is symmetric and independent of
the starting point.
\end{corollary}

In the next section, we will show that the increment random variable $Y_{\ell+1}-Y_\ell$ has moments of all orders. This and Corollary \ref{cor:FirstCoOrdRW} imply that a (single) diffusively scaled DSF path converges to the Brownian motion.

\subsection{Renewal for two paths ($k=2$)}
\label{sect:TwoPaths}

Let us consider two starting points $\u_1,\u_2$ with $\u_2(1)\geq\u_1(1)$ and $|\u_2(2) - \u_1(2)|\leq \kappa$. Like the single path case, our idea is again to start two regenerated paths from the projected points $g_{\beta_{\ell}}^{\uparrow}(\u_1)$ and $g_{\beta_{\ell}}^{\uparrow}(\u_2)$. For $\ell\geq 0$ and $i\in\{1,2\}$, we define
$$
Y_{\ell}(\u_i) := g_{\beta_{\ell}}^{\uparrow}(\u_i) ~.
$$
Moreover the random variable $Z_\ell$ denotes the distance between the two trajectories started at $\u_1$ and $\u_2$ projected on the horizontal axis and at the $\ell$-th renewal step:
\begin{equation}
\label{def:ZProcess}
Z_{\ell} = Z_\ell(\u_1, \u_2) :=   g_{\beta_{\ell}}^{\uparrow}(\u_2)(1) - g_{\beta_{\ell}}^{\uparrow}(\u_1)(1) ~.
\end{equation}
As the DSF paths are non-crossing, the process $\{Z_{\ell} : \ell \geq 0\}$ is non-negative with zero being an absorbing state.

Writting
\begin{eqnarray*}
| Z_{\ell+1} - Z_{\ell} | & \leq & | g_{\beta_{\ell+1 }}^{\uparrow}(\u_2)(1) - g_{\beta_{\ell}}^{\uparrow}(\u_2)(1) | + | g_{\beta_{\ell+1 }}^{\uparrow}(\u_1)(1) - g_{\beta_{\ell}}^{\uparrow}(\u_1)(1) | \\
& \leq & \| Y_{\ell+1}(\u_2) - Y_{\ell}(\u_2) \|_{2} + \| Y_{\ell+1}(\u_1) - Y_{\ell}(\u_1) \|_{2} \\
& \leq & 2 W_{\ell+1} ~,
\end{eqnarray*}
where $W_{\ell+1}$ is the size of the $(\ell+1)$-th renewal block, defined in (\ref{def:WidthJtRegenration}), we immediatly deduce from Proposition \ref{prop:IIDMaxSubExpTail} the next result.

\begin{corol}
\label{cor:IncrementMoments}
For any two vertices $\u_1, \u_2 \in \R^2$ with $\u_1(1) < \u_2(1)$ and $|\u_1(2) - \u_2(2)| \leq \kappa$,
\begin{itemize}
\item[(1)] the random variables $Y_{\ell+1}(\u_i)-Y_{\ell}(\u_i)$, for $\ell\geq 0$ and $i\in\{1, 2\}$ have moments of all orders and
\item[(2)] the increments of the process $\{ Z_\ell : \ell \geq 0 \}$ have moments of all orders.
\end{itemize}
\end{corol}

There is a crucial difference between the single path case and the two paths case. Indeed, at the $\ell$-th renewal step, if the distance between $g_{\beta_{\ell}}(\u_1)$ and $g_{\beta_{\ell}}(\u_2)$ is smaller than $\kappa+1$, {it may happen that the two paths coalesce during the next step: see Figure \ref{fig:Aj}. When this is the case, the two original paths (from $\u_1$ and $\u_2$) coincide beyond this renewal step and are no longer equal to the two regenerated paths (from $g_{\beta_{\ell}}^{\uparrow}(\u_1)$ and $g_{\beta_{\ell}}^{\uparrow}(\u_2)$). They are actually equal to one of the regenerated paths. This means in particular that coalescence between the two original paths may occur before that the process $\{Z_{\ell} : \ell \geq 0\}$ hits zero.}

\begin{figure}[httb]
\begin{center}
\psset{unit=1.3}
\begin{pspicture}(-3,-0.5)(5,2)
\psarc[fillcolor=gray,fillstyle=solid](0.5,0.5){1.4}{0}{180}
\psarc[fillcolor=gray,fillstyle=solid](1.35,0.5){1.4}{0}{180}
\psarc[fillcolor=white,fillstyle=solid](0.5,1.5){0.4}{0}{180}
\psarc[fillcolor=white,fillstyle=solid](1.35,1.5){0.4}{0}{180}
\psarc(0.5,0.5){1.4}{0}{180}
\psarc(1.35,0.5){1.4}{0}{180}
\psdots[dotsize=0.15,fillstyle=solid](2.1,0.08)
\psdots[dotsize=0.15,fillstyle=solid](1.35,0.7)
\psline(2.1,0.08)(1.35,0.7)
\psarc(2.1,0.08){.965}{0}{180}
\psline(0.5,0.5)(1.35,0.7)
\psdots[dotsize=0.15,fillstyle=solid](1.5,1.7)
\psline(1.35,0.7)(1.5,1.7)
\psline[linecolor = red](1.35,1.5)(1.5,1.7)
\psdots[dotsize=0.15,fillstyle=solid](1.35,0.7)
\psline(0,0)(.5,.5)
\psline(-2.7,0.5)(5,0.5)
\psline(-2.7,1.5)(5,1.5)
\rput(0.1,0.27){$g_{\beta}(\mathbf{u}_1)$}
\rput(1.03,0.9){$g_{\beta}(\mathbf{u}_2)$}
\psline{<->}(5,0.5)(5,1.5)
\rput(5.2,1){$\kappa$}
\psdots[dotsize=0.15,fillstyle=solid](0,0)
\psdots[dotsize=0.15,fillstyle=solid](0.5,0.5)
\psdots[dotsize=0.15,fillstyle=solid](0.38,1.65)
\psarc(0,0){.71}{0}{180}
\psarc(0.5,0.5){.87}{0}{180}
\psdots[dotsize=0.15, linecolor = red](0.5,1.5)
\psdots[dotsize=0.15, linecolor = red](1.35,1.5)
\end{pspicture}
\caption{{\small \textit{This picture represents a renewal step of the joint exploration process $\{(g_n(\u_1), g_n(\u_2),H_n) : n \geq 0\}$ (with only $k=2$ trajectories): the event $A_j$ occurs where $\beta=\tau_j$. The red dots represent $g_{\beta}^{\uparrow}(\u_1)$ and $g_{\beta}^{\uparrow}(\u_2)$.
Observe that the vertex $g_{\beta}(\u_1)$ connects to $g_{\beta}(\u_2)$ and not to the uniformly distributed point in $B^+(g_{\beta}^{\uparrow}(\u_1), 1)$, i.e. the two trajectories of the DSF merge.}}}
\label{fig:Aj}
\end{center}
\end{figure}


However, if the vertices $g_{\beta_{\ell}}(\u_1)$ and $g_{\beta_{\ell}}(\u_2)$ are far away (at least $\kappa+1$), it is easy to observe that the original paths and the regenerated paths starting from the projected points $g_{\beta_{\ell}}^{\uparrow}(\u_1)$ and $g_{\beta_{\ell}}^{\uparrow}(\u_2)$ would proceed together.

{Let us focus now on the case where $Z_{\ell}=g_{\beta_{\ell}}^{\uparrow}(\u_2)(1)-g_{\beta_{\ell}}^{\uparrow}(\u_1)(1)$ is large, precisely on $F_\ell:=\{ W_{\ell+1}\leq Z_\ell/3\}$. On this event, the regions explored by the DSF paths between the $\ell$-th and the $(\ell+1)$-th renewal steps are each included in rectangles centered at the respective projected vertices $g_{\beta_{\ell}}^{\uparrow}(\u_1)$ and $g_{\beta_{\ell}}^{\uparrow}(\u_2)$, and of width $W_{\ell+1}$ smaller than $Z_\ell/3$. So they are disjoint. We can then proceed to the following transformation : we interchange the point configurations of both disjoint rectangles without changing the outside. Let us denote $\NN^{\ast}$ the resulting PPP. This transformation provides:
$$
Z_{\ell+1}(\NN^{\ast}) - Z_{\ell}(\NN^{\ast}) = - \big( Z_{\ell+1}(\NN) - Z_{\ell}(\NN) \big) ~,
$$
i.e. the distribution of the increment $Z_{\ell+1}-Z_{\ell}$ is symmetric on the event $F_\ell=\{ W_{\ell+1}\leq Z_\ell/3\}$. Details are given in the proof below.} This is the main result of Corollary \ref{corol:ZprocessRwalkProperties} and this will be crucially used to obtain the tail decay of the coalescing time. See Section \ref{sec:Tail}.

In other words, the next result says that, far from the origin, the process $\{Z_\ell : \ell \geq 0\}$ behaves like a symmetric random walk satisfying certain moment bounds.

\begin{corol}
\label{corol:ZprocessRwalkProperties}
Fix any two vertices $\u_1,\u_2\in\R^2$ with $\u_1(1)\leq\u_2(1)$ and $|\u_2(2) - \u_1(2)| \leq \kappa$. Then there exist positive constants $M_0, C_0, C_1, C_2$ and $C_3$ such that:
\begin{itemize}
\item[(i)] For any $\ell\geq 0$, let us set $F_\ell:=\{W_{\ell +1} \leq Z_\ell/3\}$. Then, on the event $\{Z_\ell \geq M_0\}$, we have $\P(F^c_\ell \mid {\cal G}_{\ell} ) \leq C_3/(Z_\ell)^3$ and
\begin{equation*}
\E \big[ (Z_{\ell +1} - Z_{\ell}) \mathbf{1}_{F_\ell} \mid {\cal G}_{\ell} \big] = 0 ~.
\end{equation*}
\item[(ii)] For any $\ell \geq 0$, on the event $\{Z_\ell \leq M_0\}$,
\begin{equation*}
\E \big[ (Z_{\ell +1} - Z_{\ell})  \mid {\cal G}_\ell \big] \leq C_0 ~.
\end{equation*}
\item[(iii)] For any $\ell\geq 0$ and $m>0$, there exists $c_m>0$ such that, on the event $\{Z_\ell\in (0, m]\}$,
\begin{equation*}
\P \big( Z_{\ell+1} = 0 \mid {\cal G}_\ell \big) \geq c_m ~.
\end{equation*}
\item[(iv)] For any $\ell\geq 0$, on the event $\{Z_\ell>M_0\}$,
\begin{equation*}
\E \bigl[ ( Z_{\ell+1} - Z_{\ell} )^2 \mid {\cal  G}_\ell \bigr] \geq C_1 \; \text{ and } \;
\E \bigl[ |Z_{\ell+1} - Z_{\ell} |^3 \mid {\cal  G}_\ell  \bigr] \leq C_2 ~.
\end{equation*}
\end{itemize}
\end{corol}

\begin{proof}
For part (i), take $M_0$ sufficiently large and Proposition \ref{prop:IIDMaxSubExpTail}
gives us that on the event $\{Z_\ell \geq M_0\}$, we have $\P(F^c_\ell \mid {\cal G}_\ell)
\leq C_3/(Z_\ell)^3$ for some positive constant $C_3$.

Now let us consider the trajectories of two paths starting from the vertices $\u_1$ and $\u_2$ between the $\ell$-th and $(\ell + 1)$-th renewal steps. Note that the trajectories of these regenerated paths can be constructed with a resampled PPP over the region $\mathbb{H}^+(W^\text{move}_{\beta_\ell}(2)+\kappa)\setminus \bigl( B^+(g^{\downarrow}_{\beta_\ell} (\u_1), \kappa + 1)\cup B^+(g^{\downarrow}_{\beta_\ell} (\u_2), \kappa + 1) \bigr)$ and resampled independent uniform distributions over the (disjoint) semi-balls $B^+(g^{\uparrow}_{\beta_\ell}(\u_1),1)$ and $B^+(g^{\uparrow}_{\beta_\ell}(\u_2),1)$ without changing the joint distribution of the trajectories. Recall that there are no Poisson points in the region
$$\bigcup_{i=1}^2 \bigl (\mathbb{H}^+(W^\text{move}_{\beta_{\ell}}(2) + \kappa ) \cap B^+(g^{\downarrow}_{\beta_\ell} (\u_i), \kappa + 1) \bigr )\setminus B^+(g^{\uparrow}_{\beta_\ell} (\u_i),  1) ~.
$$
We construct a new point process in the following way:
\begin{itemize}
\item[(1)] Given ${\cal G}_\ell$, the realizations of the point process in the rectangles
$R_1 = g^{\uparrow}_{\beta_\ell}(\u_1) + [- Z_\ell / 3, Z_\ell / 3] \times
[ 0, Z_\ell  / 3] $ and  $R_2 = g^{\uparrow}_{\beta_\ell}(\u_2) + [- Z_\ell / 3, Z_\ell / 3]
\times [ 0, Z_\ell  / 3] $ are interchanged.
\item[(2)] The realization of the PPP $\NN$ outside these two rectangles is kept as it is.
\end{itemize}
We should note that both these rectangles $ R_1 $ and $ R_2 $ contain the unique Poisson point uniformly distributed over the semi-balls $B^+(g^{\uparrow}_{\beta_\ell} (\u_1), 1)$ and $B^+(g^{\uparrow}_{\beta_\ell} (\u_2), 1)$ as well. We observe that the newly constructed point process $\NN^{\ast}$ has the same distribution as a PPP conditioned to have a unique point in the
semi-balls $B^+(g^{\uparrow}_{\beta_\ell} (\u_i), 1)$ and no point in the regions $\mathbb{H}^+(W^\text{move}_{\beta_{\ell}}(2) + \kappa ) \cap B^+(g^{\downarrow}_{\beta_\ell}(\u_i),\kappa+1)\bigr) \setminus B^+(g^{\uparrow}_{\beta_\ell}(\u_i),1)$ for $i=1,2$.

Now, we restrict our attention to the event $F_\ell = \{W_{\ell + 1} \leq Z_\ell /3\}$ and consider the trajectories in between $\ell$-th and $(\ell+1)$-th renewal steps using the newly constructed PPP $\NN^{\ast}$. We remark that for this ``new'' regenerated paths, the number of steps until the next renewal step and the size of the corresponding renewal block have not changed. Moreover, the increment of each path between the $\ell$-th and the $(\ell+1)$-th renewal steps have been interchanged. This means that the increment $Z_{\ell+1}-Z_{\ell}$ is become $-(Z_{\ell+1}-Z_{\ell})$. This completes the proof of Item $(i)$.

Item $(ii)$ follows readily from the fact that
\begin{align*}
\E[(Z_{\ell + 1} - Z_\ell) \mid {\cal G}_\ell ] \leq \E[|Z_{\ell + 1} - Z_\ell| \mid {\cal G}_\ell ] \leq \E(2 W_{\ell + 1} \mid {\cal G}_\ell) < \infty ~,
\end{align*}
since conditionnaly on ${\cal G}_\ell$, $W_{\ell+1}$ admits sub-exponential decay (see Proposition \ref{prop:IIDMaxSubExpTail} for details).

For Item $(iii)$, we recall the fact that the $\sigma$-field ${\cal G}_\ell$ does not contain any information about the PPP in  the region $\mathbb {H}^+(W^{\text{move}}_{\beta_\ell}(2) + \kappa) \setminus \bigl ( B^+(g^{\downarrow}_{\beta_\ell} (\u_1), \kappa + 1) \cup B^+(g^{\downarrow}_{\beta_\ell} (\u_2), \kappa + 1)\bigr )$. Hence, it is not difficult to convince oneself that the  conditional probability $\P(Z_{\ell + 1} = 0 \mid {\cal G}_\ell)$ is strictly positive (suitable configurations are easy to build).

It then remains to check Item $(iv)$. We observe that
\begin{align*}
\E \bigr [ (Z_{\ell + 1} - Z_\ell)^2 \mid {\cal G}_\ell \bigl] & \geq
\E \bigr [ (Z_{\ell + 1} - Z_\ell)^2\mathbf{1}_{ ( |Z_{\ell + 1} - Z_\ell |^2 \geq 1 )} \mid {\cal G}_\ell \bigl]\\
& \geq \P((Z_{\ell + 1} - Z_\ell )^2 \geq 1 \mid {\cal G}_\ell) ~.
\end{align*}
Again on the event $\{Z_\ell > M_0\}$, it is not difficult to observe that the probability $\P((Z_{\ell + 1} - Z_\ell )^2 \geq 1 \mid {\cal G}_\ell)$ is  strictly positive. For the third moment,
\begin{align*}
\E[(Z_{\ell + 1} - Z_\ell)^3 \mid {\cal  G}_\ell ] \leq \E[|Z_{\ell + 1} - Z_\ell|^3 \mid {\cal G}_\ell ] \leq \E((2 W_{\ell + 1})^3\mid {\cal G}_\ell ) < \infty ~.
\end{align*}
\end{proof}

\section{Tail distribution for the coalescence time of two paths}
\label{sec:Tail}

In this section we start with two points $\u_1,\u_2$  in $\R^{2}$ such that $\u_1(1)<\u_2(1)$ and $\u_1(2)=\u_2(2)=0$. The initial history set $H_0$ is assumed empty. As explained in the introduction, a key result for proving the convergence of the DSF to the BW, lies in a precise estimate for the tail distribution of the coalescence time of two paths of the DSF:
 \begin{equation}
 \label{CoalTime:u1u2}
 T(\u_1,\u_2) := \inf \{t\geq 0 : \, \pi^{\u_1}(t) = \pi^{\u_2}(t) \}
 \end{equation}
 where $\pi^{\u_i}=(\pi^{\u_i}(t))_{t\geq 0}$ denotes the parametrization of the path $\pi^{\u_i}$. This random time is known to be almost surely finite \cite{CT13}.
 In this section we prove the following theorem on tail decay of coalescing time $T(\u_1, \u_2)$ of two DSF paths $\pi^{\u_1}$ and $\pi^{\u_2}$.

\begin{theorem}
\label{thm:CoalescingTimetail}
Assume $\u_2(1)-\u_1(1) > 0$ and $\u_2 (2) = \u_1(2) = 0$.
There exists a constant $C_{0}>0$ which does not depend on $\u_1,\u_2$ such that, for any $t>0$,
$$
\P(T(\u_1,\u_2) > t ) \leq {\frac{C_{0}}{\sqrt{t}} \max\{ 1 , \u_2(1)-\u_1(1) \} ~.}
$$
\end{theorem}

In order to prove Theorem \ref{thm:CoalescingTimetail} we develop a robust technique to obtain such an estimate for certain class of processes which need not be Markov and can be applied to a large class of models (see Remark \ref{rmk:GeneralLaplace}). We will show that processes which behave like symmetric random walks away from the origin and satisfy certain moment bounds belong to this class (for a precise statement we refer to Corollary \ref{corol:LaplaceCoalescingTimeTailNew3}). As a consequence, Corollary \ref{corol:ZprocessRwalkProperties} allows us to apply this technique for the DSF paths and gives a suitable tail decay in terms of number of renewal steps. With some additional work, we obtain the tail distribution of coalescing time.

{For the sequel, it will be crucial that the factor $\u_{2}(1)-\u_{1}(1)$ occurs in the upper bound of $\P(T(\u_1,\u_2)>t)$. Theorem \ref{thm:CoalescingTimetail} will be applied, in the proofs of criteria $(ii)$ and $(iv)$ of Theorem \ref{thm:JtConvBwebGenPath}, to starting points satisfying at each time $\u_{2}(1)-\u_{1}(1) \geq 1$.}

\subsection{A general result for upper bounding hitting time tails}
\label{sec:HittingTimeEst}

For this section, we introduce the following  notation: for a {discrete-time} process $ \{ Y_t : t \geq 0 \}$
taking non-negative values, let $ \nu^Y $ be the first hitting time to $0$, i.e.,
\begin{equation}
\label{def:genericHittingTime}
\nu^Y := \inf\{t \geq 1 : Y_t = 0\}.
\end{equation}

In this section we obtain tail decay for the hitting time of $0$ for certain class of processes
 which need not to be Markov. To start with we assume that the process is supermartingale.

 \begin{theorem}
 \label{thm:LaplaceCoalescingTimeTailNew}
 Let $\{Y_t : t \geq 0\}$ be a $\{{\cal G}_t : t \geq 0\}$ discrete-time adapted
 stochastic process taking values in $\R_+$. Suppose that:
 \begin{itemize}
 \item[(i)] For any $t\geq 0$,
 \begin{equation*}
  \E\big[(Y_{t+1}-Y_t) \mid {\cal G}_t \big] \leq 0 \text{ a.s.}
 \end{equation*}

 \item[(ii)] There exist constants $C_0,C_1>0$ such that for any $t \geq 0$,
 we have
 \begin{equation*}
 \E \bigl[ (Y_{t+1} - Y_t)^2 \mid {\cal  G}_t \bigr] \geq C_0 \; \text{ and } \;
 \E \bigl[ |Y_{t+1} - Y_t|^3 \mid {\cal  G}_t \bigr] \leq C_1,
 \end{equation*}
 on the event $\{Y_t>0\}$.
  \end{itemize}
 Then, $ \nu^Y  < \infty $ almost surely. Further,  there exists a constant $C_2>0$
 such that for any $y>0$ and any integer $n$,
 \begin{equation*}
 \P( \nu^Y > n \mid Y_0 = y) \leq \frac{C_2 y}{\sqrt{n}} .
 \end{equation*}
 \end{theorem}


 \begin{proof}
 We will denote the conditional probability and the conditional expectation given $ Y_0 = y $ as $ \P_y $ and $ \E_y $  respectively.
The proof is divided into three steps.\\

 \noindent
 \textbf{Step 1:} Assume that there exist constants $C_3,\theta_0>0$ such that for all $0<\theta<\theta_0$
 \begin{equation}
 \label{laplace:step1}
 \E_y \big( \exp (- C_3\theta^2 \nu^Y) \big) \geq \exp( -\theta y).
 \end{equation}
 Using that $x\mapsto 1 - \exp (-C_3 \theta^2 x)$ is a non-decreasing function for any $\theta>0$,
 the Markov inequality and \eqref{laplace:step1}, we get:
 \begin{equation*}
 \P_y \big( \nu^Y > n \big) \leq \frac{ \E_y  \big( 1 - \exp(-C_3\theta^2 \nu^Y) \big)}
 {1 - \exp(-C_3\theta^2 n)} \leq \frac{1 - \exp(-\theta y)}{1 - \exp(-C_3\theta^2 n)}
 \end{equation*}
 provided that $\theta<\theta_0$. Hence, for $\theta=1/\sqrt{n}$ with $n>1/\theta_0^{2}$,
 \begin{equation*}
 \P_y \big( \nu^Y > n \big) \leq \frac{ 1 - \exp(-y/\sqrt{n})}{1 - \exp(-C_3)} \leq \frac{y}{\sqrt{n}(1 - \exp(-C_3))},
 \end{equation*}
 which is the announced result with $C_2=(1-\exp(-C_3))^{-1}$.\\

 \noindent
 \textbf{Step 2:} It remains to prove the estimate \eqref{laplace:step1} on the Laplace transform of $\nu^Y$.
 To do it, we use martingale techniques. For $\theta>0$ and $j\geq 0$, let us set
 \begin{equation*}
 \psi_{\theta,j} := \E \big( \exp(-\theta(Y_{j+1}-Y_j)) \ | \ \mathcal{G}_j \big).
 \end{equation*}
 Thus we define a discrete time process as follows: $Z_0:=\exp(-\theta Y_0)=\exp(-\theta y)$
 $\P$-a.s. and for $t\geq 1$,
 \begin{equation}
 \label{laplace:def_cond}
 Z_t := \frac{\exp(-\theta Y_t)}{\prod_{j=0}^{t-1} \psi_{\theta,j}}.
 \end{equation}
 This process is a $\{\mathcal{G}_t : t\geq 0\}$-martingale since
 \begin{align*}
 \E \big( Z_{t+1} \ |\ \mathcal{G}_t \big) & =  \E \Bigl[ \frac{\exp(-\theta (Y_{t+1}-Y_t))
 \exp(-\theta Y_t)}{\prod_{j=0}^{t} \psi_{\theta,j}} \ |\ \mathcal{G}_t \Bigr] \\
 & =  \frac{Z_t}{\psi_{\theta,t}} \, \E \big[ \exp(-\theta (Y_{t+1}-Y_t)) \ |\ \mathcal{G}_t \big]  =  Z_t.
 \end{align*}
 Then, $(Z_{t\wedge \nu^Y})_{t\geq 0}$ is also a non-negative $\{\mathcal{G}_t : t\geq 0\}$-martingale
 and for any $t\geq 0$,
 \begin{equation}
 \label{laplace:step3}
 \E_y \big( Z_{t\wedge \nu^Y} \big) = \E_y \big( Z_{0} \big) = \exp(-\theta y).
 \end{equation}

 For the moment, let us assume that there exist constants $C_3,\theta_0>0$ such that for
 all $\theta\in (0,\theta_0)$ and for all index $t$,
 \begin{equation}
 \label{laplace:step2}
 \exp \big(-\theta Y_{t\wedge \nu^Y} - (t\wedge \nu^Y) C_3\theta^2 \big) \geq Z_{t\wedge \nu^Y} .
 \end{equation}

 Assuming this, we first show that $ \P ( \nu^Y < \infty ) = 1 $.
 For any $ \omega $ such that $ \nu^Y (\omega) < \infty $, we obtain that, letting by $ t \uparrow \infty $,
 $$ \exp \big(-\theta Y_{t \wedge \nu^Y} - (t \wedge \nu^Y) C_3\theta^2 \big) \to \exp ( - \theta Y_{\nu^Y}
 - \nu^Y C_3 \theta^2 ) =  \exp ( - \nu^Y C_3 \theta^2 ) $$
 since $ Y_{\nu^{Y}  (\omega) } (\omega) = 0 $.
 On the other hand, for $ \omega $ such that $ \nu^Y (\omega)
 = \infty $, we get
 $$ \exp \big(-\theta Y_{t \wedge \nu^Y} - (t \wedge \nu^Y) C_3\theta^2 \big) \leq
  \exp \big(- (t \wedge \nu^Y) C_3\theta^2 \big) \to 0 ,$$
 as $ C_3 $ are positive and $ Y_t $ is non-negative.
 We observe that $\exp \big(-\theta Y_{t\wedge \nu^Y}-(t\wedge \nu^Y) C_3 \theta^2 \big)$
 is smaller than $1$ for all $t$ since $Y_t$ is non-negative. Applying the dominated convergence
 theorem, along with (\ref{laplace:step2}) and (\ref{laplace:step3}), we can write:
 \begin{align}
 \E_y  \big( \mathbf{1}_{ \{ \nu^Y < \infty \}}  \exp(-C_3\theta^2 \nu^Y) \big)
 & =  \lim_{t \rightarrow \infty} \E_y  \Big( \exp \big(-\theta Y_{t \wedge \nu^Y}
  - (t \wedge \nu^Y) C_{3} \theta^2 \big) \Big)
 \nonumber \\
 & \geq  \limsup_{t \rightarrow \infty} \E_y  \big( Z_{t \wedge \nu^Y} \big)
 =  \exp (-\theta y) .
 \label{laplace:step2_final_relation}
 \end{align}

 Next, we let $ \theta \downarrow 0 $ in (\ref{laplace:step2_final_relation}). Again applying  the dominated
 convergence theorem, we
 obtain that $ \E_y  \big( \mathbf{1}_{ \{ \nu^Y < \infty\}}  \exp(-C_3\theta^2 \nu^Y) \big)
 \to \P_y \big( \nu^Y < \infty \big) $ while the right hand side of the above inequality converges to $1$.
 Thus, we have $ \P_y \big( \nu^Y < \infty ) \geq 1 $. Now, using the fact that $ \{ \nu^Y < \infty \} $ almost
 surely in  (\ref{laplace:step2_final_relation}), we obtain the desired relation (\ref{laplace:step1}).\\


 \noindent
 \textbf{Step 3:} If there exist constants $C_3,\theta_0>0$ such that, for any
 $\theta\in (0,\theta_0)$ and any $j\in \{0,1,\dots, t\wedge\nu^Y -1\}$,
 \begin{equation}
 \label{laplace:step4}
 \log \big( \psi_{\theta,j} \big) \geq C_{3} \theta^2
 \end{equation}
 then $\sum_{j=0}^{t\wedge\nu^Y -1} \log \big(\psi_{\theta,j} \big)
 \geq (t\wedge \nu^Y)C_{3}\theta^2$ from which (\ref{laplace:step2}) easily follows.
 We observe that, for any $ x \in \R $,  by Taylor's expansion,  $ e^x = 1 + x + x^2/2 + x^3/6 + e^{x_0} x_{0}^{4}/4! $
 where $ x_0 $ is some point in between $ 0 $ and $ x $. Thus, for all $ x \in \R $, $ e^x \geq
  1 + x + x^2/2 + x^3/6 $. Now, fix any index  $j\in \{0,1,\dots, t\wedge\nu^Y -1\}$ so that $Y_j>0$.
 Using hypotheses $(i)$ and $(ii)$,  we have, for any $\theta \in [0, \infty)$
 \begin{align*}
 \psi_{\theta,j} & =  \E \big(e^{-\theta(Y_{j+1}-Y_j)}\ |\ \mathcal{G}_j\big) \\
 & \geq  1 -\theta\  \E\big(Y_{j+1}-Y_j \ |\ \mathcal{G}_j\big) +
 \frac{\theta^2}{2} \E\big((Y_{j+1}-Y_j)^2 \ |\ \mathcal{G}_j\big)-\frac{\theta^3}{6}
 \E\big((Y_{j+1}-Y_j)^3\ |\ \mathcal{G}_j\big) \\
 & \geq  1 + C_0 \frac{\theta^2}{2}-C_1 \frac{\theta^3}{6}.
 \end{align*}
 The constants $C_0,C_1$ do not depend on $j$. The function $\theta\in [0,\infty)
 \mapsto 1+C_0\theta^2/2 -C_1\theta^3/6$ is continuous, equal to 1 at
 $\theta=0$ and increasing on the neighbourhood of 1. Hence,
 it is possible to pick $\theta_0>0$ such that for all $0<\theta<\theta_0$,
 $1<1+C_0\theta^2/2 -C_1\theta^3/6<2$. Since $\log(x)\geq (x-1)/2$ for
 $x\in (1,2)$, we obtain for any $0<\theta<\theta_0$,
 \begin{equation*}
 \frac{1}{\theta^2} \log \big( \psi_{\theta,j} \big) \geq \frac{1}{\theta^2}
 \log \big( 1 + C_0 \frac{\theta^2}{2} - C_1 \frac{\theta^3}{6} \big)
 \geq \frac{C_0}{4} - \frac{C_1 \theta}{12}.
 \end{equation*}
 We then deduce (\ref{laplace:step4}) for $\theta_0>0$ small enough and $C_{3}=C_0/8$.
 \end{proof}

 \begin{remark}
 If we use the bound $ e^x \geq 1 + x + (x^+)^2/2 $ for $ x \in \R $
 where $ x^+ = \max( x, 0) $, the requirements in (ii) could be reduced to
 $ \E \bigl[ \bigl( (Y_{t+1} - Y_t)^{+} \bigr)^2 \mid {\cal  G}_t \bigr] \geq C_0 $
 on the set $ \{ Y_t > 0 \} $. \qed
 \end{remark}

 Theorem \ref{thm:LaplaceCoalescingTimeTailNew} assumes the supermartingale structure
 which is often not available for the whole set of values that the process may take. Next we
 prove a result which only assumes the supermartingale structure when the process is away
 from the origin.

 \begin{theorem}
 \label{thm:LaplaceCoalescingTimeTailNew2}
 Let $\{Y_t : t \geq 0\}$ be a $\{{\cal G}_t : t \geq 0\}$ adapted
 stochastic process taking values in $\R_+$. 
 Suppose that there exist positive constants $ M, C_0, C_1, C_2 $ and
 $ 0 < p_0 < 1 $ such that
 \begin{itemize}
 \item[(i)]  for any $t\geq 0$,
 \begin{equation*}
 \E \big[Y_{t+1}-Y_t \mid {\cal G}_t\big]
 \leq C_0\mathbf{1}_{\{Y_t \in (0, M] \}};
 \end{equation*}

 \item[(ii)]  for any $t\geq 0$,
 \begin{equation*}
  \P \big( Y_{t+1} = 0 \mid {\cal G}_t\big) \geq c_0
  \text{ on the event }  \{Y_t \in (0, M]\};
 \end{equation*}

 \item[(iii)]  for any $t \geq 0$,  on the event $\{Y_t> M\}$, a.s.  we have
 \begin{equation*}
 \E \bigl[ (Y_{t+1} - Y_t)^2 \mid {\cal  G}_t \bigr] \geq C_1 \; \text{ and } \;
 \E \bigl[ |Y_{t+1} - Y_t|^3 \mid {\cal  G}_t \bigr] \leq C_2 .
 \end{equation*}

 \end{itemize}
 Then, $ \nu^Y  < \infty $ almost surely. Further,  there exist positive constants $C_3, C_4 $
 such that for any $y>0$ and any integer $n$,
 \begin{equation*}
 \P( \nu^Y > n \mid Y_0 = y) \leq \frac{C_3 + C_4 y}{\sqrt{n}} .
 \end{equation*}
 \end{theorem}

 \begin{proof}
 The proof follows similar steps as in the proof of Theorem \ref{thm:LaplaceCoalescingTimeTailNew}. Again we
 use the same notations as in the previous Theorem and the proof is divided into steps.\\

 We will show in Step 1 that  there exist constants $C_4, C_5 > 0, 0 < \theta_0 < -\log(1- p_0)/(2C_5) $ such that for all $0<\theta<\theta_0 $
 \begin{equation}
 \label{laplace2:step1}
 \E_y \big( \exp (- C_4 \theta^2 \nu^Y) \big) \geq e^{- 2 \theta y}  e^{- 2 C_5 \theta} \frac{ 1 -
 (1-p_0) e^{ 2 C_5 \theta}}{p_0}.
 \end{equation}
 Choosing $ \theta=1/\sqrt{n}$ as earlier and making few algebraic simplifications, we have the desired
 result.\\

\noindent \textbf{Step 1:} To obtain (\ref{laplace2:step1}), we consider the same exponential martingale again, i.e.,
 for $\theta>0$, we set  $Z_0:=\exp(-\theta Y_0)=\exp(-\theta y)$
 $\P$-a.s. and for $t\geq 1$,
 \begin{equation*}
 Z_t := \frac{\exp(-\theta Y_t)}{\prod_{j=0}^{t-1} \psi_{\theta,j}}
 \end{equation*}
 where $ \psi_{\theta,j} $ is as defined in (\ref{laplace:def_cond}).

 Let us define $ T_t $ as the number of visits of the process to the set $ (0, M]$
 upto time $ t $, i.e.,
 \begin{equation}
 \label{laplace2:defTt}
 T_t = \sum_{j = 0}^t \mathbf{1}_{\{ Y_t \in (0, M] \}}.
  \end{equation}

 Assume that there exist constants $C_6,C_7$ and $\theta_0>0$ such that for
 all $\theta\in (0,\theta_0)$ and for all index $t \geq 0$,
 \begin{equation}
 \label{laplace2:step2}
 \exp \big(-\theta Y_{t\wedge \nu^Y} - (t\wedge \nu^Y) C_6\theta^2 + \theta C_7 T_{t\wedge \nu^Y - 1}
 \big) \geq Z_{t\wedge \nu^Y} .
 \end{equation}This will be proved in Step 2.\\

 First we argue that $ \nu^Y < \infty $ almost surely. Consider the case $ y > M$.
 Let us take $ Y_t^{(M)} := Y_t \mathbf{1}_{\{ Y_t > M\}} $. We observe that the conditions of Theorem
 \ref{thm:LaplaceCoalescingTimeTailNew} are satisfied by the process $ \{ Y_t^{(M)} : t \geq 0 \} $ and
 hence by the previous Theorem the process $ Y_t^{(M)} $ will hit $0 $ in finite time, i.e., the process
  $  Y_t $ will enter the set $ [0, M] $ in finite time. If, at this point, the process $\{  Y_t \} $ is already at $0$,
 we are done. Otherwise, by assumption (ii), it has  a strictly positive probability $ p_0 $ of hitting $0$
 at the next step. So, in finitely many steps, the process will either
 hit $ 0 $, which again shows that $ \nu^Y < \infty $ almost surely, or will go out to the set $ (M, \infty) $.
 In the latter case, we are back to the starting situation. So, again the process will hit the set $  [0, M] $
 in finite time and so on. Since $ p_0 > 0 $, the process can only go out of $ (0,M] $ to $ (M, \infty ) $
 finitely many times, before it hits $ 0 $. Thus, the process will hit the set $ 0 $ in finite time
 almost surely. When the process starts in  $ (0, M] $, the situation is as above without having first to hit
 the set $ [0, M] $. So, in all cases, $ \nu^Y < \infty $ almost surely.

 Now we observe that $ T_{t\wedge \nu^Y - 1} $
 is a non-decreasing in $ t $ and hence converges to $ T_{\nu^Y-1} $ as $t \uparrow \infty$.
 Since  at each time point, the process is in the set $ (0, M] $, there is at least
 a probability of $ p_0 $ of hitting $0$, the number of visits in $ (0, M] $ before hitting $0$
 is stochastically dominated by a geometric random variable $ G $ (the total number of trials before a success)
 with probability of success $ p_0$. Thus, $ T_{\nu^Y-1} $
  is finite almost surely and is stochastically dominated by the above described geometric random
 variable $G$.

 Now, we impose further restriction on $ \theta_0. $ Assume that $ \theta_0 < - \log (1-p_0)/ (2C_7) $.
 Hence, we note that, for $ 0 < \theta < \theta_0 $, we have $ \E_y \bigl( \exp( 2 \theta C_7 G) \bigr) < +\infty $.
 Letting $ t \uparrow \infty $ in (\ref{laplace2:step2}), we have that the left hand side converges to
 to $ \exp \bigl(- C_6\theta^2 \nu^Y +   \theta C_7 T_{\nu^Y - 1} \bigr). $ Furthermore, for any $ t \geq 0 $,
 the left hand side is bounded by $ \exp \big( \theta C_7 T_{t\wedge \nu^Y - 1}\big) \leq
  \exp \bigl( \theta C_7 T_{\nu^Y - 1}   \bigr) $. Since $ T_{\nu^Y - 1} $ is stochastically dominated by
 $ G $ and $ \E_y (\exp(\theta C_7 G)) < \infty $,  the dominated convergence theorem may be applied to
 conclude
 \begin{align*}
  \exp (-\theta y) & =  \lim_{t \rightarrow \infty} \E_y  \big( Z_{t \wedge \nu^Y} \big)  \\
 & \leq \lim_{t \rightarrow \infty} \E_y  \bigl[
 \exp \bigl( -\theta Y_{t\wedge \nu^Y} - (t\wedge \nu^Y) C_6\theta^2 + \theta C_7 T_{t\wedge \nu^Y - 1}
 \bigr) \bigr] \\
 & =  \E_y  \bigl[   \exp(-C_6\theta^2 \nu^Y +\theta C_7 T_{\nu^Y - 1}   ) \bigr] \\
 & \leq \Bigl( \E_y  \bigl[   \exp(- 2 C_6\theta^2 \nu^Y ) \bigr] \Bigr)^{1/2}
 \Bigl( \E_y  \bigl[   \exp( 2 \theta C_7 T_{\nu^Y - 1}   ) \bigr]  \Bigr)^{1/2} \\
 & \leq \Bigl( \E_y  \bigl[   \exp(- 2 C_6\theta^2 \nu^Y ) \bigr] \Bigr)^{1/2}
 \Bigl( \E_y  \bigl[   \exp( 2 \theta C_7 G   ) \bigr]  \Bigr)^{1/2}
 \end{align*}
 where the inequality in fourth line is obtained by applying Cauchy-Schwartz inequality and where the final
 inequality is obtained by stochastic domination. By choice of $ \theta_0 $, the moment
 generating function of $G$ is finite for $ 0 < \theta < \theta_0 $. Using the expression of $G$ and squaring the right hand side, we obtain the inequality in (\ref{laplace2:step1}).\\

\noindent \textbf{Step 2:} To obtain (\ref{laplace2:step2}), we follow  similar steps. Fix any index
 $j\in \{0,1,\dots, t\wedge\nu^Y -1\}$ so that $Y_j>0$. If $ Y_j > M$, we have, as earlier, for suitable
 choice of $ \theta_1 > 0 $,
 \begin{equation*}
 \log \big( \psi_{\theta,j} \big) \geq  C_8 \theta^2
 \end{equation*}
 for $0<\theta<\theta_1$. For the case $ Y_j \in (0, M] $, we note that $  e^x = 1 + x + e^{x_0} x^2/2  $
 for some $ x_0 $ in between $ 0 $ and $ x $ so that $ e^x \geq 1 + x $. Thus, we have
 \begin{equation*}
 \psi_{\theta,j}  =  \E \big(e^{-\theta(Y_{j+1}-Y_j)}\ |\ \mathcal{G}_j\big)
  \geq  1 -\theta   \E\big(Y_{j+1}-Y_j \ |\ \mathcal{G}_j\big) \geq 1 - C_0 \theta .
 \end{equation*}
 Noting that the function $ \theta \mapsto 1 - C_0\theta $ is continuous, equal to 1 at
 $\theta=0$ and decreasing on the neighbourhood of 1. Hence,
 it is possible to pick $\theta_2>0$ such that for all $0<\theta<\theta_2$,
 $1 > 1 - C_0\theta > 1/2$. Since $\log(x)\geq 2(x-1)$ for
 $x \in (1/2,1)$, we obtain for any $0<\theta<\theta_2$,
 \begin{equation*}
 \log \big( \psi_{\theta,j} \big) \geq - 2 C_0 \theta
 \end{equation*}
 for $0<\theta<\theta_2$.
 Combining these two inequalities,  for $ 0 < \theta  < \min \{ \theta_1, \theta_2, 1\} = \theta_0$, we have
 \begin{align*}
  \sum_{j=0}^{t \wedge \nu^Y-1} \log \bigl( \psi_{\theta,j} \bigr) & \geq
  C_8 \theta^2 \sum_{j=0}^{t \wedge \nu^Y-1} \mathbb{ I} (\{ Y_j \in (M, \infty) \} )
 -  2 C_0 \theta  \sum_{j=0}^{t \wedge \nu^Y-1} \mathbb{ I} (\{ Y_j \in (0, M] \} ) \\
 & =  C_8 \theta^2  (t \wedge \nu^Y - T_{t \wedge \nu^Y-1} ) - 2 C_0 \theta T_{t \wedge \nu^Y-1} \\
 & \geq  C_8 \theta^2  (t \wedge \nu^Y) - (2 C_0 + C_8)  \theta T_{t \wedge \nu^Y-1}.
 \end{align*}
 This completes the proof of (\ref{laplace2:step2}).
 \end{proof}

 \begin{remark}
 Here also, we can use  the bound $ e^x = 1 + x + (x^+)^2/2 $ for $ x \in \R $.
 In such a case we can replace condition (iii) by
 $ \E \bigl[ \bigl( (Y_{t+1} - Y_t)^{+} \bigr)^2 \mid {\cal  G}_t \bigr] \geq C_0 $
 on the set $ \{ Y_t > M \} $. \qed
 \end{remark}

Finally, we deal with situations where the increments of the process $\{Y_t : t \geq 0\}$ have a null expectation on an event with high probability. This is typically the case when the considered process closely resembles to a random walk when it is far from the origin.

\begin{corol}
\label{corol:LaplaceCoalescingTimeTailNew3}
Let $\{Y_t : t \geq 0\}$ be a $\{{\cal G}_t : t \geq 0\}$ adapted stochastic process taking values in $\R_+$. Let $\nu^Y := \inf\{t \geq 1 : Y_t = 0\}$ be the first hitting time to $0$. Suppose for any $t \geq 0$ there exist positive constants $ M_0, C_0, C_1, C_2, C_3 $ such that:
\begin{itemize}
\item[(i)] There exists an event $F_t$ such that, on the event $\{Y_t> M_0\}$, we have $\P(F^c_t \mid {\cal G}_t) \leq C_0 / Y^3_t$ and
\begin{equation*}
\E \big[ (Y_{t+1}-Y_t) \mathbf{1}_{F_t} \mid {\cal G}_t\big] = 0 ~.
\end{equation*}
\item[(ii)] For any $t\geq 0$, on the event $\{Y_t \leq M_0\}$,
\begin{equation*}
\E \big[ (Y_{t+1}-Y_t)  \mid {\cal G}_t\big] \leq C_1 ~.
\end{equation*}
\item[(iii)] For any $t\geq 0$ and $ m > 0 $, there exists $c_m > 0 $ such that, on the event $\{Y_t \in (0, m]\}$,
\begin{equation*}
\P \big( Y_{t+1} = 0 \mid {\cal G}_t\big) \geq c_m ~.
\end{equation*}
\item[(iv)] For any $t \geq 0$, on the event $\{Y_t> M_0\}$, we have
\begin{equation*}
\E \bigl[ (Y_{t+1} - Y_t)^2 \mid {\cal  G}_t \bigr] \geq C_2 \; \text{ and } \; \E \bigl[ |Y_{t+1} - Y_t|^3 \mid {\cal  G}_t \bigr] \leq C_3 ~.
\end{equation*}
\end{itemize}
Then, $\nu^Y<\infty $ almost surely. Further, there exist positive constants $C_4, C_5$ such that for any $y>0$ and any integer $n$,
\begin{equation*}
\P( \nu^Y > n \mid Y_0 = y) \leq \frac{C_4 + C_5 y}{\sqrt{n}} ~.
\end{equation*}
\end{corol}

{Recall that, by Corollary \ref{corol:ZprocessRwalkProperties}, the four hypotheses $(i)$-$(iv)$ of Corollary \ref{corol:LaplaceCoalescingTimeTailNew3} are satisfied by the process $\{Z_\ell : \ell \geq 0\}$ defined in (\ref{def:ZProcess}) by $Z_{\ell}=g_{\beta_{\ell}}^{\uparrow}(\u_2)(1)-g_{\beta_{\ell}}^{\uparrow}(\u_1)(1)$.}

\begin{proof}
Let us define $ \phi : [0, \infty ) \to [0, \infty ) $ by
\begin{equation*}
\phi (u) = 1 + \frac{ 1 }{ (1+u)^{1/3}} .
\end{equation*}
 Clearly $ \phi $ is positive, in fact,  $ 1 \leq \phi (u) \leq 2 $ for  all $ u \in [0, \infty ) $.
 Furthermore,
 \begin{align*}
  & \phi^{(1)} (u) = - \frac{ 1 }{ 3 (1+u)^{4/3} }, \qquad  \phi^{(2)} (u) =  \frac{ 4 }{ 9 (1+u)^{7/3} }\\
  & \text{ and }  \phi^{(3)} (u) =  - \frac{ 28 }{ 27 (1+u)^{10/3} } < 0.
 \end{align*}
 Now define the function $ f : [0, \infty ) \to [0, \infty ) $ by $ f (0) = 0 $ and for $ x > 0 $,
 $ f(x) = \int_0^x \phi (u) du $. Then, $ f^{(1)} (u) = \phi (u) $ and $  f^{(k)} (u) = \phi^{(k-1) } (u)
 $ for $ k \geq 2 $.   Since $ \phi $ is positive, the function $ f $ is strictly increasing.

 Define the process $ Z_t = f (Y_t) $ for all $ t \geq 0 $. Since $ f $ is strictly increasing,
 we observe that $ Z_t = 0 $ if and only if $ Y_t = 0 $. Therefore, we have that $ \nu^Y = \nu^Z $.
 We choose $ M \geq  f(M_0) $ so that the function $ C ( 1 + u)^{7/3} - 3 C_2 u^2 ( 1 + u) + 4C_3 u^2  < 0 $
 for all $ u > M $ where $ C = C_3 /  C_0^{2/3} $.
 We show that for $ M $, the process $ \{ Z_t : t \geq 0 \} $ satisfy the conditions of Theorem
  \ref{thm:LaplaceCoalescingTimeTailNew2}. Therefore, using Theorem \ref{thm:LaplaceCoalescingTimeTailNew2},
 for suitable constants $ C_4, C_5 > 0$, we conclude that
 \begin{equation*}
  \P ( \nu^Y > n | Y_0 = y) = \P ( \nu^Z > n | Z_0 = f(y) ) \leq \frac{ C_4 + C_5 f(y) }{ \sqrt{n}} \leq
  \frac{ C_4 + 2 C_5 y }{ \sqrt{n}}
 \end{equation*}
 using the fact that $ f^{(1)}(u) = \phi (u) \leq 2 $ for $ u > 0 $.

 To verify the conditions of Theorem \ref{thm:LaplaceCoalescingTimeTailNew2} for the process $\{ Z_t : t \geq 0 \}$,
 we have that $ (Z_{t+1} - Z_{t})^2 = (f(Y_{t+1}) - f(Y_t))^2 = \bigl( f^{(1)} (W)  ( Y_{t+1} - Y_t) \bigr)^2
 = \bigl( \phi (W) \bigr)^2 ( Y_{t+1} - Y_t)^2 \geq  ( Y_{t+1} - Y_t)^2 $ where $ W $ is some point between  $ Y_t $ and $ Y_{t+1} $.
 Similarly $  |Z_{t+1} - Z_{t}|^3 \leq 8  | Y_{t+1} - Y_{t}|^3 $.
 Thus, condition (iii) of Theorem \ref{thm:LaplaceCoalescingTimeTailNew2}
 follows easily from (iv) and the choice of $ M$.
  Also, condition (ii) is satisfied with $ p_0 = \alpha_M $ and it is easy to check that
\begin{align*}
& \E (  Z_{t+1} - Z_t  \mid {\cal G}_t)\\
&  =    \E (   f(Y_{t+1}) - f(Y_t)   \mid {\cal G}_t)  \\
&  \leq f^{(1)} (Y_t)  \E (   Y_{t+1} - Y_t \mid {\cal G}_t)\\
&  \leq 2 C_1 \mathbf{1}_{ \{ Y_t \in (0, M_0]\}} + 2 C_3^{1/3} \mathbf{1}_{\{ Y_t \in (M_0, \infty)\}} ,
\end{align*}
 where we have used the  fact that
 $ f^{(2)}(u) = \phi^{(1)} (u) < 0 $ for all $ u > 0 $ in the Taylor's expansion.  Finally, we have to
 show that when $ Z_t > M $, $ \E (  Z_{t+1} - Z_t  \mid {\cal G}_t) \leq 0 $. Using Taylor's expansion again and
 the fact that $ f^{(4)} (u) < 0 $, we have
 \begin{align*}
  & \E (  Z_{t+1} - Z_t  \mid {\cal G}_t)  = \E (   f(Y_{t+1}) - f(Y_t)   \mid {\cal G}_t) \\
 & \leq f^{(1)} (Y_t)  \E (   Y_{t+1} - Y_t \mid {\cal G}_t) + \frac{f^{(2)} (Y_t) }{2}
 \E \bigl(  (  Y_{t+1} - Y_t)^2  \mid {\cal G}_t \bigr)\\
 & \qquad \qquad   + \frac{f^{(3)} (Y_t) }{6} \E \bigl(  ( Y_{t+1} - Y_t)^3
 \mid {\cal G}_t \bigr) \\
 & \leq 2  |\E (   Y_{t+1} - Y_t \mid {\cal G}_t)| + \frac{f^{(2)} (Y_t) }{2}
 \E \bigl(  (  Y_{t+1} - Y_t)^2  \mid {\cal G}_t \bigr)\\
 & \qquad \qquad   + \frac{f^{(3)} (Y_t) }{6} \E \bigl(  ( Y_{t+1} - Y_t)^3
 \mid {\cal G}_t \bigr) .
 \end{align*}
 We observe the second term is bounded by $ - C_2 / \bigl[ 3 ( 1 +Y_t)^{4/3} \bigr] $ while the last term
 is bounded by $   4 C_3 / \bigl[ 9 ( 1 +Y_t)^{7/3} \bigr] $. The first term is broken into two parts
 and the choice $ M $ ensures that $ Y_t > M_0 $ and we have,
 \begin{align*}
  & | \E (   Y_{t+1} - Y_t \mid {\cal G}_t) | \\
 &=  | \E \bigl(   (Y_{t+1} - Y_t ) \mathbf{1}_{ F_t }\mid {\cal G}_t \bigr) +
\E \bigl(   (Y_{t+1} - Y_t ) \mathbf{1}_{ F^c_t }\mid {\cal G}_t \bigr) | \\
 &=  \E \bigl(   (Y_{t+1} - Y_t ) \mathbf{1}_{ F^c_t }\mid {\cal G}_t \bigr) | \\
 & \leq   \E \bigl(   |Y_{t+1} - Y_t | \mathbf{1}_{{ F^c_t }} \mid {\cal G}_t \bigr)  \\
 & \leq   \bigl[ \E \bigl(   |Y_{t+1} - Y_t |^3 \mid {\cal G}_t \bigr)  \bigr]^{1/3}
 \bigl[ \P ( F^c_t \mid {\cal G}_t \bigr)  \bigr]^{2/3}  \\
 & \leq  C_3^{1/3} \bigl[   C_0 / Y_t^3  \bigr]^{2/3}
 = \frac{ C^\prime_3 }{  Y_t^2 }.
 \end{align*}

  Putting back, we have that
 \begin{align*}
  & \E (  Z_{t+1} - Z_t  \mid {\cal G}_t) \\
  & \leq 2 \frac{ C^\prime_3 }{  Y_t^2 }
 - \frac{ C_2 }{ 3 ( 1 +Y_t)^{4/3} } + \frac{  4 C_3 }{ 9 ( 1 +Y_t)^{7/3} } \\
 & = \frac1{ 9 Y_t^2 ( 1 + Y_t)^{7/3} } \bigl[ C ( 1 + Y_t)^{7/3} - 3 C_2 Y_t^2 ( 1 + Y_t) + 4C_3 Y_t^2  \bigr] < 0
 \end{align*}
 whenever $ Y_t > M_0 $. This completes the proof.
 \end{proof}

\subsection{Tail distribution of coalescence time for DSF paths}
\label{sect:CoalTimeMain}

Let us denote by $\nu=\nu(\u_1,\u_2)$ the number of (renewal) steps required by the process $\{Z_\ell : \ell \geq 1\}$ to hit $0$:
\begin{equation}
\label{Defi:nu(u1,u2)}
\nu := \inf\{\ell \geq 1 : Z_\ell = 0\} ~.
\end{equation}
Clearly the ordinate $T_{\nu} := g_{\beta_\nu}(\u_1)(2) - \u_1(2)$ gives an upper bound for the coalescence time $T(\u_1,\u_2)$ of the two paths $\pi^{\u_1}$ and $\pi^{\u_2}$. To establish Theorem \ref{thm:CoalescingTimetail}, we first focus on $\nu$. Combining Corollaries \ref{corol:ZprocessRwalkProperties} and \ref{corol:LaplaceCoalescingTimeTailNew3}, we immediately get:

\begin{proposition}
\label{prop:CoalRenewalStepEst}
There exist a positive constant $C_0$ such that for any integer $n$,
\begin{equation}
\label{eq:CoalStepEst}
\P(\nu > n ) \leq {\frac{C_{0}}{\sqrt{n}} \max\{ 1 , \u_2(1)-\u_1(1) \} ~.}
\end{equation}
\end{proposition}

The above proposition allows us to prove Theorem \ref{thm:CoalescingTimetail}.

\begin{proof}[Proof of Theorem \ref{thm:CoalescingTimetail}]
It is easy to observe that
$$
g_{\beta_{\nu}}(\u_1) = g_{\beta_{\nu}}(\u_2)\text{ implies that } g_{m}(\u_1) = g_{m}(\u_2)
$$
for some $m$ such that $m\leq  \beta_{\nu}$. In other words
$$
T_\nu := g_{\beta_{\nu}}(\u_1)^\uparrow (2) - \u_1(2)  = g_{\beta_{\nu}}(\u_1)^\uparrow (2) - \u_2(2)
$$
dominates the actual coalescing time $T(\u_1, \u_2)$ of the two paths. For any $\ell \geq 0$,
clearly the time taken between $\ell$-th and $\ell + 1$-th renewals
are dominated by the width random variable $W_{\ell + 1}$ as defined
in (\ref{def:WidthJtRegenration}). Consider an i.i.d. sequence $\{{\cal W}_i : i \geq 1\}$,
each having the same distribution as ${\cal W}$, where ${\cal W}$ is a
random variable with sub-exponentially decaying tail
such that the conditional distribution of $W_{\ell + 1} \mid {\cal G}_{\ell}$
is dominated by ${\cal W}$ (for details see Proposition \ref{prop:IIDMaxSubExpTail}).
Choose $c = 1/(\E(2{\cal W}))$ and we have,
\begin{align*}
\P(T_{\nu} > t)  & \leq \P \Big( \sum_{\ell = 1}^{\lfloor ct \rfloor + 1} W_{\ell} \geq t \Big) + \P ( \nu > ct ) \\
& \leq \P \Big( \sum_{\ell = 1}^{\lfloor ct \rfloor + 1} ({\cal W}_{\ell} - \E({\cal W})) \geq t(1- c\E({\cal W})) \Big) + \frac{C_{0}}{\sqrt{ct}} \max\{ 1 , \u_2(1)-\u_1(1) \} \\
& \leq \frac{\text{Var} \Big( \sum_{\ell = 1}^{\lfloor ct \rfloor + 1} {\cal W}_{\ell} \Big)} { (t(1- c\E({\cal W}))^2} + \frac{C_{0}}{\sqrt{ct}} \max\{ 1 , \u_2(1)-\u_1(1) \} \\
& \leq \frac{(\lfloor ct \rfloor + 1)\text{Var}({\cal W})}{(t/2 )^2} + \frac{C_{0}}{\sqrt{ct}} \max\{ 1 , \u_2(1)-\u_1(1) \} \\
& \leq \frac{C_{1}}{\sqrt{t}} \max\{ 1 , \u_2(1)-\u_1(1) \} ~,
\end{align*}
for a suitable choice of constant $C_1 > 0$. This completes the proof.
\end{proof}

\begin{remark}
\label{rmk:GeneralLaplace}
Let us end this section with a final remark.
For the DSF, Coupier and Tran showed that the coalescence time between any two DSF paths is almost surely finite \cite{CT13}, which uses Burton-Keane argument. This method gives an independent proof that the coalescence time between DSF paths $\pi^{\u_1}$ and $\pi^{\u_2}$ is almost surely finite. It is also important to observe that this method is very robust.
Similar arguments as above show that the conditions of Theorem
\ref{thm:LaplaceCoalescingTimeTailNew2} hold for other drainage network models which are also in the basin of attraction of the BW \cite{colettidiasfontes,colettivalle,FFW05,RSS15,vallezuaznabar}. Actually both models studied in \cite{colettivalle,vallezuaznabar} have crossing paths and we can easily apply Corollary \ref{corol:LaplaceCoalescingTimeTailNew3} to deal with processes which arise from absolute distance between
paths which may cross over without coalescing. In such situations, if we have that the paths behave
nearly independently when they are far apart, almost all the conditions go through as above. While dealing with crossing paths, we may have to restrict that the individual paths are of small sizes as opposed to the difference
of increments that we consider for non-crossing paths.
 \end{remark}

\section{Convergence to the Brownian web}
\label{sec:cvBW}

This section is devoted to the proof of our main result, namely Theorem \ref{th:FF_BW}, but in fact we prove a stronger version stating that the sequence $\{({\cal X}_n, \widehat{\cal X}_n) : n \geq 1\}$, where ${\cal X}_n$ denotes the scaled DSF and $\widehat{\cal X}_n$ its scaled dual forest, converges in distribution to the BW and its dual $({\cal W}, \widehat{{\cal W}})$. Before stating Theorem \ref{th:FF_BW}, we define the scaled dual forest $\widehat{\cal X}_n$ and the dual Brownian web $\widehat{{\cal W}}$.\\

Let us first specify a dual forest $\widehat{\mathfrak{F}}$ to the DSF $\mathfrak{F}$. We start with the dual vertex set $\widehat{V}$. For any $(x,t)\in\R^2$, let $(x,t)_r\in\NN$ be the unique Poisson point such that
\begin{itemize}
\item[$\bullet$] $(x,t)_r(2)<t$, $h((x,t)_r)(2)\geq t$ and $\pi^{(x,t)_r}(t)>x$ where $\pi^{(x,t)_r}$ denotes the path in ${\cal X}$ starting from $(x,t)_r$;
\item[$\bullet$] there is no path $\pi\in {\cal X}$ with $\sigma_\pi<t$ and $\pi(t)\in (x,\pi^{(x,t)_r}(t))$.
\end{itemize}
Hence, $\pi^{(x,t)_r}$ is the nearest path in ${\cal X}$ to the right of $(x,t)$ starting strictly before time $t$. It is useful to observe that $\pi^{(x,t)_r}$ is defined for any $(x,t)\in\R^2$. Similarly, $\pi^{(x,t)_l}$ denotes the nearest path to the left of $(x,t)$ which starts strictly before time $t$. Now, for each $(x,t)\in {\cal N}$ the nearest left and right dual vertices are respectively defined as
$$
\widehat{r}_{(x,t)} := \bigl(( x + \pi^{(x,t)_r}(t))/2, t \bigr) \; \text{ and } \; \widehat{l}_{(x,t)} := \bigl(( x + \pi^{(x,t)_l}(t))/2, t \bigr) ~.
$$
Then, the dual vertex set $\widehat{V}$ is given by $\widehat{V}:=\{\widehat{r}_{(x,t)}, \widehat{l}_{(x,t)} : (x,t) \in {\cal N}\}$.

Next, let us define the dual ancestor $\widehat{h}(y,s)=\widehat{h}((y,s),\NN)$ of $(y,s)\in\widehat{V}$ as the unique vertex in $\widehat{V}$ given by
\begin{equation*}
\widehat{h}(y,s) :=
\begin{cases}
\widehat{l}_{(y,s)_r} & \text{ if }(y,s)_r(2) >  (y,s)_l(2)  \\
\widehat{r}_{(y,s)_l} & \text{ otherwise.}
\end{cases}
\end{equation*}
The dual edge set is then $\widehat{E}:=\{\langle (y,s), \widehat{h}(y,s) \rangle : (y,s)\in \widehat{V}\}$. Clearly, each dual vertex has exactly one outgoing edge which goes in the downward direction. Hence, the dual graph $\widehat{\mathfrak{F}}:= (\widehat{V},\widehat{E})$ does not contain any cycle. This forest is entirely determined from $\mathfrak{F}$ without extra randomness. We obtain a dual (or backward) path $\widehat{\pi}^{(y,s)}\in\widehat{\Pi}$ starting at $(y,s)$, by linearly joining the successive $\widehat{h}(\cdot)$ steps. Thus, $\widehat{{\cal X}}:=\{\widehat{\pi}^{(y,s)} : (y,s) \in \widehat{V}\}$ denotes the collection of all dual paths obtained from $\widehat{\mathfrak{F}}$.

\begin{figure}[!ht]
\begin{center}
\psfrag{a}{\small{$\x$}}
\psfrag{b}{\small{$\widehat{r}_{\x}$}}
\psfrag{c}{\small{$\widehat{l}_{\x}$}}
\psfrag{d}{\small{$\x_l$}}
\psfrag{e}{\small{$\x_r$}}
\includegraphics[width=7.5cm,height=6cm]{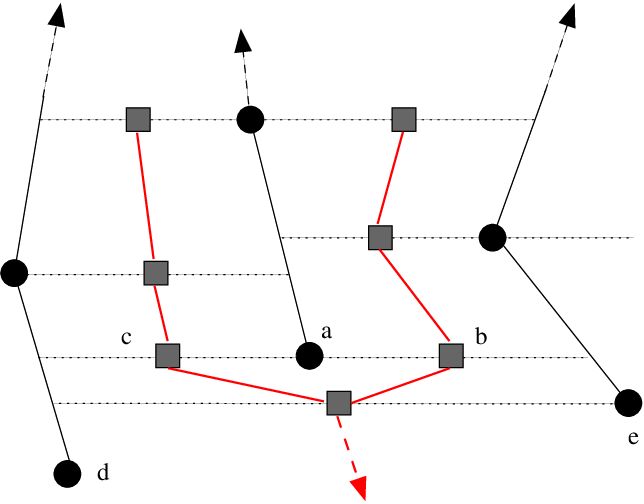}
\caption{{\small \textit{Here is a picture of the DSF $\mathfrak{F}$ (in upward direction) and its dual forest $\widehat{\mathfrak{F}}$ (in downward direction). Vertices of the DSF are black circles whereas dual vertices are grey squares. In particular, the vertex $\x$ produces two dual vertices $\widehat{l}_{\x}$ and $\widehat{r}_{\x}$. On this picture, $(\widehat{r}_{\x})_r=\x_r$ and $(\widehat{r}_{\x})_l=\x_l$ with $\x_r(2)>\x_l(2)$: this implies that $\widehat{h}(\widehat{r}_{\x})=\widehat{l}_{\x_r}$. The same is true for $\widehat{l}_{\x}$.}}}\label{fig:dual}
\end{center}
\end{figure}

Let us recall that ${\cal X}_n={\cal X}_n(\gamma,\sigma)$ for $\gamma,\sigma>0$ and $n\geq 1$, is the collection of $n$-th order diffusively scaled paths. See (\ref{defi:ScaledPath}). In the same way, we define $\widehat{{\cal X}}_n=\widehat{{\cal X}}_n(\gamma,\sigma)$ as the collection of diffusively scaled dual paths. For any dual path $\widehat{\pi}$ with starting time $
\sigma_{\widehat{\pi}}$, the scaled dual path $\widehat{\pi}_n(\gamma,\sigma) : [-\infty , \sigma_{\widehat{\pi}}/n^2\gamma] \to [-\infty, \infty]$ is given by
\begin{equation}
\label{defi:ScaledDualPath}
\widehat{\pi}_n(\gamma,\sigma) (t) := \widehat{\pi}(n^2\gamma t)/n \sigma ~.
\end{equation}
For each $n\geq 1$, the closure $\overline{\widehat{\cal X}}_n$ of $\widehat{{\cal X}}_n$ in $(\widehat{\Pi},d_{\widehat{\Pi}})$ is a $(\widehat{{\cal H}}, {\cal B}_{\widehat{{\cal H}}})$-valued random variable.\\

Now, let us introduce the dual Brownian web $\widehat{{\cal W}}$. For this, we need a topology on the family of backward paths similar to the one associated to $(\Pi,d_\Pi)$. As in the definition of $\Pi$, let $\widehat{\Pi}$ be the collection of all continuous paths $ \widehat{\pi}$ with starting time $\sigma_{\hat{\pi}} \in [-\infty,\infty]$ such that $\widehat{\pi} : [-\infty, \sigma_{\hat{\pi}}] \to [-\infty,\infty] \cup \{\ast\}$ with $\widehat{\pi}(-\infty)= \ast$ and, when $\sigma_{\hat{\pi}} = \infty$, $\widehat{\pi}(\infty)= \ast$. As earlier $t \mapsto (\widehat{\pi}(t),t)$ is continuous from $[-\infty, \sigma_{\hat{\pi}}]$ to $(\R^{2}_c,\rho)$. We thus equip $\widehat{\Pi}$ with the metric
\begin{multline*}
d_{\widehat{\Pi}} (\widehat{\pi}_1,\widehat{\pi}_2) =  |\tanh(\sigma_{\hat{\pi}_1})-\tanh(\sigma_{\hat{\pi}_2})| \\
\vee
\sup_{t\leq \sigma_{\hat{\pi}_1} \vee \sigma_{\hat{\pi}_2}} \Bigl| \frac{\tanh(\widehat{\pi}_1(t\wedge\sigma_{\hat{\pi}_1}))}{1+|t|} - \frac{\tanh(\widehat{\pi}_2(t\wedge\sigma_{\hat{\pi}_2}))}{1+|t|}\Bigr|
\end{multline*}
making $(\widehat{\Pi}, d_{\widehat{\Pi}})$ a complete, separable metric space. Let us recall that with a slight abuse of notation, the closure of any element $X$ in $(\Pi,d_{\Pi})$ or $(\widehat{\Pi},d_{\widehat{\Pi}})$ will be still denoted by $X$. \\
The metric space of compact sets of $\widehat{\Pi}$ is denoted by $(\widehat{{\mathcal H}}, d_{\widehat{{\mathcal H}}})$, where $d_{\widehat{{\mathcal H}}}$ is the Hausdorff metric on $\widehat{{\mathcal H}}$, and let $ {\mathcal B}_{\widehat{{\mathcal H}}}$ be the corresponding Borel $\sigma$-field. The BW and its dual denoted by $({\mathcal W},\widehat{{\mathcal W}})$ are a $({\mathcal H}\times \widehat{{\mathcal H}}, {\mathcal B}_{{\mathcal H}}\times {\mathcal B}_{\widehat{{\mathcal H}}})$-valued random variable such that:
\begin{itemize}
\item[$(i)$] $\widehat{{\mathcal W}}$ is distributed as $-{\mathcal W}$, the BW rotated $180^0$ about the origin;
\item[$(ii)$] ${\mathcal W}$ and $\widehat{{\mathcal W}}$ uniquely determine each other: $\widehat{\mathcal{W}}$ consists of a collection of coalescing paths running backward in time and that a.s. do not cross the paths of $\mathcal{W}$, in the sense that for any paths $\pi\in \mathcal{W}$ and $\widehat{\pi}\in \widehat{\mathcal{W}}$ such that $\sigma_\pi<\sigma_{\widehat{\pi}}$, we have for all $s,t$ such that $\sigma_\pi \leq s<t\leq \sigma_{\widehat{\pi}}$,
\begin{align}
\label{def:NonCrossingPath}
(\pi(s) - \widehat{\pi}(s))(\pi(t) - \widehat{\pi}(t))\geq 0 ~.
\end{align}
\end{itemize}
See Schertzer et al. \cite[Theorem 2.4]{SSS17}. The interaction between the paths in ${\mathcal W}$ and $\widehat{\mathcal W}$ is that of Skorohod reflection (see \cite{STW00}).\\

We can finally state our main result:

\begin{theorem}
\label{th:FF_BW_complete}
There exist $\sigma=\sigma(\lambda)>0$ and $\gamma=\gamma(\lambda)>0$ such that the sequence
$$
\big\{ \big( \overline{{\cal X}}_n(\gamma,\sigma) , \overline{\widehat{\cal X}}_n(\gamma,\sigma) \big) : \, n \geq 1 \big\}
$$
converges in distribution to $({\cal W}, \widehat{{\cal W}})$ as $({\cal H}\times \widehat{{\cal H}},{\cal B}_{{\cal H}\times \widehat{{\cal H}}})$-valued random variables as $n\rightarrow\infty$.
\end{theorem}

Because of the intricate dependencies of the DSF model, we are not able to apply the earlier techniques available in the literature, as Theorem \ref{thm:BwebConvergenceNoncrossing1} below, in order to obtain Theorem \ref{th:FF_BW}. This is the reason why we provide in Section \ref{sect:NewCVCriteria} new convergence criteria (Theorem \ref{thm:JtConvBwebGenPath}) regarding joint convergence to $(\WW, \widehat{\WW})$ for non-crossing path models. Let us mention here that ideas sustaining this result are already present in \cite{RSS15b} (Section 2.3). In Section \ref{sebsec:DSFBwebFinalVerification}, we use results obtained in Sections \ref{sect:Renewal} and \ref{sec:Tail} to show that the sequence $\{({\cal X}_n,\widehat{\cal X}_n) : n \geq 1\}$ satisfies the conditions of Theorem \ref{thm:JtConvBwebGenPath}.

\subsection{Convergence criteria for non-crossing path models}
\label{sect:NewCVCriteria}

Let us recall and comment the first convergence criteria to the BW, provided by Fontes et al. \cite{FINR04}, in order to motivate new convergence criteria given in Theorem \ref{thm:JtConvBwebGenPath}. This section focuses on non-crossing path models. The reader may refer to \cite{SSS17} for a very complete overview on the topic.

Let $\Xi\subset\Pi$. For $t>0$ and $t_0,a,b\in\R$ with $a<b$, consider the counting random variable $\eta_\Xi(t_0,t;a,b)$ defined as
\begin{equation}
\label{def:EtaPtSet}
\eta_{\Xi}(t_0,t;a,b) := \Card \big\{ \pi(t_0+t) :\, \pi \in \Xi , \, \sigma_{\pi}\leq t_0 \text{ and } \pi(t_0) \in [a,b] \big\}
\end{equation}
which considers all paths in $\Xi$, born before $t_0$, that intersect $[a,b]$ at time $t_0$ and counts the number of different positions these paths occupy at time $t_0+t$. In Theorem 2.2 of \cite{FINR04}, Fontes et al. provided the following convergence criteria.

\begin{theorem}[Theorem 2.2 of \cite{FINR04}]
\label{thm:BwebConvergenceNoncrossing1}
Let $\{\Xi_n : n \in \N\}$ be a sequence of $({\mathcal H},B_{{\mathcal H}})$ valued random variables with non-crossing paths. Assume that the following conditions hold:
\begin{itemize}
\item[$(I_1)$] Fix a deterministic countable dense set ${\mathcal D}$  of $\R^2$. For each $\x \in {\mathcal D}$, there exists $\pi_n^{\x} \in \Xi_n$ such that for any finite set of points $\x^1, \dotsc, \x^k \in {\mathcal D}$, as $n \to \infty$, we have
$(\pi^{\x^1}_n, \dotsc, \pi^{\x^k}_n)$ converges in distribution to $(W^{\x^1}, \dotsc, W^{\x^k} )$, where $(W^{\x^1}, \dotsc, W^{\x^k} )$ denotes coalescing Brownian motions starting from the points $\x_1, \ldots, \x_k$.
\item[$(B_1)$] For all $t>0$, $\limsup_{n\rightarrow \infty}\sup_{(a,t_0)\in\R^2}\P(\eta_{\Xi_n}(t_0,t;a,a+\epsilon)\geq 2)\rightarrow 0$  as $\epsilon\downarrow 0$.
\item[$(B_2)$] For all $t>0$, $\frac{1}{\epsilon}\limsup_{n\rightarrow\infty}\sup_{(a,t_0)\in \R^2}\P(\eta_{\Xi_n}(t_0,t;a,a+\epsilon)\geq 3)\rightarrow 0$ as $\epsilon\downarrow 0$.
\end{itemize}
Then $\Xi_n$ converges in distribution to the standard Brownian web ${\mathcal W}$ as $n \to \infty$.
\end{theorem}

Let us first mention that for a sequence of $({\cal H}, {\cal B}_{{\cal H}})$-valued random variables $\{\Xi_n: n \in \N\}$ with non-crossing paths, Criterion $(I_1)$ implies tightness (see Proposition B.2 in the Appendix of \cite{FINR04} or Proposition 6.4 in \cite{SSS17}) and hence subsequential limit(s) always exists. Moreover, Criterion $(B_1)$ has in fact been shown to be redundant with $(I_1)$ for non-crossing path models (see Theorem 6.5 of \cite{SSS17}). Combining $(I_1)$ with Theorem \ref{theorem:Bwebcharacterisation}, we obtain that any such subsequential limit $\Xi$ a.s. contains a random subset which is distributed as the standard BW ${\mathcal W}$.

There are several approaches to prove the other inclusion $\Xi\subset{\mathcal W}$. Criterion $(B_2)$ is often verified by applying an FKG type correlation inequality together with a bound on the distribution of the coalescence time between two paths. However, FKG is a strong property which may not hold for models with interactions. This strategy seems really hard to carry out in the DSF context. In the literature, new criteria have been suggested to replace $(B_2)$: let us mention for instance Criterion $(E)$ proposed by Newman et al \cite{NRS05}. See also Theorem 6.3 of Schertzer et al. \cite{SSS17}. In the same reference, Schertzer et al. have given in Theorem 6.6 a new criterion replacing $(B_2)$, called the \textit{wedge condition}. Our convergence result (Theorem \ref{thm:JtConvBwebGenPath} below) appears as a generalization of Theorem 6.6 of \cite{SSS17} by considering the joint convergence of $\{(\Xi_n, \widehat{\Xi}_n) : n\geq 1\}$ to the BW and its dual. Here, $\widehat{\Xi}_n$ merely denotes a $(\widehat{{\cal H}}, \widehat{{\cal B}}_{\widehat{{\cal H}}})$-valued random variable made up of paths running backward in time. Theorem \ref{thm:JtConvBwebGenPath} also replaces the wedge condition by the fact that no limiting primal and dual paths can spend positive Lebesgue time together: this is condition $(iv)$. We believe that Theorem \ref{thm:JtConvBwebGenPath} is robust and can be applied for studying convergence to the BW for a large variety of models with non-crossing paths.

\begin{theorem}
\label{thm:JtConvBwebGenPath}
Let $\{(\Xi_n, \widehat{\Xi}_n) : n\geq 1\}$ be a sequence of $({\cal H}\times\widehat{{\cal H}}, {\cal B}_{{\cal H}\times\widehat{{\cal H}}})$-valued random variables with non-crossing paths only, satisfying the following assumptions:
\begin{itemize}
\item[(i)] For each $n \geq 1$, paths in $\Xi_n$ do not cross (backward) paths in $\widehat{\Xi}_n$ almost surely: there does not exist any $\pi\in\Xi_n$, $\widehat{\pi}\in\widehat{\Xi}_n$ and $t_1,t_2\in (\sigma_{\pi},\sigma_{\widehat{\pi}})$ such that $(\widehat{\pi}(t_1) - \pi(t_1))(\widehat{\pi}(t_2) - \pi(t_2))<0$ almost surely.
\item[(ii)] $\{\Xi_n : n \in \N\}$ satisfies $(I_1)$.
\item[(iii)] { $\{(\widehat{\pi}_n(\sigma_{\widehat{\pi}_n}),\sigma_{\widehat{\pi}_n}) : \widehat{\pi}_n \in \widehat{\Xi}_n \}$, the collection  of starting points of all the backward paths in $\widehat{\Xi}_n$, as $n \to \infty$, becomes dense in $\R^2$.}
\item[(iv)] {For any sub sequential limit $({\cal Z},\widehat{{\cal Z}})$ of $\{(\Xi_n,\widehat{\Xi}_n) : n \in \N\}$, paths of ${\cal Z}$ do not spend positive Lebesgue measure time together with paths of $\widehat{{\cal Z}}$, i.e.,  almost surely there do not exist $\pi\in{\cal Z}$ and $\widehat{\pi}\in\widehat{{\cal Z}}$ such that $\int_{\sigma_\pi}^{\sigma_{\widehat{\pi}}} \mathbf{1}_{\pi(t)=\widehat{\pi}(t)} dt > 0$}.
\end{itemize}
Then $(\Xi_n,\widehat{\Xi}_n)$ converges in distribution to $({\cal W}, \widehat{{\cal W}})$ as $n \to \infty$.
\end{theorem}

This section ends with the proof of Theorem \ref{thm:JtConvBwebGenPath}.

\begin{proof}
As mentioned in Section 6.2 of \cite{SSS17}, conditions  $(i)$ and $(ii)$ imply that the sequence $\{(\Xi_n, \widehat{\Xi}_n) : n\geq 1\}$ is jointly tight and then subsequential limit(s) always exists. Let $({\cal Z}, \widehat{{\cal Z}})$ be one of them. Our goal is to identify the distribution of this limiting value with $({\cal W}, \widehat{{\cal W}})$.

As the sequence $\{\Xi_n : n\geq 1\}$ satisfies $(I_1)$, for any $(x,t)\in\Q^2$, there a.s. exists a path $\pi^{(x,t)}$ in ${\cal Z}$ starting from the point $(x,t)$ and distributed as a Brownian motion starting from $x$ at time $t$. Because of the non-crossing paths property of the limit ${\cal Z}$-- which inherits this property from $\Xi_n$ (condition $(i)$)--, similar arguments as in the proof of Proposition 3.1 of \cite{FINR04} ensure that $\pi^{(x,t)}$ is a.s. the only path in ${\cal Z}$ starting at $(x,t)$. This means that ${\cal Z}_{\Q^2}$ is distributed as a collection of coalescing Brownian motions:
\begin{equation}
\label{eq:ZSkeleton}
{\cal Z}_{\Q^2} \stackrel{d}{=} {\cal W}_{\Q^2} ~.
\end{equation}
(See also discussions in Section 6.2 of \cite{SSS17}).

In order to assert that the closure ${\cal Z}$ of ${\cal Z}_{\Q^2}$ in $(\Pi,d_{\Pi})$ is a standard Brownian web, we have to prove that ${\cal Z}$ contains no more paths than ${\cal W}$. This is the role of the wedge condition and Theorem 6.6 of \cite{SSS17}. Let us first introduce some notation. For any backward paths $\widehat{\pi}^l$ and $\widehat{\pi}^r$ in $\widehat{\Pi}$ that are ordered with $\widehat{\pi}^l(s)<\widehat{\pi}^r(s)$ at time $s:=\min\{\sigma_{\widehat{\pi}^r}, \sigma_{\widehat{\pi}^l}\}$, we define $T(\widehat{\pi}^l,\widehat{\pi}^r) := \sup\{t<s : \widehat{\pi}^l(t) = \widehat{\pi}^r(t)\}$ (possibly equal to $-\infty$) as the first hitting time of $\widehat{\pi}^l$ and $\widehat{\pi}^r$ (which is actually the coalescing time of these paths). The \textit{wedge with left boundary $\widehat{\pi}^l$ and right boundary $\widehat{\pi}^r$} is the following open set of $\R^2$:
\begin{equation}
\label{def:wedgenoncpt}
A(\widehat{\pi}^l, \widehat{\pi}^r) := \bigl\{ (y,u)\in\R^{2} : \, T(\widehat{\pi}^l,\widehat{\pi}^r)< u < s \, \mbox{ and } \,\widehat{\pi}^l(u) < y < \widehat{\pi}^r(u) \bigr\} ~.
\end{equation}
A path $\pi\in\Pi$, is said to \textit{enter the wedge $A(\widehat{\pi}^l, \widehat{\pi}^r)$ from outside} if there exist $t_1,t_2$ with $\sigma_{\pi}<t_1<t_2$ such that $(\pi(t_1), t_1)\notin\bar{A}$ and $ (\pi(t_2), t_2)\in A$, where $\bar{A}$ denotes the closure of $A$ in $\R^2$. The \textit{bottom point} of $A(\widehat{\pi}^l, \widehat{\pi}^r)$ is $(\widehat{\pi}^l(T(\widehat{\pi}^l,\widehat{\pi}^r)),T(\widehat{\pi}^l,\widehat{\pi}^r)) = (\widehat{\pi}^r(T(\widehat{\pi}^l,\widehat{\pi}^r)),T(\widehat{\pi}^l,\widehat{\pi}^r))$. {The \textit{wedge condition} states that}
\begin{center}
a.s. no path in ${\cal Z}$ enters any wedge of $\widehat{{\cal Z}}_{\Q^2}$ from outside.
\end{center}
This is Criterion $(U)$ of \cite{SSS17}. This condition combined with $(i)$, $(ii)$ and $(iii)$ implies (Theorem 6.6 of \cite{SSS17}) that $\Xi_n$ converges in distribution to ${\cal W}$ as $n$ tends to infinity, i.e., ${\cal Z}$ is distributed as ${\cal W}$. By condition $(i)$, primal and dual paths do not cross with probability $1$. Hence, the only way for a path $\pi$ in ${\cal Z}$ to enter a wedge of $\widehat{{\cal Z}}_{\Q^2}$ from outside is through its bottom point by spending a time of positive Lebesgue measure with the dual path started from the bottom point of the wedge. But this is forbidden by condition $(iv)$. So the wedge condition holds and ${\cal Z}$ is distributed as ${\cal W}$.

Next, we focus on the dual paths in $\widehat{{\cal Z}}$. From condition $(iii)$, it follows that for any $(x,t)\in\Q^2$, a.s. there exists a backward path $\widehat{\pi}^{(x,t)}$ in $\widehat{\cal Z}$ starting from $(x,t)$. Since paths in ${\cal Z}$ and $\widehat{{\cal Z}}$ do not cross, the position of $\widehat{\pi}^{(x,t)}$ at the rational time $s<t$ can be specified as follows:
\begin{align*}
\widehat{\pi}^{(x,t)}(s) = & \sup\{y \in \Q : \pi^{(y,s)} \in {\cal Z}_{\Q^2}, \pi^{(y,s)}(t) < x\} \\
= & \inf\{y \in \Q : \pi^{(y,s)} \in {\cal Z}_{\Q^2}, \pi^{(y,s)}(t) > x\},
\end{align*}
which means that the dual paths in $\widehat{{\cal Z}}_{\Q^2}$ are uniquely determined by the forward paths in ${\cal Z}_{\Q^2}$. Since the dual paths in $\widehat{\WW}$ do not cross the paths in $\WW$,
it follows that the dual paths in
$\widehat{\cal W}_{\Q^2}$ are also a.s. determined by the forward paths in $\WW_{\Q^2}$. We then deduce from (\ref{eq:ZSkeleton}) that
\begin{equation*}
\widehat{\cal Z}_{\Q^2} \stackrel{d}{=} \widehat{\cal W}_{\Q^2} ~.
\end{equation*}

As previously, we can conclude using conditions $(i)$ and $(iv)$ that a.s. paths of $\widehat{{\cal Z}}$ do not enter any wedge in ${\cal Z}_{\Q^2}$, which has the same distribution as ${\mathcal W}_{\Q^2}$, from outside. We then conclude thanks to the next result which is a slight variant of Theorem 1.9 of \cite{SS08} (see also Theorem 3.9 in \cite{SSS17} and the following remark), whose proof is omitted here:

\begin{lemma}
\label{cor:WeakerWedgeDoubleBweb}
{Let $(\WW, \widehat{{\cal Z}})$ be a $({\cal H}\times \widehat{\cal H}, {\cal B}_{{\cal H}\times \widehat{\cal H}})$-valued random variable with $\WW$ denoting the Brownian web such that
a.s. paths of $\widehat{{\cal Z}}$ do not enter any wedge in ${\mathcal W}_{\Q^2}$ from outside and the set of starting points of dual paths in $\widehat{{\cal Z}}$, given by $\{\sigma_{\widehat{\pi}} : \widehat{\pi} \in \widehat{{\cal Z}}\}$, is dense in $\R^2$.} Then, we have
$$
\widehat{{\cal Z}} \stackrel{d}{=} \widehat{\cal W} ~.
$$
\end{lemma}

This completes the proof of Theorem \ref{thm:JtConvBwebGenPath}: the distribution of the subsequential limit $({\cal Z}, \widehat{{\cal Z}})$ is identified as $({\cal W}, \widehat{{\cal W}})$.
\end{proof}

\subsection{Verification of conditions of Theorem \ref{thm:JtConvBwebGenPath}}
\label{sebsec:DSFBwebFinalVerification}

In this section, we show that the sequence of diffusively scaled path families $\{({\cal X}_n,\widehat{{\cal X}}_n) : n \geq 1\}$ obtained from the DSF and its dual forest satisfies the conditions in Theorem \ref{thm:JtConvBwebGenPath}.

Conditions $(i)$ and $(iii)$ of Theorem \ref{thm:JtConvBwebGenPath} hold by construction. Indeed, paths of ${\cal X}$ do not cross (backward) paths of $\widehat{{\cal X}}$ with probability $1$. The same holds for the scaled sets ${\cal X}_n$ and $\widehat{{\cal X}}_n$. Moreover, the collection $\{(\widehat{\pi}_n(\sigma_{\widehat{\pi}_n}),\sigma_{\widehat{\pi}_n}) : \widehat{\pi}_n \in \widehat{\Xi}_n\}$ of all starting points of the scaled backward paths in $\widehat{\Xi}_n$ becomes dense in $\R^2$ as $n \rightarrow \infty$.

The next two sections are respectively devoted to the proofs of conditions $(ii)$ and $(iv)$. This will conclude the proof of Theorem \ref{th:FF_BW}.

\begin{remark}
{In \cite{CT13}, it was proved that a.s. there is no bi-infinite path in the DSF.
It was also asked whether the non-existence of bi-infinite path in the DSF could be proved using some duality argument. The joint convergence of the scaled DSF and its dual to the double Brownian web $(\WW, \widehat{\WW})$ gives a positive answer to this question.
From the construction it is evident that the DSF has a bi-infinite path if and only if the dual graph is not connected. Now if there are dual paths which do not coalesce but converge to coalescing Brownian motions under diffusive scaling, then we have at least one scaled forward path entrapped between these two scaled dual paths. The joint convergence to the double Brownian web $(\WW, \widehat{\WW})$ forces that there must be a limiting forward Brownian path approximating this entrapped forward scaled path. Further this limiting Brownian path must spend positive Lebesgue measure time together with a backward Brownian path leading to a contradiction and proves that  there is no bi-infinite path in the DSF a.s.}

\end{remark}

\subsubsection{Verification of condition $(ii)$}
\label{sect:(I1)}

Let us prove that the diffusively scaled sequence $\{{\cal X}_n : n\geq 1\}$ satisfies condition $(ii)$, i.e. Criterion $(I_1)$ of Theorem \ref{thm:BwebConvergenceNoncrossing1}. The main ingredients on which $(I_1)$ is based have been stated in Section \ref{sect:Renewal}. On one hand, single path of the DSF can be simultaneously broken down into independent pieces through renewals steps (Proposition \ref{prop:SinglePtRWalk}).  On the other hand, for multiple paths, they behave independently as long as they explore disjoint regions and the size of renewal block between any two consecutive renewal steps admits sub-exponentially decaying tails (Proposition \ref{prop:IIDMaxSubExpTail}). Thenceforth, to get $(I_1)$, we follow the strategy of Ferrari et al \cite{FFW05}, which was also used in \cite{RSS15}. The proof here is very similar to that of \cite{RSS15} (see Section 5.1) but in a continuous setting. For this reason we only provide the main steps so that the reader may understand the method without referring to \cite{RSS15}.\\

Let us first focus on a single path, $\pi^{\0}$ starting at the origin $\u_1^{(0)}=\0$. Let $\{\u^{(\ell)}_1 : \ell \geq 0\}$ be the sequence of renewal vertices allowing to break down $\pi^{\0}$ into independent pieces. Let us scale $\pi^{\0}$ into $\pi^{\0}_n$ as in (\ref{defi:ScaledPath}) with
$$
\sigma := \bigl( \text{Var} \bigl( \u^{(2)}_1(1)- \u^{(1)}_1(1) \bigr) \bigr)^{1/2} \; \text{ and } \; \gamma := \E \bigl(\u^{(2)}_1(2) - \u^{(1)}_1(2)\bigr) ~.
$$
The parameters $\sigma$ and $\gamma$ depend on $\lambda$, $k$ and $\kappa$. From now on, the diffusively scaled sequence $\{{\cal X}_n : n\geq 1\}$ is considered w.r.t. these parameters, but for ease of writing, we drop $(\gamma,\sigma)$ from our notation. {Proposition \ref{prop:SinglePtRWalk} together with Corollary \ref{cor:IncrementMoments}  allow us an application of   Donsker's invariance principle to show that $\pi^{\0}_n$ converges in distribution in $(\Pi,d_{\Pi})$ to $B^{\0}$ a standard Brownian motion started at $\0$}.

Thus we obtain that, for any sequences $(\v_n)$ and $(\w_n)$ such that $\v_n(2)=\w_n(2)=0$, $\w_n(1)<0<\v_n(1)$ with $(\v_n(1)-\w_n(1))/n \to 0$, the couple $(\pi^{\w_n}_n,\pi^{\v_n}_n)$ converges in distribution (in the suitable product metric space) to $(B^{\0},B^{\0})$. This result means that whenever two paths are close to each other, precisely within a $o(n)$ distance, then they will quickly coalesce. Although we can deal without it (see e.g. \cite{FFW05}), this is directly implied by the estimated on the coalescing time that we have established at Theorem \ref{thm:CoalescingTimetail}: for any $t>0$, $\P\big(T(\v_n,\w_n)> n^2 \gamma t\big)=o_n(1)$.

For showing the joint convergence of multiple paths, we use the fact that paths behave (almost) independently when they are separated by a large distance (roughly, at least of order $n$). This is possible since the size of renewal blocks between two consecutive renewal steps admits sub-exponentially decaying tails. Hence, distributions of two paths far enough from each other can be realized using independent PPP's. Thus, when paths come close to each other, they coalesce very quickly as indicated just above.

This strategy dealing with dependent paths, originally introduced in \cite{FFW05}, has been modified later to treat the case of long range interactions in \cite{colettidiasfontes} and \cite{RSS15}. We again emphasize the fact that the dependency structure of the DSF model is much more complicated compared to models previously cited.\\

The main change w.r.t. the proof in Section 5.1 of \cite{RSS15} concerns Proposition 5.4 which estimates the horizontal deviations of a path in terms of the height of the rectangle on which the configuration is known. Here is the result corresponding to our setting.

\begin{prop}
\label{prop:HistPathStayClose}
Let $0 < \beta <\alpha $. Consider the rectangle $R := [-m^{\beta},m^{\beta}]\times[0,m^{\beta}]$ for some $m \geq 1$. Let $\pi^{\0}$ be the path of the DSF starting at $\0$. Then,
\begin{equation*}
\P \Big( \sup_{0\leq s\leq m^{\beta}} |\pi^{\0} (s)| \ \geq \ 3 m^{\alpha} \mid \NN\cap R \Big) \leq C_{0} \exp\big(- C_{1} m^{\frac{\alpha-\beta}{2}} \big) ~.
\end{equation*}
\end{prop}

\begin{proof}
We first consider the case where $\sup_{0\leq s\leq m^{\beta}} \pi^{\0} (s)  \geq  3 m^{\alpha}$. The proof for the other case, i.e., $\sup_{0\leq s\leq m^{\beta}} \pi^{\0} (s)  \leq  - 3 m^{\alpha}$
is similar and hence omitted.
Let $\NN'$ be another PPP independent of $\NN$. We consider two paths, say $\pi^{(2m^\alpha, 0)}$ and $\pi^{(2m^\alpha, 0)}_{\text{new}}$, both starting from $(2m^\alpha,0)$, and using respectively the PPP's $\NN$ and $(\NN'\cap R)\cup (\NN\cap R^{c})$. In other words, for the path $\pi^{(2m^\alpha, 0)}_{\text{new}}$, the PPP inside the rectangle $R$ has been re-sampled. Since both paths $\pi^{\0}$ and $\pi^{(2m^\alpha, 0)}$ are constructed with the same PPP $\NN$, the non-crossing path property applies and gives:
$$
\sup_{0\leq s\leq m^{\beta}} \pi^{\0}(s) \ \geq \ 3 m^{\alpha} \; \Rightarrow \; \sup_{0\leq s\leq m^{\beta}} \pi^{(2m^\alpha, 0)}(s) \ \geq \ m^{\alpha} ~.
$$
Now, let us consider the sequence $(W_j)_{j\geq 1}$ of sizes of renewal blocks associated with the single path $\pi^{(2m^\alpha, 0)}_{\text{new}}$. By construction, it does not depend on the configuration $\NN\cap R$. After each renewal step, the $y$-ordinate of the moving vertex increases by at least $\kappa\geq 6$ and hence the path $\pi^{(2m^\alpha, 0)}_{\text{new}}$ can admit at most $\lfloor m^\beta\rfloor$ renewal steps before crossing the horizontal line $\{\x : \x(2)=m^\beta\}$. So, on the event
$$
A := \Big\{ \sum_{j=1}^{\lfloor m^\beta\rfloor} W_j \leq m^{\alpha} \Big\} ~,
$$
$\pi^{(2m^\alpha,0)}_{\text{new}}$ cannot exit the rectangle $[m^\alpha,3m^\alpha]\times[0,m^\beta]$. Moreover, on $A$, the paths $\pi^{(2m^\alpha,0)}$ and $\pi^{(2m^\alpha,0)}_{\text{new}}$ must agree over time interval $[0,m^\beta]$. We can then write:
\begin{eqnarray*}
\lefteqn{\P \Big( \sup_{0\leq s\leq m^{\beta}} \pi^{\0}(s) \geq 3 m^{\alpha} \mid \NN\cap R \Big)} \\
& \leq & \P \Big( \sup_{0\leq s\leq m^{\beta}} \pi^{(2m^\alpha, 0)}(s) \ \geq \ m^{\alpha} \mid \NN\cap R \Big) \\
& \leq & \P \Big( \sup_{0\leq s\leq m^{\beta}} \pi^{(2m^\alpha, 0)}_{\text{new}}(s) \ \geq \ m^{\alpha} \, , \, A \mid \NN\cap R \bigr) + \P \big( A^{c} \mid \NN\cap R \big) \\
& = & \P \big( A^{c} \mid \NN\cap R \big) \; = \; \P ( A^{c}) ~.
\end{eqnarray*}
We conclude using Proposition \ref{prop:IIDMaxSubExpTail}:
$$
\P( A^{c}) \leq \lfloor m^\beta \rfloor \P({\cal W} \geq m^{\alpha-\beta}) \leq C_0 \exp \big( -C_1 m^{\frac{\alpha-\beta}{2}} \big) ~,
$$
for suitable positive constants $C_0,C_1$.

Similar argument using paths starting from the point $(-2m^\alpha,0)$ completes the proof.
\end{proof}


\subsubsection{Verification of condition $(iv)$}
\label{sect:iv}

To show condition $(iv)$, we mainly follow the proof of Theorem 2.9 in \cite{RSS15b}, which was in a discrete setting. As a key ingredient, the coalescence time estimate (Theorem \ref{thm:CoalescingTimetail}) will be used in the proof of Lemma \ref{lem:coalescenceDSF} below.\\

Let $({\cal Z}, \widehat{\cal Z})$ be any subsequential limit of $\{({\cal X}_n, \widehat{\cal X}_n): n \geq 1\}$. {By Skorokhod's representation theorem we may assume that the convergence happens almost surely}. Instead of working with a subsequence, for ease of notation we may assume that the sequence $\{({\cal X}_n, \widehat{\cal X}_n): n\geq 1\}$ converges to $({\cal Z}, \widehat{\cal Z})$ almost surely in the $({\cal H}\times \widehat{{\cal H}}, d_{{\cal H}\times \widehat{{\cal H}}})$ metric space.

We have to prove that, with probability $1$, paths in ${\cal Z}$ do not spend positive Lebesgue measure time together with the dual paths in $\widehat{\cal Z}$. This means that for any $\delta>0$ and any integer $m\geq 1$, the probability of the event
\begin{equation*}
A(\delta, m) := \left\lbrace
\begin{array}{c}
\exists\mbox{ paths }\pi\in {\cal Z}, \widehat{\pi}\in \widehat{\cal Z}\\
\mbox{and }t_0\in\R\mbox{ s.t. }\\
-m<\sigma_{\pi}<t_0<t_0 + \delta<\sigma_{\widehat{\pi}}<m \text{ and }-m<\pi(t) = \widehat{\pi}(t)<m\\
\mbox{for all }t\in [t_0, t_0+\delta]
\end{array}
\right\rbrace
\end{equation*}
has to be $0$.\\

To show that $\P(A(\delta,m))=0$, we introduce a generic event $B^{\epsilon}_n(\delta, m)$ defined as follows. Given an integer $m\geq 1$ and $\delta,\epsilon>0$,
\begin{multline*}
B^{\epsilon}_n(\delta, m) := \\
\left\lbrace
\begin{array}{c}
\text{$\exists$ paths $\pi_1^n,\pi_2^n,\pi_3^n \in {\cal X}_n$ s.t. $\sigma_{\pi_1^n},\sigma_{\pi_2^n}\leq 0$, $\sigma_{\pi_3^n} \leq \delta$ }\\
\text{and $\pi_1^n(0),\pi_1^n(\delta)\in [-m,m]$} \\
\text{with $|\pi_1^n(0)-\pi_2^n(0)|<\epsilon$ but $\pi_1^n(\delta) \not= \pi_2^n(\delta)$} \\
\text{and with $|\pi_1^n(\delta)-\pi_3^n(\delta)|<\epsilon$ but $\pi_1^n(2\delta) \not= \pi_3^n(2\delta)$} \\
\end{array}
\right\rbrace ~.
\end{multline*}
The event $B^{\epsilon}_n(\delta, m)$ means that there exists a path $\pi_1^n$ localized in $[-m,m]$ at time $0$ as well as at time $\delta$ which is approached (within distance $\epsilon$) by two path $\pi_2^n$ and $\pi_3^n$ respectively at times $0$ and $\delta$ while still being different from them respectively at time $\delta$ and $2\delta$. Thanks to the coalescence time estimate (Theorem \ref{thm:CoalescingTimetail}), the following lemma, proved at the end of the section, shows that $B^{\epsilon}_n(\delta, m)$ has a small probability:

\begin{lemma}
\label{lem:coalescenceDSF}
For any integer $m\geq 1$, real numbers $\epsilon,\delta>0$, there exists a constant $C_0(\delta,m)>0$ (only depending on $\delta$ and $m$) s. t. for all large $n$,
\begin{equation*}
\P( B^{\epsilon}_n (\delta, m) ) \leq C_0(\delta, m) \, \epsilon ~.
\end{equation*}
\end{lemma}


Let us now explain how Lemma \ref{lem:coalescenceDSF} allows us to conclude. For $j=1,\ldots,\lfloor\frac{6m}{\delta}\rfloor$, let us set $t^j:=-m+(j\delta)/3$ and
\begin{multline*}
B^{\epsilon}_n(\delta, m ; j)\\
:= \left\lbrace
\begin{array}{c}
\text{$\exists$ paths $\pi_1^n,\pi_2^n,\pi_3^n \in {\cal X}_n$ s.t. $\sigma_{\pi_1^n},\sigma_{\pi_2^n}\leq t^j$, $\sigma_{\pi_3^n} \leq t^{j+1}$ and } \\
\pi_1^n(t^j), \pi_1^n(t^{j+1})\in [-2m,2m]\text{ with $|\pi_1^n(t^j)-\pi_2^n(t^j)|<4\epsilon$} \\
\text{ but $\pi_1^n(t^{j+1}) \not= \pi_2^n(t^{j+1})$ and with $|\pi_1^n(t^{j+1})-\pi_3^n(t^{j+1})|<4\epsilon$} \\
\text{ but $\pi_1^n(t^{j+2}) \not= \pi_3^n(t^{j+2})$}\\
\end{array}
\right\rbrace ~.
\end{multline*}
The event $B^{\epsilon}_n(\delta, m; j)$ corresponds to the event $B^{4\epsilon}_n(\delta/3,2m)$ considered in Lemma \ref{lem:coalescenceDSF}, 
and shifted up by $t^j$. Hence, by the translation invariance property of the DSF and Lemma  \ref{lem:coalescenceDSF}:
$$
\P (B^{\epsilon}_n(\delta, m ; j)) = \P (B^{4\epsilon}_n(\delta/3, 2m)) \leq 4 C_{0}(\delta/3, 2m) \, \epsilon
$$
for all $n$ large enough. The expected result will follow from:
\begin{equation}
\label{eq:but_fin}
A(\delta, m) \subset \liminf_{ n\to\infty} \bigcup_{j = 1}^{\lfloor\frac{6m}{\delta}\rfloor} B^{\epsilon}_n(\delta, m ; j) ~,
\end{equation}
since we then have:
\begin{eqnarray*}
\P(A(\delta, m)) & \leq & \limsup_{\epsilon\to 0} \P\Bigl( \liminf_{n\to\infty} \cup_{j=1}^{\lfloor\frac{6m}{\delta}\rfloor} B^{\epsilon}_n(\delta, m ; j) \Bigr) \\
& \leq & \limsup_{\epsilon\to 0} \liminf_{n\to\infty} \sum_{j=1}^{\lfloor\frac{6m}{\delta}\rfloor} \P(B^{\epsilon}_n(\delta, m ; j)) \\
& \leq & \limsup_{\epsilon\to 0} \frac{6m}{\delta} 4 C_{0}(\delta/3, 2m) \, \epsilon \; = \; 0 ~.
\end{eqnarray*}

{It then remains to prove \eqref{eq:but_fin}.
Recall that $({\cal Z},\widehat{{\cal Z}})$ is a subsequential limit for $({\cal X}_n, \widehat{{\cal X}}_n)$. By Skorohod's representation theorem we may assume that the convergence happens almost surely. Let us work on the event $A(\delta, m)$, and consider $\pi\in {\cal Z}$, $\widehat{\pi}\in \widehat{{\cal Z}}$ and $t_0\in (\sigma_{\pi}, \sigma_{\widehat{\pi}})$ as in the definition the event $A(\delta, m)$.
It is useful to recall that the convergence of $({\cal X}_n, \widehat{{\cal X}}_n)$ to $({\cal Z}, \widehat{{\cal Z}})$ w.r.t. the Hausdorff metric implies that, for all $n$ large enough, we can find $\pi^{n}$
in ${\cal X}_n$ starting before time $t_0$ and $\widehat{\pi}^n \in \widehat{{\cal X}}_n$ starting after $t_0 + \delta$ that approximate $\pi$ and $\widehat{\pi}$ in the sense that}
\begin{align*}
& \max \Big\{ |\sigma_{\pi} - \sigma_{\pi^{n}}| ,|\sigma_{\widehat{\pi}} - \sigma_{\widehat{\pi}^{n}}|, |\pi(\sigma_{\pi}) - \pi^n(\sigma_{\pi^n})| , |\widehat{\pi}(\sigma_{\widehat{\pi}}) - \widehat{\pi}^n(\sigma_{\widehat{\pi}^n})|, \\
& \qquad \sup_{t\in [t_0,t_0+\delta]} |\pi(t) - \pi^{n}(t)|\vee  |\widehat{\pi}(t) - \widehat{\pi}^{n}(t)| \Big\} < \epsilon_1 .
\end{align*}

{Let us first assume that $\pi^{n}(t_0)<\widehat{\pi}^{n}(t_0)$. Since by construction paths in ${\cal X}_n$ can not cross  paths in $\widehat{{\cal X}}_n$, we must have $\pi^{n}(t)<\widehat{\pi}^{n}(t)$ on the whole time interval $[t_0,t_0+\delta]$. Let $j_0$ be the first index such that $j_0 := \min\{j \geq 1 : -m + (j\delta)/3 \geq t_0\}$. $\pi^n$ plays the role of $\pi^n_1$ as in the definition of $B^\epsilon_n(\delta, m ; j_0)$ and as $\pi^n_2$, we consider the (scaled) DSF path starting from the nearest scaled Poisson point $(x,t)$ to the point $(\widehat{\pi}_n(t^{j_0}),t^{j_0} )$ with $t< t^{j_0}$ and $x > \widehat{\pi}_n(t)$.  As
the forward paths in ${\cal X}_n$ can not cross the dual paths in $\widehat{{\cal X}}_n$, we must have $\pi^n_1(t^{j_0} + \delta/3) \neq \pi^n_2(t^{j_0}  + \delta/3)$. It is not difficult to observe that for all large $n$, the paths $\pi^n_1$ and $\pi^n_2$ satisfies the definition of $B^\epsilon_n(\delta, m ; j)$. With a similar proof, we can show the existence of a third path $\pi^n_3$ satisfying the requirements $B^\epsilon_n(\delta, m ; j)$. The other case, i.e., $\pi^{n}(t_0)> \widehat{\pi}^{n}(t_0)$ can be treated similarly. This completes the proof of \eqref{eq:but_fin}.} \hfill $\Box$

\bigskip

Let us end with the proof of Lemma \ref{lem:coalescenceDSF} which is close to the proof of Lemma 2.11 of \cite{RSS15b}. Both results are mainly based on the coalescence time tail estimates. With respect to Lemma 2.11 of \cite{RSS15b} two additional difficulties appear here: paths of the DSF are non-Markovian and constructed on a Poisson point process. Proposition \ref{prop:IIDMaxSubExpTail} will help us to control this long range dependence.

\begin{proof}[Proof of Lemma \ref{lem:coalescenceDSF}]
Fix $0 < 2 \beta < \alpha < 1$.
First, with high probability, we show that it is enough to consider the (unscaled) paths starting from Poisson vertices in a `thin' rectangular strip $S$ to study the event $B^\epsilon_n(\delta, m)$: $$S := [-2n\sigma m, 2n\sigma m]\times [-2n^\beta, 0].$$
This will help us to control the explored region until these paths cross the line $\{\x \in \R^2 : \x(2)=n^2\gamma \delta\}$.

Define the boxes of side length $n^\beta$ with lower sides on the lines
$ y = -2 n^\beta$, $ y = - n^\beta $ and $y = 0$. These boxes are given for $0 \leq j \leq \lfloor 4 n \sigma m / n^\beta \rfloor $ and $0\leq l \leq 2$ by
\begin{align*}
R_l(j) & := [- 2 n\sigma m + jn^{\beta}, - 2 n\sigma m + (j+1)n^{\beta}]\times [- l n^\beta, (-l + 1 )n^\beta ].
\end{align*}
Define the event $D_n$ as
\begin{align*}
D_n := \bigcap_{j = 0}^{\lfloor 4n\sigma m / n^\beta\rfloor} \cap_{l=0}^2  \{ R_l(j) \cap \NN \neq \emptyset\}.
\end{align*}
In other words, the event $D_n$ states that each of the above boxes must contain at least one Poisson point. It is not difficult to see that $\lim_{n\to \infty}  \P(D_n^c) = 0$. We observe that
on the event $D_n$, all the  paths crossing the segment
$[-n \sigma m, n \sigma m]\times\{0\}$ must start from Poisson points  inside the rectangular strip $S$ otherwise it contradicts the fact that interior of a history semi-ball must be free of Poisson points. Hence on the event $ B^\epsilon_n(\delta,m)\cap D_n$, the scaled paths $\pi^n_1$ and $\pi^n_2$ considered in $B^\epsilon_n(\delta, m)$ must start from
the (scaled) Poisson vertices in the rectangular strip $S$.

Next we show that evolution of the paths starting from Poisson
points in $S$ until they cross the line $y = \in n^2 \gamma \delta $ is
independent of the point process $\NN \cap \mathbb{H}^+(n^2\gamma \delta+2n^\beta)$.
Define the event
\begin{align*}
E_n := \{ & \text{ There exists a path }\pi \text{ starting from Poisson point in }S \\
& \text{ such that the history generated by it till it crosses }\\
&\text{ the line }y = n^2\gamma \delta \text{ does not intersect with }\mathbb{H}^+( n^2\gamma \delta + n^\beta) \}.
\end{align*}
Since between any two successive (marginal) renewals, the concerned path must progress
at least $\kappa + 1$, any path starting from a Poisson point in $S$ can have at most
$\lfloor (n^2\gamma \delta + 2n^\beta)/(\kappa + 1) \rfloor + 1 $ renewals until it
crosses the line $y = n^2\gamma \delta$. Thanks to Proposition \ref{prop:IIDMaxSubExpTail}, as the width of
the region explored between any two successive renewals decays sub-exponentially and
the region $S$ has area $4n \sigma m \times 2n^\beta$, by applying union bound we have  that the
probability $\P(E^c_n)$ decays sub-exponentially. Hence for all large $n$, we can focus on the event
$\P(B^\epsilon_n(\delta, m)\cap D_n \cap E_n)$.

We consider the event $\{B^\epsilon_n(\delta, m), D_n, E_n,
n \sigma \pi^n_1 (\delta) \in [k,k+1)\}$ for $k \in \Z$.
Define the event $G_n^\epsilon (\delta, m)$ as
\begin{equation*}
G^{\epsilon}_n(\delta, m) := \left\lbrace
\begin{array}{c}
\text{$\exists$ paths $\pi_1^n,\pi_2^n \in {\cal X}_n$ s.t. $\sigma_{\pi_1^n},\sigma_{\pi_2^n}\leq 0$ and $\pi_1^n(0),\pi_1^n(\delta)\in [-m,m]$} \\
\text{with $|\pi_1^n(0)-\pi_2^n(0)|<\epsilon$ but $\pi_1^n(\delta) \not= \pi_2^n(\delta)$} \\
\end{array}
\right\rbrace ~.
\end{equation*}
Because of  the non-crossing nature of paths, we must have
 \begin{align*}
& \{B^\epsilon_n(\delta, m), D_n, E_n, n \sigma \pi^n_1 (\delta) \in [k,k+1)\}
 \subset \\
 & \qquad \{\pi^{(k - n \sigma \epsilon, n^2 \gamma \delta)}( 2n^2 \gamma \delta) \neq \pi^{(k + n \sigma \epsilon, n^2 \gamma \delta)}( 2n^2 \gamma \delta), \\
 & \qquad \qquad G^\epsilon_n(\delta, m), D_n, E_n,
n \sigma \pi^n_1 (\delta) \in [k,k+1)\} ,
 \end{align*}
 because these paths are separated by $\pi_1^n$ and $\pi^n_3$.

 For $\lfloor - n \sigma m \rfloor - 1 \leq k \leq \lfloor n \sigma m \rfloor$, define the event  $F_{n}(k)$ as
\begin{align*}
F_{n}(k) := \big\{ & k - n \sigma \epsilon - n^{\alpha} \leq \pi^{(k - n \sigma \epsilon, n^{2}\gamma\delta)}(n^{2}\gamma\delta + n^{\beta}) \\
& \leq \pi^{(k + n \sigma \epsilon, n^{2}\gamma\delta)}(n^{2}\gamma\delta + n^{\beta}) \leq k + n \sigma \epsilon + n^{\alpha} \big\}.
\end{align*}
The event $F_{n}(k)$ asks that the paths starting at $(k - n \sigma \epsilon, n^{2}\gamma\delta)$
and $(k + n \sigma \epsilon, n^{2}\gamma\delta)$ do not fluctuate too much till time $n^{2}\gamma\delta+n^{\beta}$.
We showed earlier that, on the event $B^\epsilon_n(\delta)\cap D_n \cap E_n$, the DSF paths starting from  Poisson vertices in the rectangular strip $S$
do not explore the point process $\NN \cap \mathbb{H}^{+}(n^2\gamma \delta + n^\beta)$ until they  cross the line $y = n^2 \gamma \delta $. Recall that $0 < 2 \beta < \alpha < 1$
and observe that on the event $F_n(k)^{c}$, at least one of the two paths starting from $(k - n\sigma \epsilon, n^{2}\gamma\delta)$ and $(k + n\sigma \epsilon, n^{2}\gamma\delta)$ admits fluctuations larger than $n^{\alpha}$ on the time interval $[n^{2}\gamma\delta,n^{2}\gamma\delta + n^{\beta}]$. By Proposition \ref{prop:HistPathStayClose}, this has a probability smaller than $C_0e^{-C_1 n^{(\alpha-\beta)/2}}$. This gives that for any $\lfloor - n \sigma m \rfloor - 1 \leq k \leq \lfloor n \sigma m \rfloor $, the probability of the event $$(F_{n}(k))^c \cap \{n \sigma \pi^n_1(\delta) \in [k,k+1)\}\cap B^\epsilon_n(\delta,m)\cap  D_n \cap E_n $$ decays to $0$ sub-exponentially and uniformly in $k$. Hence
we can focus on the event  $F_{n}(k) \cap \{n\sigma\pi^n_1(\delta) \in [k,k+1)\}\cap B^\epsilon_n(\delta,m)\cap  D_n \cap E_n $, and the non-crossing path property forces
 the paths starting at $(k - n \sigma \epsilon -n^{\alpha}, n^{2}\gamma\delta +  n^{\beta})$
  and $(k + n \sigma \epsilon + n^{\alpha}, n^{2}\gamma\delta + n^{\beta})$
   to be still different at time $2n^{2}\gamma\delta$. So we obtain,
\begin{align}
\label{Eq:FnInd}
& \P(F_{n}(k) \cap \{n\sigma \pi^n_1(\delta) \in [k,k+1)\}\cap B^\epsilon_n(\delta,m)\cap D_n \cap E_n)\nonumber\\
 & \leq \P \bigl (\{\pi^{(k - n \sigma \epsilon -n^{\alpha}, n^{2}\gamma\delta + n^{\beta})}(2 n^2 \gamma \delta) \neq \pi^{(k + n \sigma \epsilon + n^{\alpha}, n^{2}\gamma\delta + n^{\beta})}(2 n^2 \gamma \delta)\}\cap \nonumber\\
& \qquad \{ n\sigma\pi^n_1(\delta)  \in [k,k+1)\}\cap  G^\epsilon_n(\delta,m)\cap D_n \cap E_n \bigr ).
\end{align}
Observe that the event $\{\pi^{(k - n \sigma \epsilon -n^{\alpha}, n^{2}\gamma\delta + n^{\beta})}(2 n^2 \gamma \delta) \neq \pi^{(k + n \sigma \epsilon + n^{\alpha}, n^{2}\gamma\delta + n^{\beta})}(2 n^2 \gamma \delta)\}$ depends only on the point process $\NN \cap \mathbb{H}^+(n^2 \gamma \delta +  n^\beta)$
 and the event $\{n\sigma\pi^n_1(\delta)  \in [k,k+1)\}\cap  G^\epsilon_n(\delta,m)\cap
  D_n \cap E_n$ depends only on the point process $\NN \cap \mathbb{H}^-(n^2\gamma \delta + n^\beta)$. Hence we have independence of the two events  in (\ref{Eq:FnInd}) and Theorem \ref{thm:CoalescingTimetail} gives that for all large $n$,
\begin{align*}
\P \big( & \pi^{(k - n \sigma \epsilon -n^{\alpha}, n^{2}\gamma\delta + n^{\beta})}(2n^{2}\gamma\delta)\not= \pi^{(k + n \sigma \epsilon + n^{\alpha}, n^{2}\gamma\delta +  n^{\beta})}(2n^{2}\gamma\delta) \big) \\
& \qquad \leq \frac{C_{0} (2n\sigma\epsilon + 2n^{\alpha})}{\sqrt{ n^{2}\gamma\delta - n^{\beta}}} \leq C_{0} \epsilon
\end{align*}
where $C_{0}=C_{0}(\delta)>0$ is suitably chosen.

 As the events $\{ n\sigma\pi^n_1(\delta) \in [k,k+1)\}$ are disjoint for different $k\in\Z$, it follows:
$$
\sum_{k = \lfloor - n \sigma m \rfloor - 1}^{\lfloor  n \sigma m \rfloor }
\P\bigl( \{n\sigma\pi^n_1(\delta)  \in [k,k+1)\}\cap  G^\epsilon_n(\delta,m)\cap  D_n \cap E_n \bigr)
 \leq \P( G^\epsilon_n(\delta,m)).
$$
Fix any $0 < \theta < 1$. The above discussion shows that for all large $n$ we have
$$
\P(B^\epsilon_n(\delta, m)) \leq \theta + C_{0} \epsilon \P( G^\epsilon_n(\delta,m)).
$$
In order to estimate the probability of $G_n^\epsilon(\delta, m)$ we define another event
\begin{align*}
H^\epsilon(\delta,m,l) := \{\pi^{(l, 0)} (n^2 \gamma \delta) \neq \pi^{(l+1, 0)} (n^2 \gamma \delta)\}\text{ for }\lfloor -n\sigma m \rfloor - 1 \leq l \leq \lfloor n\sigma m \rfloor.
\end{align*}
By non-crossing property we have that $G_n^\epsilon (\delta, m) \subset \bigcup_{l= \lfloor -n\sigma m \rfloor - 1}^{\lfloor n\sigma m \rfloor } H^\epsilon(\delta,m,l)$. To observe this inclusion relation, consider the event $B^{\epsilon}_n \cap \{\pi^n_1(0) < \pi^n_2(0) \}$ and observe that the paths starting from the points  $(\lfloor \pi^n_1(0) \rfloor,0) $ and $(\lfloor \pi^n_2(0) \rfloor + 1,0) $
must be different at time $n^2 \gamma \delta$. Similar reasoning follows for the event $B^{\epsilon}_n \cap \{\pi^n_1(0) > \pi^n_2(0) \}$.
Translation invariance of Poisson point process and use of Theorem \ref{thm:CoalescingTimetail} give that for all $\lfloor -n\sigma m \rfloor - 1 \leq l \leq \lfloor n\sigma m \rfloor$
$$
\P (H^\epsilon(\delta,m,l)) \leq \frac{C_0}{\sqrt{n^2 \gamma \delta}}.
$$

Hence, for all large $n$ we have
\begin{align*}
\P (B^\epsilon (\delta, m))\leq  & \theta + C_0 \epsilon \P (G^\epsilon_n (\delta, m)) \leq \theta + C_0 \epsilon (2 n \sigma m + 1) \P (H^\epsilon (\delta, m,l))\\
\leq  & \theta + C_0\epsilon ,
\end{align*}
where $C_0(m,\delta ) > 0$ is adjusted accordingly.
Since $\theta > 0$ is chosen arbitrarily, This completes the proof.
\end{proof}


\section{Theorem \ref{theo:sublinRST}: a sketch of the proof}

Recall that the Radial Spanning Tree (RST), initially introduced in \cite{BB07}, is a tree rooted at the origin $O$ with vertex set $\NN\cup\{O\}$ in which each vertex $\x\in\NN$ is connected to the closest Poisson point to $\x$ but inside the open ball $\{\y\in\R^2 : \, \|\y\|_2<\|\x\|_2\}$. Theorem 2.1 of \cite{BB07} states that the RST a.s. admits semi-infinite paths in each direction $\theta\in[0,2\pi)$. In particular, the (random) number $\chi_r$ of semi-infinite paths of the RST crossing the circle $\mathcal{C}_r$ with radius $r$, tends to infinity with probability $1$. Theorem \ref{theo:sublinRST} claims that
\begin{equation}
\label{sublin:InfinitePaths}
\E \chi_{r} = o(r^{3/4+\epsilon}) ~,
\end{equation}
for any $\epsilon>0$. Actually, our strategy to prove (\ref{sublin:InfinitePaths}) has been already developed in Section 6 of \cite{C2017} for a similar geometric random tree called the \textit{Radial Poisson Tree}. So we only focus here on the (minor) changes w.r.t. \cite{C2017}.\\

By isotropy, it is sufficient to prove that, for $0<\alpha<1/4$, $\E\chi_{r}(2r^{\alpha})$ tends to $0$ as $r\to\infty$ where $\chi_{r}(2r^{\alpha})$ counts the intersection points between the semi-infinite paths of the RST and the arc of the circle $\mathcal{C}_r$, centred at $(0,-r)$ and with length $2r^{\alpha}$. Approximating the RST around $(0,-r)$ by the DSF with direction $-e_y$ (especially using Lemma 3.4 of \cite{BB07} instead of Lemma 6.4 of \cite{C2017}) we show that
\begin{equation}
\label{sublin:ApproxRSTDSF}
\limsup_{r\to\infty} \E \chi_{r}(2r^{\alpha}) \leq \limsup_{r\to\infty} \E \eta_{r}(\alpha,\beta,\eps)
\end{equation}
where $\eps,\beta>0$ are such that $\alpha<\beta/2$ and $\beta+\eps<1/2$, and where $\eta_{r}(\alpha,\beta,\eps)$ counts the  intersection points between the horizontal segment $[-r^{\alpha},r^{\alpha}]\times\{-r\}$ and paths of the DSF starting from the outside of the rectangle $[-r^{\beta/2+\eps},r^{\beta/2+\eps}]\times [-r,-r-r^{\beta}]$.

Controlling with high probability the deviations of DSF paths (with Theorem 4.10 of \cite{BB07} instead of Lemma 6.6 of \cite{C2017}), (\ref{sublin:ApproxRSTDSF}) also holds if paths counted by $\eta_{r}(\alpha,\beta,\eps)$ are assumed to cross the lower side of the corresponding rectangle, i.e. the horizontal segment $[-r^{\beta/2+\eps},r^{\beta/2+\eps}]\times\{-r-r^{\beta}\}$. Thus, standard arguments based on the invariant translation property of the DSF (see the proof of Lemma 6.7 of \cite{C2017}) leads to
\begin{equation}
\label{sublin:InvTranslation}
\limsup_{r\to\infty} \E \chi_{r}(2r^{\alpha}) \leq \limsup_{r\to\infty} \E \tilde{\eta}_{r}(\alpha,\beta)
\end{equation}
where $\tilde{\eta}_{r}(\alpha,\beta)$ is defined as the number of intersection points between the horizontal axis $\R\times\{-r\}$ and DSF paths crossing the segment $[-r^{\alpha},r^{\alpha}]\times\{-r-r^{\beta}\}$. {Thenceforth,
\begin{equation}
\label{Isotropy1to0}
\limsup_{r\to\infty} \E \tilde{\eta}_{r}(\alpha,\beta) \leq 1
\end{equation}
allows to conclude. Indeed (\ref{Isotropy1to0}) implies that $c(\alpha):= \limsup \E \chi_{r}(2r^{\alpha})$ is smaller than $1$ for any $0<\alpha<1/4$. Let $M>0$ and $\alpha<\alpha'<1/4$. By isotropy of the RST and for $r$ large enough, we get $\E \chi_{r}(2r^{\alpha'})\geq M \E \chi_{r}(2r^{\alpha})$. Taking supremum limits, the inequality $1\geq M c(\alpha)$ follows. When $M\to\infty$, this forces $c(\alpha)=0$.}

{It remains to prove (\ref{Isotropy1to0}).} For $i=\lfloor -r^{\alpha}\rfloor,\ldots,\lfloor r^{\alpha}\rfloor$, let us denote by $\gamma_i$ the DSF path starting at the deterministic point $(i,-r-r^{\beta})$ and by $Y_i$ the number of edges crossing the horizontal unit segment $[i,i+1)\times\{-r-r^{\beta}\}$. Then, a.s.
\begin{equation}
\label{sublin:ComparaisonAS}
\tilde{\eta}_{r}(\alpha,\beta) \leq 1 \, + \,  \sum_{i=\lfloor -r^{\alpha} \rfloor}^{\lfloor r^{\alpha}\rfloor} (Y_i + 1) \, \ind_{\{\gamma_{i} \not=\gamma_{i+1} \mbox{ at time } -r\}} ~,
\end{equation}
where the event $\{\gamma_{i} \not=\gamma_{i+1} \mbox{ at time } -r\}$ means that paths $\gamma_{i}$ and $\gamma_{i+1}$ are still disjoint when they cross the horizontal axis $\R\times\{-r\}$. Since $\alpha<\beta/2$, one can find parameters $p,q>1$ such that $\alpha<\beta/(2p)$ and $1/p+1/q=1$. Then, the H\"older's inequality combined with our coalescence time estimate (Theorem \ref{thm:CoalescingTimetail}) gives:
\begin{eqnarray*}
\E \tilde{\eta}_{r}(\alpha,\beta) & \leq & 1 + 3 r^{\alpha} \E (Y_0+1) \, \ind_{\{\gamma_{0} \not=\gamma_{1} \mbox{ at time } -r\}} \\
& \leq & 1 + 3 r^{\alpha} \big(\E (Y_{0}+1)^{q}\big)^{1/q} \, \P\big( \gamma_{0} \not=\gamma_{1} \mbox{ at time } -r \big)^{1/p} \\
& \leq & 1 + 3 \big(\E (Y_{0}+1)^{q}\big)^{1/q} \frac{C_{0}^{1/p} r^{\alpha}}{r^{\beta/(2p)}}
\end{eqnarray*}
which tends to $1$ as $r\to\infty$. Above, we have used the fact that {the number $Y_0$ of DSF edges crossing an horizontal segment with unit length, admits moments of all orders. Indeed, the event $\{Y_0>\ell\}$ with large $\ell$, forces the existence of an edge counted by $Y_0$ with length larger than $\ell^{\delta}$, for some $\delta>0$. This implies the existence of an empty semi-ball with radius $\ell^{\delta}$ and the claim easily follows.}

{In the proof of $\lim \E \chi_{r}(2r^{\alpha})=0$, we have worked through a rectangle whose horizontal and vertical sizes $r^{\alpha}$ and $r^{\beta}$ have been chosen as follows. On the one hand, DSF paths inside this rectancle have to coalesce with high probability, which is ensured whenever $\alpha\leq \beta/2$ (Theorem \ref{thm:CoalescingTimetail}). On the other hand, the approximation of RST paths by DSF paths has to be valid in the whole rectangle, which requires $\beta<1/2$. The combination of these two conditions explains the exponent $3/4$ in Theorem \ref{theo:sublinRST}.}\\

Finally, the proof of the almost sure convergence of $\chi_{r}/r^{3/4+\epsilon}$ to $0$ follows from the convergence in expectation using the same arguments as in Section 7 of \cite{C2017}.


\end{document}